\documentclass[aos,preprint]{imsart}

\RequirePackage[OT1]{fontenc}
\RequirePackage{amsthm,amsmath}
\RequirePackage[authoryear]{natbib}
\RequirePackage[citecolor=blue,urlcolor=blue,breaklinks]{hyperref}

\usepackage{amsfonts, amssymb, graphicx, mathtools, epstopdf}
\usepackage{accents,bm}

\makeatletter
\newcommand{\inlineoverset}[2]{%
  {\mathop{#2}^{\vbox to 2.5\ex@{\kern-\tw@\ex@
   \hbox{\scriptsize \ensuremath{#1}}\vss}}}}
\makeatother
\numberwithin{equation}{section}
\theoremstyle{plain}

\newtheorem{theorem}{Theorem}[section]

\newtheorem{lemma}{Lemma}[section]

\newcommand{\I}{\ensuremath{\operatorname{\mathbb{I}}}}

\newcommand{\cov}{\ensuremath{\mathrm{cov}}}
\newcommand{\E}{\ensuremath{\operatorname{E}}}
\newcommand{\D}{\ensuremath{\operatorname{D}}}

\newcommand{\var}{\ensuremath{\operatorname{var}}}

\makeatletter

\def\mybigx#1{\dimen@#1\relax
\mathchoice
{\vbox to \dimen@{}}%
{\vbox to \dimen@{}}%
{\vbox to .7\dimen@{}}%
{\vbox to .5\dimen@{}}}%
\def\mybig#1{{\hbox{$\left#1\mybigx{0.8em}\right.\n@space$}}}
\makeatother

\begin{document}

\begin{frontmatter}

\title{Nonparametric graphon estimation}
\runtitle{Nonparametric graphon estimation}

\begin{aug}
\author{\fnms{Patrick J.} \snm{Wolfe}\ead[label=e1]{p.wolfe@ucl.ac.uk, s.olhede@ucl.ac.uk}}
\and
\author{\fnms{Sofia C.} \snm{Olhede}}

\runauthor{Wolfe \& Olhede}

\affiliation{University College London}

\address{Department of Statistical Science\\
University College London\\
Gower Street\\
London WC1E 6BT, UK\\
\printead{e1}}

\end{aug}

\begin{abstract}
We propose a nonparametric framework for the analysis of networks, based on a natural limit object termed a graphon. We prove consistency of graphon estimation under general conditions, giving rates which include the important practical setting of sparse networks. Our results cover dense and sparse stochastic blockmodels with a growing number of classes, under model misspecification. We use profile likelihood methods, and connect our results to approximation theory, nonparametric function estimation, and the theory of graph limits. 
\end{abstract}



\end{frontmatter}

\section{Introduction}

Networks are fast becoming part of the modern statistical landscape \citep*{Durrett, diaconis2007graph, bickel2009nonparametric, choi2012stochastic, fienberg2012brief, zhao2012consistency, arias2013cluster, ball2013network, choi2012co}. Yet we lack a full understanding of their large-sample properties in all but the simplest settings, hindering the development of models and inference tools that admit theoretical performance guarantees. 

In this article we introduce a nonparametric framework for the analysis of networks, which relates to kernel-based random graph models \citep{janson2010asymptotic, sussman2013universally}, stochastic blockmodels \citep{airoldi2008mixed, rohe2011spectral}, and degree-based models \citep{chatterjee2011random, BickelLevina}.  We use this framework to establish consistency of likelihood-based network inference under general conditions, and to show convergence rates across a range of network regimes, from dense to sparse.  Our framework thus addresses one of the biggest factors limiting the use of statistical network models in practice: a lack of flexible and transparent analysis tools that admit coherent statistical interpretations \citep{fienberg2012brief}.

Our methodology derives from a large-sample theory tailored to network data, in which well-defined limiting objects play a role akin to the infinite-dimensional functions that underpin classical nonparametric statistics \citep{bickel2009nonparametric}. An exchangeable stochastic network can be viewed as a partial observation of this limiting object under Bernoulli sampling \citep{diaconis2007graph}. Hence our theory is closely related to that of generalized linear models \citep{green1994nonparametric} and of contingency tables \citep{fienberg2012maximum}, as well as to nonparametric function approximation. High-dimensional statistical theory in this setting is nascent, and so the linkages we develop below provide for a foundational understanding of nonparametric statistical network analysis.

\section{Model elicitation}\label{model-elicit}

A network can be represented by an $n \times n$ data matrix $A$, whose $ij$th entry describes the relation between node $i$ and node $j$ of the network. In the most fundamental setting of graph theory, $A$ is a symmetric, binary-valued contingency table: it is sparse yet structured, with $A_{ij} \in \{0,1\}$ denoting the absence or presence of an edge between nodes $i$ and $j$, and with fixed, structural zeros along the main diagonal. 

We call $A$ an adjacency matrix, and model it as a realization of $\binom{n}{2}$ independent Bernoulli trials. Independently for $1 \leq i < j \leq n$, we have 
\begin{equation}
\label{eq:BernoulliModel}
A_{ij} \,\vert\, p_{ij} \sim \operatorname{Bernoulli}(p_{ij}), \qquad A_{ji} = A_{ij}, \quad A_{ii} = 0 .
\end{equation}
Each Bernoulli trial $ A_{ij} $ has success probability $p_{ij}$, which in turn we model using a bivariate function termed a \emph{graphon} that derives from the theory of graph limits \citep{lovasz2012large}. 

A graphon is a nonnegative symmetric function, measurable and bounded, that represents a discrete network as an infinite-dimensional analytic object. It is a basic characterization, allowing us to go from the discrete set of probabilities $ \{ p_{ij} \}_{i<j}$ to a limit object $f\left(x,y\right)$ defined on $(0,1)^2$, independently of the network size. Various summaries of the network can be calculated as functionals of the graphon; for example, a network's degree distribution is characterized by its graphon marginal $ \int_0^1 f\left(\cdot,y\right) \,dy$.

To model both dense and sparse networks, we allow the success probabilities $p_{ij}$ appearing in~\eqref{eq:BernoulliModel} to depend on $n$. We link these to a scaled graphon $\rho_n f\left(x,y\right)$ through a random sample $ \{\xi_i \}_{i=1}^n $ of uniform variates, via a scale parameter $ \rho_n >0 $ that specifies the expected probability of a network edge:
\begin{equation}
\label{pij-model}
p_{ij} = \rho_n f\left(\xi_i,\xi_j\right) ; \,\, \left\{\xi_1, \ldots, \xi_n\right\} \overset{\mathit{iid}}{\sim} \operatorname{Uniform}(0,1),
\,\, \textstyle \iint f\left(x,y\right) \, dx \, dy = 1 . \!\!\!\!
\end{equation}
Observe that $ \E A_{ij} = \E_\xi p_{ij} = \rho_n $ for all $ 1 \leq i < j \leq n $, and so $ \rho_n $ specifies the sparsity of the generated network. We assume the sequence $\{ \rho_n \}_{n = 2, 3, \ldots}$ to be fixed and monotone non-increasing.

This is a canonical model based on exchangeable random networks \citep{bickel2009nonparametric, BickelLevina}, and is also strongly related to other statistical modeling paradigms. It relates the infinite-dimensional graphon $f\left(x,y\right)$ to the set of probabilities $\{ p_{ij} \}_{i<j} $ sampled via $\xi$. This modeling strategy is similar to time series analysis, where a sampled autocovariance is related to an infinite-dimensional spectral representation. As with an independent increments process, we may think of each $\xi_i$ in~\eqref{pij-model} as a latent variable. Furthermore, $\xi_i$ is associated with the $i$th network node, acting as a latent random index into the graphon. This reflects the fact that the observed ordering of the network nodes conveys no information. 

Similarly, the ordering of a given graphon $f\left(x,y\right)$ along the $x$ and $y$ axes has no inherent meaning; that is, $f\left(x,y\right)$ has a built-in invariance to ``rearrangements'' of the $x$ and $y$ axes. This is similar to statistical shape analysis, where we seek to describe objects in a manner that is invariant to their orientation in Euclidean space. Thus $f\left(x,y\right)$ represents an equivalence class of all symmetric functions that can be obtained from one another through measure-preserving transformations of $[0,1]$. 

This notion was formalized by \citet{aldous1981representations} and \citet{hoover1979relations} in the context of exchangeable infinite arrays. Their eponymous theorem asserts that any such array admits a representation in terms of some $f(x,y,\alpha)$. This representation is unique up to measure-preserving transformation \citep{diaconis2007graph}, and the value of $\alpha$ is not identifiable from a single network observation \citep{bickel2009nonparametric}. The Aldous--Hoover representation thus relates~\eqref{pij-model} to an exchangeable infinite array $\{A_{ij}\}_{i,j=1}^\infty$ of binary random variables, such that for all $n = 1, 2, \ldots$, all permutations $\Pi$ of $\{ 1,\ldots,n \}$ and all $a \in \{0,1\}^{n \times n}$, we have that $\Pr(A_{ij} = a_{ij}, 1\leq i<j \leq n) = \Pr(A_{ij} = a_{\Pi(i)\Pi(j)}, 1\leq i,j \leq n)$. 

By putting an observed $n \times n$ adjacency matrix $A$ in correspondence with a finite set of rows and columns of $\{A_{ij}\}_{i,j=1}^\infty$, we arrive at a model for exchangeable networks, or for sub-networks thereof. Exchangeability implies that once we condition on the latent variable $\xi_i$ associated to network node $i$, then all linkages $ A_{i\cdot} $ to node $i$ are conditionally independent and identically distributed. This follows from de~Finetti's representation of a sum of exchangeable indicator variables \citep{diaconis1977finite}.

\section{Main result}\label{sec:mainresult}

Our main result is that whenever a graphon $f$ is H\"older continuous, and maximum likelihood fitting is used to derive a nonparametric estimator of $f$ from $A$, then this estimator will be consistent as long as $\rho_n = \omega \bigl( n^{-1} \log^3 n \bigr)$, and its rate of convergence can be established. 

To construct our estimator, we will calculate group averages after forming $k$ groups from $n$ nodes. Any such grouping can be represented as an integer partition of $n$ via a vector $ h \in \{ 2, \ldots, n \}^k $, such that $ \smash{ \sum_{a=1}^k } h_a = n $. Thus may view $ n^{-1} h $ as the probability mass function of a random variable with range $\left\{ 1, \ldots, k \right\}$, indexed via a cumulative distribution function $ H $ and its generalized inverse $ H^{-1} $:
\begin{subequations}
\begin{align}
\label{H}
H(u) & = \frac{1}{n} \sum_{a=1}^{ \lfloor u \rfloor } h_a ; \quad u\in \left[ 0 , k \right] , \,\, H(u) \in \left\{ 0, \tfrac{h_1}{n} , \tfrac{h_1+h_2}{n}, \ldots, 1 \right\} ,
\\ \label{Hinv} H^{-1}(x) & = \inf_{ u \in \left[ 0 , k \right] }\left\{ H(u)\ge x\right\}; \quad x\in \left( 0 , 1 \right], \,\, H^{-1}(x) \in \left\{ 1, \ldots, k \right\} .
\end{align}
\end{subequations}

The central difficulty in constructing a nonparametric graphon estimator is that we do not know the ordering of our observed adjacency matrix $A$, relative to the ordered sample $\{ \xi_{(i)} \}_{i=1}^n$ indexing the graphon $f$. We thus define an estimator $\smash{ \hat f }$ as a composition of two operations: first we re-index the rows and columns of $A$ according to some permutation $\Pi$ of $\{ 1, \ldots, n \}$, and then we group them in accordance with $H$:
\begin{align*}
\hat f\left( x , y ; h \right) = \hat \rho_n^{-1} \bar A_{ H^{-1}(x) H^{-1}(y) },
\quad \hat \rho_n = \tbinom{n}{2}^{-1} \sum_{i<j} A_{ij} , \quad & (x,y) \in \left( 0, 1 \right)^2;
\\ \bar A_{ab} = \frac{ 1 }{ h_a \left\{ h_b - \I\left(a=b\right) \right\} } \sum_{ j = n H(b-1) + 1 }^{ n H(b) } \sum_{ i = n H(a-1) + 1 }^{ n H(a) } A_{ \Pi(i) \Pi(j) } , \quad & 1 \leq a,b \leq k.
\end{align*}

We then define the mean-squared error of $ \hat f $ relative to $f$ as
\begin{equation*}
\inf_{\sigma\in \mathcal{M} } \iint_{(0,1)^2} \bigl| f\bigl( \sigma(x),\sigma(y) \bigr) - \hat f\left(x,y;h\right) \bigr|^2 \, dx \, dy ,
\end{equation*}
where $\mathcal{M}$ is the set of all measure-preserving bijections of the form $ \sigma \colon [0,1] \to [0,1]$. This error criterion is based on the so-called cut distance in the theory of graph limits \citep{lovasz2012large}, and allows for all possible rearrangements of the axes of $f$ \citep{choi2012co}.

Any estimator $ \hat f $ can be viewed as a Riemann sum approximation of $f$, and thus we must understand when such sums converge. Lebesgue's criterion asserts that a bounded graphon on $(0,1)^2$ is Riemann integrable if and only if it is almost everywhere continuous. A sufficient condition is that $f$ is $\alpha$-H\"older continuous for some $0 < \alpha \leq 1$, where we write
\begin{equation}
\label{eq:Holder-defn}
f \in \operatorname{\textrm{H\"older}}^\alpha(M) \Leftrightarrow \sup_{(x,y) \neq (x',y') \in (0,1)^2} \frac{ \left| f\left(x,y\right) - f\left(x',y'\right) \right| }{ \left| \left(x,y\right) - \left(x',y'\right) \right|^\alpha } \leq M < \infty .
\end{equation}
This assumption ensures that $f$ is uniformly continuous, so that its approximation error can be controlled through Riemann sums.

Under this model specification, we obtain our main result, which we prove in Appendix~\ref{ap-graphcon}.

\begin{theorem}[Consistency of smooth graphon estimation]\label{graphconst}
Assume a sequence of graphon estimators $ \hat f\left( x , y ; h\right) $ is fitted under the model of~\eqref{pij-model}, with $k = \omega(1)$ and $\bar h = n / k$ the average group size, where
\begin{enumerate}

\item \label{cond:Holder1} The graphon $f$ is symmetric, bounded away from zero and $\alpha$-H\"older continuous, $ 0 < \alpha \leq 1$; 

\item The scaling sequence $ \rho_n $ satisfies $ \rho_n = \omega\bigl( n^{-1} \log^3 n \bigr)$, and $\max_n \rho_n f$ is bounded away from unity;

\item Every admissible partition $H$ has group sizes bounded uniformly above and below by $ h_\vee = o(n)$, $ h_\wedge = \omega( \log^{1/2} n ) $, and may be composed with any permutation $ \Pi $ of $\{1,\ldots,n\}$ to yield $ \hat f\left( x , y ; h \right) $.

\end{enumerate}

Suppose furthermore that the minimum effective sample size of every possible fitted grouping, $ \binom{ h_\wedge }{ 2 } \rho_n $, and the average effective sample size across all groupings, $ \bar h^2 \rho_n$, both grow sufficiently rapidly in $n$:
\begin{equation*}
h_\wedge^2 \rho_n 
= \omega \bigl( \log n \bigr) , \qquad
\bar h^2 \rho_n
= \omega \bigl( \max \left\{ \bar h^2 / n , 1 \right\} \log^3 n \bigr) .
\end{equation*}
Then if $ \hat f\left( x , y ; h \right) $ is fitted by blockmodel maximum profile likelihood estimation as described in Section~\ref{sec:blockmodel-approx} below, the mean-squared error of $\hat f$ satisfies
\begin{equation*}
{\cal O}_P \left( \frac{ \log \bar h }{ \bar h^2 \rho_n } + \sqrt{ \frac{ \log^2 \left( 1 / \rho_n \right) \log \left( n / \bar h \right) }{ n \rho_n } } + \left( \frac{ h_\vee }{ n } \right)^{2\alpha} + \frac{ \log \left( h_\vee / \rho_n \right) }{ n^{ \alpha / 2 } } \right) .
\end{equation*}
\end{theorem}

The terms appearing in this expression each stem from a different portion of the nonparametric inference problem of graphon estimation, and will be derived and discussed in Section~\ref{sec:BlockmodelCons}--\ref{sec:rates} below.

\section{Nonparametric graphon approximation via blockmodels}
\label{sec:blockmodel-approx}

To understand Theorem~\ref{graphconst}, we must first describe how a particular class of statistical network model---the stochastic blockmodel---lends itself naturally to nonparametric approximation. Later, in Section~\ref{sec:BlockmodelCons}, we will establish blockmodel consistency under model misspecification, in settings ranging from dense \citep{chatterjee2012matrix, choi2012co} to very sparse networks.

\subsection{Stochastic blockmodels and nonparametric graphon approximation}

A $k$-community blockmodel $( k, z, \theta )$ is a statistical network model that consists of two main components:
\begin{enumerate}
\item \emph{A community assignment function} $z \colon \left\{1, \ldots, n \right\} \to \left\{1, \ldots, k \right\} $. This mapping assigns each of $n$ network nodes to exactly one of $k$ groupings or ``communities,'' each of size $ h_a , 1 \leq a \leq k $.
\item \emph{A block mean estimator} $ \theta \colon \left\{1, \ldots, k \right\}^n \times \left[0,1\right]^{n \times n} \to \left[0,1\right]^{k \times k}$. This assigns an interaction rate $ \theta_{ab} $ to every pair $(a,b)$ of communities, based on the observations $\left\{ A_{ij} : i \in z^{-1}(a) , j \in z^{-1}(b) \right\}$.
\end{enumerate}

Any community assignment function $z$ thus has two components: a vector $h(z) = \left( h_1, \ldots, h_k \right)$ of community sizes equivalent to some $H$ as defined in~\eqref{H}, and a permutation $ \Pi_z $ of $ \left\{ 1, \ldots, n \right\} $ that re-orders the set of network nodes prior to applying the quantile function $H^{-1}(\cdot /n)$ as defined in~\eqref{Hinv}. Thus the community to which $z$ assigns node $i$ is determined by the composition $ H^{-1} \circ \Pi_z$:
\begin{equation}
\label{eq:z-h}
z_i = H^{-1} \left\{ \Pi_z(i) / n \right\} , \quad 1 \leq i \leq n .
\end{equation}
Each $z$ thus represents a re-ordering of the network nodes, followed by a partitioning of the unit interval. Each $\theta_{ab}$ in turn describes the expected rate of interaction between the nodes in communities $a$ and $b$. 

If $k$ grows with $n$, then the nonparametric properties of blockmodels come to the fore \citep{rohe2011spectral, choi2012stochastic, fishkind2013consistent, zhao2012consistency}. In the theory of graph limits \citep{lovasz2012large}, such a model is known as the ``blowup'' of a weighted graph to the domain $(0,1)^2$, or as a ``stepfunction approximation'' of a given graphon $f\left(x,y\right)$.

There are strong theoretical reasons why an arbitrary graphon should be well approximated by blocks \citep{lovasz2012large}. These reasons stem from a fundamental result in combinatorics known as Szemer\'edi's regularity lemma, which cuts across graph theory, analysis and number theory. In our context, this lemma suggests that any sufficiently large graph behaves approximately like a $( k, z, \theta )$-blockmodel for some $k$. However, this value of $k$ may potentially be very large, and so regularizing strategies are needed to infer a blockmodel approximation with good risk properties while requiring relatively few degrees of freedom.

\subsection{Fitting blockmodels to inhomogeneous random graphs}

Once $f\left(x,y\right)$ has been specified and a uniform random sample $ \{ \xi_i \}_{i=1}^n $ realized, our network reduces to a set of $\binom{n}{2}$ $\operatorname{Bernoulli}( p_{ij} )$ trials that are conditionally independent given $ \{ \xi_i \}_{i=1}^n $. We refer to this as an \emph{inhomogeneous random graph} model \citep{bollobas2007phase} for the observed data matrix $A \in \{0,1\}^{n \times n}$. From~\eqref{pij-model}, the conditional log-probability of observing a given adjacency matrix $A$ is
\begin{equation*}
\log \Pr(A \,\vert\, \{ p_{ij} \}_{i<j} ) = \sum_{i<j} \left\{ A_{ij} \log\left( p_{ij} \right) + \left( 1 - A_{ij} \right) \log\left( 1 - p_{ij} \right) \right\} .
\end{equation*}
 
Adopting the notation of \citet*{choi2012stochastic}, we write the log-likelihood function of a blockmodel $( k, z, \theta )$ with respect to an observed data matrix $A$ as
\begin{align}
\nonumber
L(A;z,\theta) & = \sum_{ i < j } \left\{ A_{ij} \log \theta_{z_i z_j} + \left( 1 - A_{ij} \right) \log \left( 1 - \theta_{z_i z_j} \right) \right\} , \quad 1 \leq i,j \leq n
\\ \nonumber & = \sum_{ a \leq b } \mathop{\sum_{\,\,i \in z^{-1}(a),}}_{j \in z^{-1}(b)} \left\{ A_{ij} \log \theta_{ab} + \left( 1 - A_{ij} \right) \log \left( 1 - \theta_{ab} \right) \right\} , \quad 1 \leq a,b \leq k
\\ \nonumber & = \sum_{ a \leq b } \log \theta_{ab} \mathop{\sum_{\,\,i \in z^{-1}(a),}}_{j \in z^{-1}(b)} A_{ij} + \sum_{ a \leq b } \log \left( 1 - \theta_{ab} \right) \mathop{\sum_{\,\,i \in z^{-1}(a),}}_{j \in z^{-1}(b)} \left( 1 - A_{ij} \right)
\\ \label{eq:blockmodelLL} & = \sum_{ a \leq b } h_{ab}^2 \left\{ \bar A_{ab} \log \theta_{ab} + \left( 1 - \bar A_{ab} \right) \log \left( 1 - \theta_{ab} \right) \right\} ,
\end{align}
where $ \bar A_{ab} $ is the arithmetic average of the values of $A$ in the $(a,b)$th block:
\begin{equation}
\label{eq:hab}
\bar A_{ab} = \frac{ 1 }{ h_{ab}^2 } \mathop{\sum_{\,\,i \in z^{-1}(a),}}_{j \in z^{-1}(b)} A_{ij} , 
\qquad h_{ab}^2 =
\begin{cases}
\binom{ h_a }{ 2 } & \text{if $ a = b $,}
\\ h_a h_b & \text{if $ a \neq b $.}
\end{cases}
\end{equation}
and $h_a$ is the size of the $a$th community. Note that this aligns with our earlier definition of $\hat f$, and that the quantities $ h_{ab}^2 , \bar A_{ab} , \theta_{ab} $ all depend on the community assignment function $z$. The structural zeros along the main diagonal of $A$ imply that $h_{ab}$ differs for diagonal blocks ($a=b$) relative to off-diagonal blocks. 
We see from~\eqref{eq:blockmodelLL} that for any fixed assignment $z \in \{1,\ldots,k\}^n$, the log-likelihood $ L(A;z,\theta) $ of $A$ will be maximized in $\theta \in [0,1]^{k \times k}$ by taking $\theta_{ab} = \bar A_{ab}$. This is because each sample proportion $ \bar A_{ab} $ is an extended maximum likelihood estimator for its expectation; ``extended'', because we include the boundary $ \{0,1\}^{k \times k}$ of the parameter space, allowing for the possibility that $\theta_{ab} = \bar A_{ab} \in \{0, 1 \}$. Thus the extended maximum likelihood estimator coincides with the method of moments estimator for $\theta_{ab}$.

Note that~\eqref{eq:blockmodelLL} is a continuous function in $\theta$, and so (by the extreme value theorem) $L(A;z,\theta)$ attains its supremum over the compact set $[0,1]^{k \times k}$. Thus we ``profile out'' $\theta$ from the log-likelihood $L(A;z,\theta)$:
\begin{align}
\nonumber
L(A;z) & = \max_{ \theta \in [0,1]^{ k \times k } } L(A;z,\theta)
\\ \nonumber & = \sum_{ a \leq b } h_{ab}^2 \left\{ \bar A_{ab} \log \bar A_{ab} + \left( 1 - \bar A_{ab} \right) \log \left( 1 - \bar A_{ab} \right) \right\}
\\ \label{eq:blockmodelPL} & = \sum_{ i < j } \left\{ A_{ij} \log \bar A_{z_i z_j} + \left( 1 - A_{ij} \right) \log \left( 1 - \bar A_{z_i z_j} \right) \right\} .
\end{align}

Any maximizer of~\eqref{eq:blockmodelPL} over a fixed, non-empty subset $\mathcal{Z}_k \subseteq \{1,\ldots,k\}^n $ is a maximum profile likelihood estimator (MPLE) of $z$ with respect to $\mathcal{Z}_k$. We may equivalently re-cast the problem of likelihood maximization as one of Bernoulli Kullback--Leibler divergence minimization, with
\begin{equation*}
\D\left( p \,\middle\vert\middle\vert\, p' \right) = p \log \bigl( \tfrac{ p }{ \, p' } \bigr) + ( 1 - p ) \log \bigl( \tfrac{ 1 - p }{ \, 1 - p' } \bigr) 
\end{equation*}
denoting the Kullback--Leibler divergence of a $\operatorname{Bernoulli}( p' )$ distribution from a $\operatorname{Bernoulli}( p )$ one.

Equipped with this definition, observe that any MPLE $ \hat z(A,\mathcal{Z}_k) $ satisfies
\begin{align}
\label{eq:MPLE}
\hat z(A,\mathcal{Z}_k) & = \operatornamewithlimits{argmax}_{ z \in \mathcal{Z}_k } \sum_{ i < j } \left\{ A_{ij} \log \bar A_{z_i z_j} + \left( 1 - A_{ij} \right) \log \left( 1 - \bar A_{z_i z_j} \right) \right\} 
\\ \nonumber & = \operatornamewithlimits{argmax}_{ z \in \mathcal{Z}_k } \max_{ \theta \in [0,1]^{k \times k} } L(A;z,\theta)
\\ \nonumber & = \operatornamewithlimits{argmin}_{ z \in \mathcal{Z}_k } \min_{ \theta \in [0,1]^{k \times k} } \sum_{ i < j } \D\left( A_{ij} \,\middle\vert\middle\vert\, \theta_{z_i z_j} \right)
\\ \nonumber & = \operatornamewithlimits{argmin}_{ z \in \mathcal{Z}_k } \sum_{ i < j } \D\left( A_{ij} \,\middle\vert\middle\vert\, \bar A_{z_i z_j} \right) .
\end{align}

Maximizing the profile log-likelihood of~\eqref{eq:blockmodelPL} to obtain an MPLE $ \hat z(A,\mathcal{Z}_k) $ is thus equivalent to minimizing the sum of divergences $ \sum_{ i < j } \D\left( A_{ij} \,\middle\vert\middle\vert\, \bar A_{z_i z_j} \right) $. This sum serves as a proxy for its ``oracle'' counterpart based on the matrix $p \in [0,1]^{n \times n}$ of Bernoulli parameters of the underlying generative model. This corresponds to an idealized ``best blockmodel approximation'' of $p$.

With this in mind, we define an ``oracle MPLE'' $z(p,\mathcal{Z}_k)$ in direct analogy to~\eqref{eq:MPLE}. Let $ \bar p(z)_{ab} $ denote the arithmetic average of the $h_{ab}^2$ elements of $p$ in the $(a,b)$th block induced by $z$:
\begin{equation}
\label{eq:bar-p-ab}
\bar p(z)_{ab} = \frac{ 1 }{ h_{ab}^2 } \mathop{\sum_{\,\,i \in z^{-1}(a),}}_{j \in z^{-1}(b)} p_{ij} ,
\end{equation}
where we recall that $ h_{ab}^2 $ also depends on the choice of community assignment function $z$. We then have
\begin{align}
\label{eq:MPLEoracle}
\bar z(p,\mathcal{Z}_k) & = \operatornamewithlimits{argmax}_{ z \in \mathcal{Z}_k } \sum_{ i < j } \left\{ p_{ij} \log \bar p_{ z_i z_j } + \left( 1 - p_{ij} \right) \log \left( 1 - \bar p_{ z_i z_j } \right) \right\}
\\ \nonumber & = \operatornamewithlimits{argmin}_{ z \in \mathcal{Z}_k } \sum_{ i < j } \D\left( p_{ij} \,\middle\vert\middle\vert\, \bar p_{ z_i z_j } \right) .
\end{align}
Observe that neither $ \hat z(A,\mathcal{Z}_k) $ nor $ \bar z(p,\mathcal{Z}_k) $ is unique, since permuting the community labels $\{1,\ldots,k\}$ does not affect the likelihood of community assignment in~\eqref{eq:MPLE} or~\eqref{eq:MPLEoracle}. Even aside from the issue of label switching, we are not guaranteed uniqueness; see \citet{chatterjee2011random} and \citet{rinaldo2013maximum} for discussion of this issue in the specific context of network modeling, as well as \citet{fienberg2012maximum} in the general setting of log-linear models for sparse contingency tables.

\section{Sparse blockmodel consistency under model misspecification}
\label{sec:BlockmodelCons}

We now establish that an observed matrix $A \in \{0,1\}^{n \times n}$ of binary adjacencies yields ``oracle'' information on its generative $p \in (0,1)^{n \times n}$ at a rate that depends both on the sparsity of the network and on the speed at which the admissible network community sizes grow with $n$. We show that for suitable sequences of sets $\mathcal{Z}_k(n) \subseteq \{1,\ldots,k\}^n $ of admissible blockmodels, the maximum profile likelihood assignment method $ \hat z(A,\mathcal{Z}_k) $ implies that the likelihood risk of a fitted blockmodel, as measured by summing the divergences $\D\left( p_{ij} \,\middle\vert\middle\vert\, \bar A_{ \hat z_i \hat z_j} \right) $, approaches the risk $\sum_{ i < j } \D\left( p_{ij} \,\middle\vert\middle\vert\, \bar p_{ z_i z_j } \right)$ of the \emph{best possible} blockmodel approximation as $n$ grows large.

Theorem~\ref{excess} (proved in Appendix~\ref{sec:excessProof}) makes this statement precise and provides a set of sufficient conditions, driven primarily by the effective sample size of each fitted block.

\begin{theorem}[Controlling excess blockmodel risk]\label{excess}
$\!\!\!\!$For each $n = 2, 3, \ldots$, let $A \in \{0,1\}^{n \times n}$ be the adjacency matrix of a simple random graph with independent $\operatorname{Bernoulli}( p_{ij} )$ edges, and consider a corresponding sequence of $k$-community blockmodel estimators, with $k = k(n)$ a function of $n$. Assume:
\begin{enumerate}
      
 \item \label{cond:rhoAvg} \emph{The expected edge density $ \binom{n}{2}^{-1} \sum_{i<j} p_{ij}(n) $ of $A$ does not approach $0$ or $1$ too rapidly in~$n$}: there exists a monotone non-increasing, strictly positive sequence $ \bar \rho(n) $, such that for all $n$ sufficiently large, $ \bar \rho(n) \leq \binom{n}{2}^{-1} \sum_{i<j} p_{ij}(n) \leq 1 - \sqrt{ \bar \rho(n) } $.

 \item \label{cond:rhoMin} \emph{Likewise, no block density $ \{ \bar p_{ z_i z_j }(n) \}_{ i < j, z \in \mathcal{Z}_k(n) } $ approaches $0$ or $1$ too rapidly in~$n$}: there exists a monotone non-increasing, strictly positive sequence $ \rho_\wedge(n) $, such that $\rho_\wedge(n) \leq \bar \rho(n)$ and $\rho_\wedge(n) \leq \bar p_{ z_i z_j }(n) \leq 1 - \sqrt{ \rho_\wedge(n) } $ for all $z \in \mathcal{Z}_k(n)$, $1 \leq i<j \leq n$ and $n$ sufficiently large.

 \item \label{cond:nab} \emph{ The sizes $ \{ h_{z_i}(n) \}_{ 1 \leq i \leq n, z \in \mathcal{Z}_k(n)} $ of all possible communities grow sufficiently rapidly in $n$}: there exists a monotone strictly increasing sequence $ h_\wedge(n) $ taking values in $\{2, \ldots , \lfloor n/k(n) \rfloor$ such that for all $n$ sufficiently large, $ h_\wedge(n) \leq \min_{z \in \mathcal{Z}_k(n)} \left\{ \min_{ 1 \leq i \leq n } h_{z_i}(n) \right\} $.

\end{enumerate}

Assume that the sequences $ \mathcal{Z}_k, \bar \rho, \rho_\wedge, h_\wedge $ are fixed in advance and independent of all other quantities. Let $ \bar h = n / k \in [1,n] $, and suppose that the minimum effective sample size of every possible fitted block, $ \binom{ h_\wedge }{ 2 } \rho_\wedge $, and the average effective sample size across all blocks, $ \bar h^2 \bar \rho$, both grow sufficiently rapidly in $n$:
\begin{equation*}
h_\wedge^2 \rho_\wedge 
= \omega \bigl( \log n \bigr) , \qquad
\bar h^2 \bar \rho
= \omega \bigl( \max \left\{ \bar h^2 / n , 1 \right\} \log^3 n \bigr) .
\end{equation*}
Then for all sequences of subsets $ \mathcal{Z}_k \subseteq \{1,\ldots,k\}^n $ that respect condition~\ref{cond:nab}, we have as $n \rightarrow \infty$ that for any choice of $ z \in \mathcal{Z}_k $, deterministic or random,
\begin{multline}
\label{eq:norm-risk-thm}
\frac{ \sum_{ i < j : \bar A_{ z_i z_j } \notin \{0,1\} } \D\left( p_{ij} \,\middle\vert\middle\vert\, \bar A_{ z_i z_j } \right) }{ \sum_{ i < j : \bar A_{ z_i z_j } \notin \{0,1\} } p_{ij} } 
\\ = \frac{ \sum_{ i<j } \D\left( p_{ij} \,\middle\vert\middle\vert\, \bar p_{ z_i z_j } \right) }{ \sum_{ i < j } p_{ij} } + \mathcal{O}_P \left( \max \left\{ \frac{ 1 }{ \bar h^2 \bar \rho } , \frac{ \log \left( n / \bar h \right) }{ n \bar \rho } \right\} \right) .
\end{multline}
For $ \hat z(A,\mathcal{Z}_k) = \operatorname{argmax}_{ z \in \mathcal{Z}_k } \sum_{ i < j } \left\{ A_{ij} \log \bar A_{z_i z_j} + \left( 1 - A_{ij} \right) \log \left( 1 - \bar A_{z_i z_j} \right) \right\} $,
\begin{multline}
\label{oracy1}
\frac{ \sum_{ \bar A_{ \hat z_i \hat z_j } \notin \{0,1\} } \! \D \! \left( p_{ij} \! \,\middle\vert\middle\vert\, \! \bar A_{ \hat z_i \hat z_j } \! \right) }{ \sum_{ i < j : \bar A_{ \hat z_i \hat z_j } \notin \{0,1\} } p_{ij} } 
\\ = \frac{ \min_{ z \in \mathcal{Z}_k } \! \sum_{ i<j } \! \D \! \left( p_{ij} \! \,\middle\vert\middle\vert\, \! \bar p_{ z_i z_j } \! \right) }{ \sum_{ i<j } p_{ij} } 
+ \mathcal{O}_P \! \left( \! \max \! \left\{ \! \frac{ \log \bar h }{ \bar h^2 \bar \rho } , \sqrt{ \frac{ \log^2 \left( 1 / \rho_\wedge \right) \log \left( n / \bar h \right)}{ n \bar \rho } } \right\} \! \right) .
\end{multline}
These results also hold marginally with respect to the model of~\eqref{pij-model}.
\end{theorem}

Theorem~\ref{excess} is significant because it gives conditions under which the excess risk of a fitted blockmodel converges to zero, implying that blockmodel parameters can be estimated consistently even when the true generative model giving rise to $A$ is unknown. It predicts different rates of convergence for different network sparsity regimes. Depending on the growth of $k$ with $n$, either the first or the second of two rate terms in~\eqref{oracy1} will dominate. 

We may summarize these regimes as follows:
\begin{enumerate}
\item \emph{Dense networks}: If $\rho_\wedge$ and $\bar \rho$ remain constant in $n$, and $k$ grows with $n$ as $ k = \mathcal{O} ( n^{3/4} )$, then Theorem~\ref{excess} predicts a convergence rate of at least $ \sqrt{ \log(n) / n } $. If instead $k$ grows like $n^{\delta}$ for $ 3/4 < \delta < 1$, then this rate will decrease to $ \log n / n^{2(1-\delta)} $. 
\item \emph{Sparse networks}: If $\rho_\wedge$ and $\bar \rho$ decrease like $n^{-2\gamma}$ for $ 0 < \gamma < 1/2 $, and $k = \mathcal{O} ( n^{3/4-\gamma/2} )$, then Theorem~\ref{excess} predicts the rate $ \log(n)^{3/2} / n^{1/2-\gamma} $. If $k$ grows like $n^{\delta}$ for $ 3/4-\gamma/2 < \delta < 1 - \gamma $, then this rate will decrease to $ \log n / n^{2(1-\delta-\gamma)} $.
\item \emph{Ultra-sparse networks}: If $\rho_\wedge$ and $\bar \rho$ decrease like $\log(n)^{3+\beta} / n$ for $ \beta > 0 $, then Theorem~\ref{excess} predicts rate $ \log(n)^{-\beta/2} $ whenever $k = \mathcal{O} ( n^{1/2} )$, matching the regime of \citet*{choi2012stochastic}.
\end{enumerate}
In each of these cases, the given conditions on $\rho_\wedge$ can be relaxed accordingly.

Theorem~\ref{excess} is the first such result known for sparse or ultra-sparse networks---those for which $\bar \rho = o(1)$, so that the average number of connections per node can grow sublinearly, here as slowly as logarithmically in $n$. This complements the recent result of \citet{choi2012co} for fixed-$k$ fitting of dense bipartite graphs---those for which $ \rho_\wedge$ and $\bar \rho $ remain constant, so that the average number of connections per node grows linearly in $n$. Theorem~\ref{excess} extends this regime, allowing for the growth of $k$ with $n$, while also yielding an improved convergence rate of $ \sqrt{ \log(k) / n } $ for dense graphs.

To understand why Theorem~\ref{excess} holds in this setting, we begin by conditioning on a choice of community assignment function $z$. Blocks of network edges then comprise independent sets of independent Bernoulli trials. Conditionally upon $z$, sample proportions $ \bar A_{ z_i z_j } \,\vert\, z $ of these blocks are thus independent Poisson--Binomial variates. Without additional restrictions, however, a fitted block could be any size---even as small as a single Bernoulli trial. Thus it is necessary to constrain the set $\mathcal{Z}_k \subseteq \{1,\ldots,k\}^n$ of admissible blockmodels, and also to constrain the allowable global and local sparsity of the network, so that the effective sample size of every possible $ \bar A_{ z_i z_j } \,\vert\, z $ grows in $n$. This ensures that all block-wise sample proportions $ \bar A_{ z_i z_j } \,\vert\, z $ behave like Normal variates in the large-sample limit, when appropriately standardized. 

There are then two main technical challenges:
\begin{enumerate}
\item \emph{Double randomness}: While every $ \bar A_{ z_i z_j } \,\vert\, z $ is amenable to analysis, choosing $\hat z$ by profile likelihood maximization introduces ``double randomness,'' coupling all blocks and precluding a direct analysis of $ \bar A_{ \hat z_i \hat z_j } $. Instead, we take the approach of \citet*{choi2012stochastic}, and show that results for $ \bar A_{ z_i z_j } \,\vert\, z $ hold uniformly for any choice of $z$ --- and therefore that they also hold for $ \bar A_{ \hat z_i \hat z_j } $.
\item \emph{Likelihood zeros}: The assumption that all $p_{ij} \in (0,1)$ ensures that each $ \D\left( p_{ij} \,\middle\vert\middle\vert\, \bar p_{ z_i z_j } \right) $ is finite. However, $ \D\left( p_{ij} \,\middle\vert\middle\vert\, \bar A_{ \hat z_i \hat z_j } \right) $ will fail to be finite if $ \bar A_{ \hat z_i \hat z_j } \in \{0,1\} $, in which case the $(\hat z_i , \hat z_j)$th block has saturated. Such blocks add $0$ to the likelihood; their parameters are not estimable \citep{fienberg2012maximum}. The theorem conditions allow us to control the probability of these likelihood zeros, by requiring the effective sample size of each block to grow sufficiently rapidly in $n$.
\end{enumerate}
This latter point is particularly important, since only values in the interior of the parameter space $[0,1]^{k \times k}$ are estimable \citep[Theorem~7]{fienberg2012maximum}. As in the case of additional structural zeros \citep[Corollary~8]{fienberg2012maximum}, the Fisher information matrix will be rank-deficient, and the degrees of freedom must be adjusted accordingly in order to obtain correct inferential conclusions. This explains why the random denominator term is necessary in the left-hand side of~\eqref{oracy1}.

We may connect this understanding to the three sparsity regimes described above: the case of dense networks, corresponding to the setting of exchangeable random graphs; that of sparse networks, where the density of network edges $ \smash{ \binom{n}{2}^{-1} } \sum_{i<j} p_{ij} $ decays as some power of $n$; and that of ultra-sparse networks, where the edge density decays at a rate approaching $ \log(n) / n $. This is the so-called connectivity threshold, above which an inhomogeneous random graph will be fully connected with probability approaching $1$ as $ n \rightarrow \infty$ \citep{alon1995note}. If the edge density were instead to decay at a rate of $ 1 / n $---the extremely sparse setting of \citet{bollobas2009metric}---then the resulting networks would fail in general to be connected, and Poisson rather than Normal limiting behavior would hold for each block \citep{olhede2013degree}.

\section{From blockmodels to smooth graphon estimation}
\label{sec:graph-est}

We now present our final result leading to consistent graphon estimation. To go beyond conditional estimation of inhomogeneous random graphs via blockmodels, we will assume additional structure via graphon smoothness. This smoothness will in turn allow us to control estimation risk, by sending the main term in Theorem~\ref{excess} to zero.

A blockmodel first orders the rows and columns of $A$, and then groups its entries according to a vector of community sizes $ h \in \{ 2, \ldots, n \}^k $. This specifies a partition $ H $ in accordance with~\eqref{H}, which in turn induces a piecewise-constant approximation of the graphon $f\left(x,y\right)$ along blocks. To see this, define the domain $ \omega_{ab} \subseteq [0,1)^2 $ of the $(a,b)$th block as
\begin{align}
\nonumber
\omega_{ab} & = \left[ H(a-1) , H(a) \right) \times \left[ H(b-1) , H(b) \right) , \quad 1 \leq a,b \leq k , 
\intertext{and define the blockmodel approximation $ \bar f\left( x , y ; h \right) $ of $ f\left( x , y \right) $ via the local averages $ \bar f_{ab} , 1 \leq a,b \leq k $:}
\label{omegaxy}
\bar f\left( x , y ; h \right) & = \bar f_{ H^{-1}(x) H^{-1}(y) }, \quad \bar f_{ab} = \frac{ 1 }{ \left| \omega_{ab} \right| } \iint_{ \omega_{ab} }
f\left( x , y \right) \, dx \, dy .
\end{align}

If $f\left(x,y\right)$ is smooth as well as bounded, then results from approximation theory allow the error $ \| f - \bar f \| $ to be controlled in any $L_p$ norm, as a function of the maximum over all block diameters $ ( h_a^2 + h_b^2 )^{1/2} / n $ for $ 1 \leq a,b \leq k$ \citep[see also Lemma~\ref{normbound}]{devore1998nonlinear}. 

Recall from~\eqref{eq:z-h} that any blockmodel community assignment vector $z$ is a composition $ H^{-1} \circ \Pi_z$ for some partition $ H $ of $[0,1]$ and permutation $ \Pi_z $ of $\{1, \ldots, n\}$, so that $ z_i = H^{-1} \left\{ \Pi_z(i) / n \right\} , 1 \leq i \leq n$. From~\eqref{eq:bar-p-ab}, we may express $ \bar p(z) $ for any $ 1 \leq a,b \leq k $ as
\begin{align}
\nonumber
\bar p( z )_{ab} & = \frac{ 1 }{ h_{ab}^2 } \sum_{i<j} p_{ij} \I\bigl[ H^{-1} \left\{ \Pi_z(j) / n \right\} = b \bigr] \I\bigl[ H^{-1} \left\{ \Pi_z(i) / n \right\} = a \bigr]
\\ \label{eq:pbar-xi} & = \frac{ 1 }{ h_{ab }^{2} } \sum_{ j = n H(b-1)+1 }^{ n H(b) } \sum_{ i = n H(a-1)+1}^{ n H(a) \I\left( a \neq b \right) + \left( j - 1 \right) \I\left( a = b \right) } p_{ \, \Pi_z^{-1}(i) \, \Pi_z^{-1}(j) } .
\end{align}
Thus $ \bar p(z)_{ a b } $ is an average over $ h_{ab }^{2} $ graphon evaluations $ f\,\bigl( \xi_{\Pi_z^{-1}(i)},\xi_{\Pi_z^{-1}(j)} \bigr) $, since the model of~\eqref{pij-model} asserts that $ p_{ij}(n) \propto f\left(\xi_i,\xi_j\right) $. These evaluations occur at random points determined by $\{ \xi_1, \ldots \xi_n \}$ according to the inverse of the permutation $ \Pi_z $, while $ H $ determines the size of each block.

From this simple observation, we will show that it is possible to relate $ \bar p(z)_{ a b } $ to $f\left(x,y\right)$ by choosing an ``oracle'' permutation $ \Pi_z(i) $ whose inverse yields the ordered sample $\{ \xi_{(1)}, \ldots \xi_{(n)} \}$. To see this, first note that whenever the H\"older condition of~\eqref{eq:Holder-defn} is satisfied, we have by Lemma~\ref{ftilde} that
\begin{equation*}
f\left( \xi_{(i)},\xi_{(j)} \right) 
= f\,\bigl( \tfrac{i}{n+1} , \tfrac{j}{n+1} \bigr) 
+ {\cal O}_P\bigr( n^{-\alpha/2} \bigr) ,
\end{equation*}
because each $ \xi_{(i)} $ converges in probability to its expectation $ i / \left(n+1\right) $ at a rate no worse than $ n^{-1/2} $, and~\eqref{eq:Holder-defn} relates this to $ \bigl| f\left( \xi_{(i)},\xi_{(j)} \right) - f\,\bigl( \frac{i}{n+1} , \frac{j}{n+1} \bigr) \bigr| $. Now take $ \Pi_z(i) = (i)^{-1} $, where $(i)^{-1}$ denotes the rank of $\xi_i$ from smallest to largest, and observe that $ f\,\bigl( \xi_{\Pi_z^{-1}(i)} , \xi_{\Pi_z^{-1}(j)} \bigr)$ evaluates to $ f\left( \xi_{(i)} , \xi_{(j)} \right) $. 

The key point is that when $f$ is $\alpha$-H\"older continuous, then convergence of the ordered sample $ \{ \xi_{(i)} \}_{i=1}^n $ governs convergence of the random averages comprising $ \bar p( z )_{ab} $ in ~\eqref{eq:pbar-xi}. Indeed, if $ h_\vee $ uniformly upper-bounds the largest possible community size, then by Lemma~\ref{lem:pbar-ftilde}, we have that
\begin{equation*}
\Pi_z = ( \cdot )^{-1}
\, \Rightarrow \,
\rho_n^{-1} \bar p_{ z_{(i)} z_{(j)} } \, - \, \bar f\left(\xi_{(i)},\xi_{(j)} ; h \right) = {\cal O}_P \bigl( n^{-\alpha/2} + ( n / h_\vee )^{-\alpha} \bigr) ,
\end{equation*}
where we recall from~\eqref{omegaxy} that ${\bar f}\left(x,y ; h \right) $ is the local block average of $f$.

As a consequence, we can control the oracle estimation risk featured in Theorem~\ref{excess} as follows.

\begin{theorem}[Controlling absolute risk]\label{absolute}
Assume in the scaled exchangeable graph model of~\eqref{pij-model} that:
\begin{enumerate}
\item \label{cond:Holder} The graphon $f$ is a positive, symmetric function on $(0,1)^2$, and is $\alpha$-H\"older continuous, $ 0 < \alpha \leq 1$; 
\item \label{cond:rho} Furthermore, $f$ is bounded away from zero and $\max_n \rho_n f$ is bounded away from unity;
\item \label{cond:oracle-perm} Each set $ \mathcal{Z}_k(n) \subseteq \{1,\ldots,k\}^n $ of admissible blockmodel assignments has the following property: If $ H $ is generated by some $ z \in \mathcal{Z}_k $, then $ H^{-1} \circ \Pi \in \mathcal{Z}_k $ for every permutation $ \Pi $ of $\{1,\ldots,n\}$.
\end{enumerate}
Then for $ h_\vee (n)$ the largest community size in each $ \mathcal{Z}_k(n) $, the oracle likelihood risk in Theorem~\ref{excess} satisfies
\begin{equation}
\label{eq:min-risk-Holder}
\frac{ \min_{ z \in \mathcal{Z}_k } \sum_{ i<j } \D\left( p_{ij} \,\middle\vert\middle\vert\, \bar p_{ z_i z_j } \right) }{ \sum_{ i<j } p_{ij} } 
= {\cal O}_P \bigl( n^{-\alpha} + ( n / h_\vee )^{-2\alpha} \bigr) .
\end{equation}
\end{theorem}

We prove this theorem in Appendix~\ref{sec:sbs-pf} by using the oracle choice of permutation $ (\cdot)^{-1} $ to upper-bound the risk via a block approximation ${\bar f}\left(x,y ; h \right) $ of $f\left(x,y \right)$, based on some $ z^* $ which achieves the minimum in~\eqref{eq:min-risk-Holder}. Conditions~\ref{cond:Holder} and~\ref{cond:rho} are then sufficient to guarantee the claimed rate of approximation. Condition~\ref{cond:oracle-perm} ensures that $ H^{-1} \circ (\cdot)^{-1} \in \mathcal{Z}_k $, since we do not know $ z^* $ or the requisite ordering $ (\cdot)^{-1} $ in advance.

\section{Rates of convergence}
\label{sec:rates}

We see directly that the rate of convergence in Theorem~\ref{absolute} depends on the H\"older continuity of $f$ in two ways: through the convergence of the ordered sample $ \{ \xi_{(i)} \}_{i=1}^n $ (variance), and through the rate at which $ h_\vee / n $ goes to zero in $n$ (bias).  This rate is also self-scaling relative to the sparsity of the network, as it does not depend on $ \rho_n $.

In contrast, Theorem~\ref{excess} depends strongly both on the network sparsity factor $\rho_n$, as well as the minimum and average admissible block sizes, $h_\wedge$ and $ \bar h $.  The conditions of Theorem~\ref{excess} ensure that excess blockmodel risk can be controlled under model misspecification, enabling groupings of nodes with good risk properties to be estimated, despite the variability of the data.

Together, the results of Theorems~\ref{excess} and~\ref{absolute} enable us to establish mean-square graphon consistency at the rates indicated in Theorem~\ref{graphconst}, namely
\begin{equation*}
{\cal O}_P \left( \frac{ \log \bar h }{ \bar h^2 \rho_n } + \sqrt{ \frac{ \log^2 \left( 1 / \rho_n \right) \log \left( n / \bar h \right) }{ n \rho_n } } + \left( \frac{ h_\vee }{ n } \right)^{2\alpha} + \frac{ \log \left( h_\vee / \rho_n \right) }{ n^{ \alpha / 2 } } \right) .
\end{equation*}
The first two terms come directly from Theorem~\ref{excess}, while the third is from Theorem~\ref{absolute}.  The final term comes from relating the discrete quantities featured in these theorems to the graphon itself, and is driven in part by the fact that we do not know the ordering of the data relative to the $\operatorname{Uniform}(0,1)$ variates $\{ \xi_i \}_{i=1}^n$ by which the graphon is sampled.  The $\mathcal{O}\bigl( n^{-1/2} \bigr)$ variance of the ordered sample $\{ \xi_{(i)} \}_{i=1}^n$ subsequently appears, and is modulated by the regularity of the graphon through its H\"older continuity exponent $ \alpha $.

\section{Conclusion}

In this article we have established a number of new results within a nonparametric framework for network inference, based on graphons as natural limiting objects.  Understanding graphons as analytic objects, as well as the behavior of dense and sparse networks based on them, is fundamental to advancing our nonparametric understanding of networks.  

To this end, we have established consistency of graphon estimation under general conditions, giving rates which include the important practical setting of sparse networks.  By treating dense and sparse stochastic blockmodels with a growing number of classes, under model misspecification, our results improve substantially upon what is currently known in the literature.
 
Our results link strongly to approximation theory, nonparametric function estimation, and the theory of graph limits, and thus provide for a foundational understanding of nonparametric statistical network analysis.

\appendix

\newcommand{\Graphonconsistency}{\ref{graphconst}}
\section{Proof of Theorem~\protect\Graphonconsistency{} and its lemmas}
\label{ap-graphcon}

\subsection{Proof of Theorem~\protect\Graphonconsistency{}}

\begin{proof}
We note from Lemma~\ref{rhohat} that for $ (x,y) \in (0,1)^2$
\begin{equation*}
\hat{f}\left(x,y;h\right)=\hat{\rho}_n^{-1} \bar A_{H^{-1}(x) H^{-1}(y)}
=\bigl\{1+{\cal O}_P\bigl( n^{-1/2} \bigr) \bigr\}{\rho}_n^{-1} \bar A_{H^{-1}(x) H^{-1}(y)} .
\end{equation*}
Recalling the definition of $ \bar A_{ab} $, we see that uniformly for all choices of $H$ and $\Pi$, and for all $ 1 \leq a, b \leq k$, we have $ 0 \leq \E \bar A_{ab} \leq \rho_n \sup_{(x,y)\in(0,1)^2} f\left(x,y\right) $ and $ 0 \leq \E \bar A_{ab}^2 \leq \rho_n^2 \sup_{(x,y)\in(0,1)^2} f^2\left(x,y\right) $. 

Since $f$ is by hypothesis H\"older continuous on a bounded domain, it is bounded, and thus $ \bar A_{ab} = {\cal O}_P\bigl( \rho_n \bigr)$ and $ \bar A_{ab}^2 = {\cal O}_P\bigl(\rho_n^2 \bigr)$ by Markov's inequality. We will thus expand the squared error term in the integrand of the graphon mean-squared error pointwise, using the fact that the error term should be evaluated at the infimum over measure preserving bijections. Therefore this error be upper-bounded by its evaluation at some $ \sigma^\ast \in \mathcal{M} $, which we will choose in accordance with the proof of Lemma~\ref{graphconst-D} below:
\begin{align*}
&\inf_{\sigma\in \mathcal{M} } \!\!\iint_{(0,1)^2} \!\!\!\!\!\!\!\!\!\bigl| f\bigl( \sigma(x),\sigma(y) \bigr) \!\!-\!\!
\bigl\{1\!+\!{\cal O}_P\bigl( n^{-1/2} \bigr) \bigr\}{\rho}_n^{-1} \bar A_{H^{-1}(x) H^{-1}(y)} \bigr|^2 \, dx \, dy\\
& \le \iint_{(0,1)^2} \bigl| f\bigl( \sigma^\ast(x),\sigma^\ast(y) \bigr) -
{\rho}_n^{-1} \bar A_{H^{-1}(x) H^{-1}(y)} \bigr|^2 \, dx \, dy+{\cal O}_P\bigl( n^{-1/2} \bigr)
\\ & \leq \iint_{ \hat f \notin \{0,1\} } \bigl| f\bigl( \sigma^\ast(x),\sigma^\ast(y) \bigr) - {\rho}_n^{-1} \bar A_{H^{-1}(x) H^{-1}(y)} \bigr|^2 \, dx \, dy+{\cal O}_P\bigl( n^{-1/2} \bigr)
\\ & \leq 2 \left( \sup f \right)\!\!\! \iint_{ \hat f \notin \{0,1\} } \hskip-0.75cm \rho_n^{-1} \D\left\{ \rho_n f\bigl( \sigma^\ast(x),\sigma^\ast(y) \bigr) \,\middle\vert\middle\vert\, \rho_n \hat f\left( x , y ; h \right) \right\} \,dx \,dy + {\cal O}_P\bigl( n^{-1/2} \bigr) ,
\end{align*}
where the last two lines follow from Lemmas~\ref{fhatbad} and~\ref{KLDiv}, respectively. By Lemma~\ref{graphconst-D}, we have
\begin{multline*}
2 \left( \sup f \right) \iint_{ \hat f \notin \{0,1\} } \hskip-0.75cm \rho_n^{-1} \D\left\{ \rho_n f\bigl( \sigma^\ast(x),\sigma^\ast(x) \bigr) \,\middle\vert\middle\vert\, \rho_n \hat f\left( x , y ; h \right) \right\} \,dx \,dy
= 2 \left( \sup f \right) \\ \cdot\iint_{ \hat f \notin \{0,1\} } \hskip-0.75cm f\left( x , y \right) \,dx \,dy
\textstyle
\frac{ \sum_{ i < j : \bar A_{ z_i z_j } \notin \{0,1\} } \D\left( p_{ij} \,\middle\vert\middle\vert\, \bar A_{ z_i z_j } \right) }
{ \sum_{ i < j: \bar A_{z_i z_j } \notin \{0,1\} } p_{ij} } \left\{ 1 + {\cal O}_P\left( n^{-\alpha / 2  } \right) \right\}\\
+ {\cal O}_P \left( \frac{ \log \left( h_\vee / \rho_n \right) }{ n^{ \alpha / 2 } } + \frac{ \log h_\vee }{ \rho_n n } \right) ,
\end{multline*}
uniformly in $z$. The conditions of Theorem~\ref{graphconst} are sufficient for Theorems~\ref{excess} and~\ref{absolute} to hold, and so if $ \hat f$ is fitted by maximum profile likelihood, then we may substitute terms from Theorems~\ref{excess} and~\ref{absolute} to obtain
\begin{multline*}
2 \left( \sup f \right) \iint_{ \hat f \notin \{0,1\} } \hskip-0.75cm \rho_n^{-1} \D\left\{ \rho_n f\bigl( \sigma^\ast(x),\sigma^\ast(x) \bigr) \,\middle\vert\middle\vert\, \rho_n \hat f\left( x , y ; h \right) \right\} \,dx \,dy
= 2 \left( \sup f \right) \\ \cdot\iint_{ \hat f \notin \{0,1\} } \hskip-0.75cm f\left( x , y \right) \,dx \,dy
\cdot
\textstyle \left[\textstyle {\cal O}_P \bigl( n^{-\alpha} + \bigl( \frac{n }{ h_\vee} \bigr)^{-2\alpha} \bigr) + \textstyle \mathcal{O}_P \left(\max \left\{ \frac{ \log \bar h }{ \bar h^2 \rho_n } , \sqrt{ \frac{ \log^2 \frac{ 1 }{ \rho_n } \log \frac{ n }{ \bar h } }{ n \rho_n } } \right\} \right) \right] \\+ {\cal O}_P \left( \frac{ \log \left( h_\vee / \rho_n \right) }{ n^{ \alpha / 2 } } + \frac{ \log h_\vee }{ \rho_n n } \right) .
\end{multline*}
\begin{equation*}
\vspace{-2\baselineskip}
\end{equation*}
\end{proof}

\subsection{Auxiliary lemmas needed for Theorem~\protect\Graphonconsistency{}}

\begin{lemma}\label{rhohat}
Assume the setting of Theorem~\ref{graphconst}. Then $\E \hat \rho_n = \rho_n $, $ \var \hat \rho_n = {\cal O}\bigl( \rho_n^2 / n \bigr)$.
\end{lemma}

\begin{proof}
Since $i<j$ and $k<l$, we have that $ \E A_{ij} \,\vert\, \xi = \rho_n f\left( \xi_i , \xi_j \right) $ and $
\cov\left(A_{ij},A_{kl}\,\middle\vert \, \xi\right) = \rho_n f(\xi_{i},\xi_{j})\left\{1-\rho_n f(\xi_{i},\xi_{j})\right\} \I\left(i=k\right) \I\left(j=l\right) $.
We first use the law of total expectation to deduce
\begin{equation*}
\textstyle
\E \hat \rho_n = \tbinom{n}{2}^{-1} \sum_{i<j} \E_{\xi} \left\{ \rho_n f\left(\xi_i,\xi_j\right) \right\} = \rho_n \iint_{(0,1)^2} f(x,y)\,dx\, dy=\rho_n.
\end{equation*}
The necessary marginal variances and covariances can then be established hierarchically:
\begin{align*}
\nonumber
\var(A_{ij})&=\E_{\xi}\left\{\var\left(A_{ij}\,\middle\vert \,\xi\right)\right\}+\var_{\xi}\left\{\E(A_{ij}\,\middle\vert \,\xi)\right\}
\\ & = \left\{ \E \rho_n f(\xi_{i},\xi_{j}) \right\} \left\{ 1 - \E \rho_n f(\xi_{i},\xi_{j}) \right\} 
 = \rho_n \left( 1 - \rho_n \right) ,
\\ \cov(A_{ij},A_{kl})&=\!\E_{\xi}\!\left\{\cov\!\left(A_{ij},A_{kl}\middle\vert \xi\right)\right\}\!+\!\cov_{{\xi}}\!\left\{\E\left(\!A_{ij}\middle\vert {\xi}\right),\!
\E\left(A_{kl}\,\middle\vert \,{\xi}\right)\right\}\!,(i,j) \neq (k,l).
\end{align*}
Since $ \E f\left(\xi_i,\xi_j\right)f\left(\xi_k,\xi_l\right) = \iint_{(0,1)^2} f^2(x,y)\,dx\, dy $ if $i=k$ and $ j=l$, and $ \bigl\{ \iint_{(0,1)^2} f(x,y)\,dx\, dy \bigr\}^2 $ if $i\neq k$ and $j\neq l$, we obtain when either $i \neq k $ or $ j \neq l $ that
\begin{align*}
\cov_{{\xi}}\left(A_{ij} , A_{kl} \right)&=
\cov_{{\xi}}\left\{\E(A_{ij}\,\vert\,{\xi}),\E(A_{kl}\,\middle\vert \,{\xi})\right\}\\
&=\E_{\xi}\left\{\rho_n f(\xi_{i},\xi_{j})\rho_n f(\xi_{k},\xi_{l}) \right\}-
\E_{\xi}\left\{\rho_n f(\xi_{i},\xi_{j})\right\}\E_{\xi}\left\{\rho_n f(\xi_{k},\xi_{l})\right\}\\
&\le \rho_n^2 \max\{\var f(\xi_{i},\xi_{j}) ,\var f(\xi_{k},\xi_{l})\}\\
&\textstyle \le \rho_n^2\left[\iint_{(0,1)^2} \left\{ f(x,y) \right\}^2 \,dx\, dy-\bigl\{\iint_{(0,1)^2} f(x,y)\,dx\, dy\bigr\}^2 \right].
\end{align*}
Because $\cov_{{\xi}}\left\{A_{ij} , A_{kl} \right\}=0$ when all $i,j$ and $k,l$ are distinct, and since $i\neq j$ and $k\neq l$, we obtain
\begin{align*}
\var \hat \rho_n & =\tbinom{n}{2}^{-2}\sum_{i<j} \var A_{ij}+\tbinom{n}{2}^{-2} \!\!\!\sum_{i\neq k \cup j\neq l} \cov\left(A_{ij},A_{kl}\right)\\
&\le\tbinom{n}{2}^{-2} \rho_n \left( 1 - \rho_n \right) +\tbinom{n}{2}^{-2} \!\!\!
\sum_{i\neq k \cup j\neq l} \cov\left(A_{ij},A_{kl}\right)\left[\I(i=k)+\I(i=l)\right.\\
&\left.+\I(j=k)+\I(j=l) \right]\\
&\textstyle\le \tbinom{n}{2}^{-2} \rho_n \left( 1 - \rho_n \right)+4
n\tbinom{n}{2}^{-2}\rho_n^2 \left[\iint_{(0,1)^2} \left\{ f(x,y) \right\}^2 \,dx\, dy-1 \right].
\end{align*}
The order term of $ {\cal O}(\rho_n^2/n) $ follows, as $\rho_n^2/n\ge \rho_n/n^2 \Leftrightarrow \rho_n\ge 1/n$, since $\rho_n=\omega\bigl(n^{-1}\log^3 n \bigr)$.
\end{proof}

\begin{lemma}\label{fhatbad}
Assume the setting of Theorem~\ref{graphconst}. Then
\begin{equation*}
\sup_{ \sigma \in \mathcal{M} }
\iint_{ \hat f \in \{0,1\}} \bigl| f\bigl( \sigma(x),\sigma(x) \bigr) - \hat f\left(x,y;h\right) \bigr|^2 \, dx \, dy
= \mathcal{O}_P \left( e^{ - \binom{ h_\wedge }{ 2 } \rho_n + 2 \log ( 1/ \rho_n ) } \right) .
\end{equation*}
\end{lemma}

\begin{proof}
We apply Lemma~\ref{lem:normalization} to control $\sum_{i<j} \I\left(\bar{A}_{ab}\in\{0,1\}\right) $ marginally, after observing that
\begin{align*}
\sup_{ \sigma \in \mathcal{M} } &
\iint_{ \hat f \in \{0,1\}} \bigl| f\bigl( \sigma(x),\sigma(x) \bigr) - \hat f\left(x,y;h\right) \bigr|^2 \, dx \, dy
 \le\iint_{ \hat f \in \{0,1\}} 2\rho_n^{-2} \, dx \, dy
\\ & = 2 \left( \rho_n n \right)^{-2} \!\!\!\! \sum_{a,b:\bar{A}_{ab}\in\{0,1\}} h_a h_b 
\leq 2 \left( \rho_n n \right)^{-2} \!\!\!\! \sum_{a\leq b:\bar{A}_{ab}\in\{0,1\}} 4h_{ab}^2 
\\
&= 8 \left( \rho_n n \right)^{-2} \sum_{i<j} \I\left(\bar{A}_{ab}\in\{0,1\}\right) .
\end{align*}
\begin{equation*}
\vspace{-2\baselineskip}
\end{equation*}
\end{proof}

\begin{lemma}\label{graphconst-D}
Assume the setting of Theorem~\ref{graphconst}. Then for any $z\in {\cal Z}_k$, 
\begin{multline}
\label{order}
\frac{ \inf_{ \sigma \in \mathcal{M} } \iint_{ \hat f \notin \{0,1\} } \rho_n^{-1} \D\left\{ \rho_n f\bigl( \sigma(x) , \sigma(y) \bigr) \,\middle\vert\middle\vert\, \rho_n \hat f\left( x , y ; h \right) \right\} \,dx \,dy }{ \iint_{ \hat f \notin \{0,1\} } f\left( x , y \right) \,dx \,dy }
\\ =\textstyle
\frac{ \sum_{ i < j : \bar A_{ z_i z_j } \notin \{0,1\} } \D\left( p_{ij} \,\middle\vert\middle\vert\, \bar A_{ z_i z_j } \right) }{ \sum_{ i < j: \bar A_{z_i z_j } \notin \{0,1\} } p_{ij} } \left\{ 1 + {\cal O}_P\left( \frac{ 1 }{ n^{\alpha / 2 } } \right) \right\} 
+ {\cal O}_P \left( \frac{ \log \left( h_\vee / \rho_n \right) }{ n^{ \alpha / 2 } } + \frac{ \log h_\vee }{ \rho_n n } \right) .
\end{multline}
\end{lemma}

\begin{proof}
We first treat the numerator of~\eqref{order}, whose infimum is over $\mathcal{M}$, the set of all measure-preserving bijective maps of the form $\sigma \colon [0,1] \to [0,1]$. We may write
\begin{align}
\nonumber
0 & \leq \inf_{ \sigma \in \mathcal{M} } \iint_{ \hat f \notin \{0,1\} } \rho_n^{-1} \D\left\{ \rho_n f\bigl( \sigma(x) , \sigma(y) \bigr) \,\middle\vert\middle\vert\, \rho_n \hat f\left( x , y ; h \right) \right\} \,dx \,dy
\\ \label{KLdiv-estimator} & = \textstyle \inf_{ \sigma \in \mathcal{M} } \sum_{ a,b : \bar A(z)_{ a b } \notin \{0,1\} }
\iint_{ \omega(z)_{ a b } } \rho_n^{-1} \D\left\{ \rho_n f\bigl( \sigma(x) , \sigma(y) \bigr) \,\middle\vert\middle\vert\, \bar A(z)_{ a b } \right\} \,dx \,dy ,
\end{align}
since $ \hat f $ is constant on blocks. Observe that for each individual summand in~\eqref{KLdiv-estimator}, we may write
\begin{align}
\label{eq:local-split} 
&\iint_{ \omega(z)_{ a b } } 
\rho_n^{-1} \D\left\{ \rho_n f\bigl( \sigma(x) , \sigma(y) \bigr) \,\middle\vert\middle\vert\, \bar A(z)_{ a b } \right\} \,dx \,dy \\
\nonumber
&\qquad =
\int_{ H(b-1) }^{ H(b) } \int_{ H(a-1) }^{ H(a) } \!\!\!\rho_n^{-1} \D\left\{ \rho_n f(\cdot) \,\middle\vert\middle\vert\, \bar A(z)_{ a b } \right\} \,dx \,dy
\\ &\qquad =  \nonumber
\sum_{ j = n H(b-1)+1}^{ n H(b) } 
\sum_{ i = n H(a-1)+1}^{ n H(a) } \!\!\!\!\!
\int_{\frac{j-1}{n}}^{\frac{j}{n}}
\int_{\frac{i-1}{n}}^{\frac{i}{n}} \!\!\!
\rho_n^{-1} \D\left\{ \rho_n f\bigl( \sigma(x) , \sigma(y) \bigr) \,\middle\vert\middle\vert\, \bar A(z)_{ a b } \right\} dx dy .
\end{align}
We now restrict our choice of $ \sigma \in \mathcal{M} $ to satisfy the following property:
\begin{equation}
\label{eq:perm-choice}
\int_{\frac{j-1}{n}}^{\frac{j}{n}}
\int_{\frac{i-1}{n}}^{\frac{i}{n}}\!\!\!\!\!
f\bigl( \sigma(x) , \sigma(y) \bigr) 
dx dy
=\!\!\!
\int_{\frac{\Pi(j)-1}{n}}^{\frac{\Pi(j)}{n}}
\int_{\frac{\Pi(i)-1}{n}}^{\frac{\Pi(i)}{n}}
\!\!\!\!\!f\left( x , y \right)
dx dy ,
\, 1 \leq i , j \leq n,
\end{equation}
for some permutation $ \Pi $ of $ \{ 1, \ldots, n \} $. Such a choice of measure-preserving bijection can always be made, as it simply partitions the unit interval into $ n + 1 $ subintervals of the form $\left[ (i-1)/n , i/n \right), 1 \leq i \leq n$, and permutes their order in accordance with $\Pi$. We make this choice in order to preserve the H\"older continuity of $f$ on each domain $ (x,y)\in \left( \frac{i-1}{n},\frac{i}{n} \right) \times \left( \frac{j-1}{n},\frac{j}{n} \right) $, as will be shown below.

Thus we may write, combining~\eqref{KLdiv-estimator}--\eqref{eq:perm-choice}, 
\begin{multline}
\label{eq:perm-bound}
\!\! \inf_{ \sigma \in \mathcal{M} } 
\iint_{ \hat f \notin \{0,1\} } \rho_n^{-1} \D\left\{ \rho_n f\bigl( \sigma(x) , \sigma(y) \bigr) \,\middle\vert\middle\vert\, \rho_n \hat f\left( x , y ; h \right) \right\} \,dx \,dy
\\ \!\! \leq \min_{ \Pi \in S_n} \hskip-0.5cm
\sum_{ a,b : \bar A_{ a b } \notin \{0,1\} }
\sum_{ j = n H(b-1)+1}^{ n H(b) } 
\sum_{ i = n H(a-1)+1}^{ n H(a) } \\
\cdot
\int_{\frac{\Pi(j)-1}{n}}^{\frac{\Pi(j)}{n}}
\int_{\frac{\Pi(i)-1}{n}}^{\frac{\Pi(i)}{n}}
\rho_n^{-1} \D\left\{ \rho_n f\left( x , y \right) \,\middle\vert\middle\vert\, \bar A(z)_{ a b } \right\} \,dx \,dy ,
\end{multline}
with $ S_n $ the set of permutations of $ \{ 1, \ldots, n \} $. From Lemma~\ref{f-small-domain} we then obtain
\begin{multline*}
n^2 \int_{\frac{\Pi(j)-1}{n}}^{\frac{\Pi(j)}{n}}
\int_{\frac{\Pi(i)-1}{n}}^{\frac{\Pi(i)}{n}}
\rho_n^{-1} \D\left\{ \rho_n f\left( x , y \right) \,\middle\vert\middle\vert\, \bar A(z)_{ a b } \right\} \,dx \,dy 
\\ = \rho_n^{-1} \D\left[ \rho_n f\left( \xi_{ \mybig( \Pi\left\{i\right\} \mybig) } , \xi_{ \mybig( \Pi\left\{j\right\} \mybig) } \right) \,\middle\vert\middle\vert\, \bar A(z)_{ a b } \right] + {\cal O}_P \left( \left\{ \log \left( 1 / \rho_n \right) + \log \tbinom{ h_\vee }{ 2 } \right\} n^{-\alpha / 2} \right) ,
\end{multline*}
where $ \xi_{ \left( \Pi\left\{i\right\} \right) } $ is the $ \Pi(i) $th element of the ordered sample $ \{ \xi_{(i)} \}_{i=1}^n $. Starting from~\eqref{eq:perm-bound}, we then have
\begin{align}
\nonumber
& \inf_{ \sigma \in \mathcal{M} }
\iint_{ \hat f \notin \{0,1\} } \rho_n^{-1} \D\left\{ \rho_n f\bigl( \sigma(x) , \sigma(y) \bigr) \,\middle\vert\middle\vert\, \rho_n \hat f\left( x , y ; h \right) \right\} \,dx \,dy
\\ \nonumber & \quad \leq \min_{ \Pi \in S_n} \frac{ 1 }{ n^2 } 
\sum_{ a,b : \bar A(z)_{ a b } \notin \{0,1\} } 
\sum_{ j = n H(b-1)+1}^{ n H(b) } \left[
\sum_{ i = n H(a-1)+1}^{ n H(a) \I\left( a \neq b \right) + \left( j - 1 \right) \I\left( a = b \right) } \hskip-1cm \left\{ 1 + \I\left( a = b \right) \right\}
+ \I\left( i = j \right)
\right]
\\ \nonumber &  \cdot
\rho_n^{-1} \D\left[ \rho_n f\left( \xi_{ \mybig( \Pi\left\{i\right\} \mybig) } , \xi_{ \mybig( \Pi\left\{j\right\} \mybig) } \right) \,\middle\vert\middle\vert\, \bar A(z)_{ a b } \right] 
+ {\cal O}_P \left( \left\{ \log \left( 1 / \rho_n \right) + \log \tbinom{ h_\vee }{ 2 } \right\} n^{ - \alpha / 2 } \right)
\\ \nonumber & \quad \leq \frac{ 1 }{ n^2 } \left[ 2
\!\!\!\! \!\!\!\!\! \sum_{ i < j : \bar A(z)_{ z_i z_j } \notin \{0,1\} } \hskip-0.7cm \rho_n^{-1} \D\left\{ p_{ij} \,\middle\vert\middle\vert\, \bar A(z)_{ z_i z_j } \right\} + \!\! \!\!\!\!\!\!\! \!\!\!\!\!\sum_{ 1 \leq i \leq n : \bar A(z)_{ z_i z_i } \notin \{0,1\} } \hskip-0.8cm \rho_n^{-1} \D\left\{ \rho_n f\left( \xi_i , \xi_i \right) \,\middle\vert\middle\vert\, \bar A(z)_{ z_i z_i } \right\} \right]
\\ \label{eq:perm-choice} & \hskip5.35cm + {\cal O}_P \left( \left\{ \log \left( 1 / \rho_n \right) + \log \tbinom{ h_\vee }{ 2 } \right\} n^{ - \alpha / 2 } \right) ,
\end{align}
where we have chosen $ \Pi = \left( \cdot \right)^{-1} \circ \Pi_z^{-1} $, so that $ \Pi(i) = \left( \Pi_z^{-1} \{ i \} \right)^{-1} $, with $ \left( i \right)^{-1} $ the rank of $ \xi_i $, from smallest to largest. This choice allows us to match each $ \xi_{ \left( \Pi\left\{i\right\} \right) } $ to the corresponding group assignment $ z_i $ of the $i$th network node. To see this, recall from~\eqref{eq:z-h} that $ z_i = H^{-1} \left\{ \Pi_z(i) / n \right\} , 1 \leq i \leq n $, and from~\eqref{eq:hab} and~\eqref{eq:barp} respectively that
\begin{align*}
\bar A(z)_{ ab }
& = \frac{ 1 }{ h_{ab }^{2} } \sum_{ j = n H(b-1)+1 }^{ n H(b) } \sum_{ i = n H(a-1)+1}^{ n H(a) \I\left( a \neq b \right) + \left( j - 1 \right) \I\left( a = b \right) } A_{ \Pi_z^{-1}(i) \Pi_z^{-1}(j) } ,
\\ \bar p(z)_{ ab } 
& = \frac{ 1 }{ h_{ab }^{2} } \sum_{ j = n H(b-1)+1 }^{ n H(b) } \sum_{ i = n H(a-1)+1}^{ n H(a) \I\left( a \neq b \right) + \left( j - 1 \right) \I\left( a = b \right) } \rho_n f\left(\xi_{\Pi_z^{-1}(i)},\xi_{\Pi_z^{-1}(j)} \right) .
\end{align*}
Note that $\bar p(z)_{ ab } = \E \left\{ \bar A(z)_{ ab } \,\vert\, \xi , z \right\}$.
Thus we relate each $ p_{ij} = \rho_n f\left( \xi_i , \xi_j \right)$ to the average $ \bar A(z)_{ z_i z_j } $ of the block to which it is assigned by $z$.

Continuing from~\eqref{eq:perm-choice}, we appeal to Lemma~\ref{diag-term} to bound the diagonal term, thereby obtaining
\begin{multline*}
\inf_{ \sigma \in \mathcal{M} }
\iint_{ \hat f \notin \{0,1\} } \rho_n^{-1} \D\left\{ \rho_n f\bigl( \sigma(x) , \sigma(y) \bigr) \,\middle\vert\middle\vert\, \rho_n \hat f\left( x , y ; h \right) \right\} \,dx \,dy
\\ \leq \frac{ 1 - \frac{ 1 }{ n } }{ \binom{n}{2} }
\!\!\!\! \!\! \sum_{ i < j : \bar A(z)_{ z_i z_j } \notin \{0,1\} } \hskip-0.9cm \rho_n^{-1} \D\left\{ p_{ij} \,\middle\vert\middle\vert\, \bar A(z)_{ z_i z_j } \right\} 
+ {\cal O}_P \left( \left\{ \log \left( 1 / \rho_n \right) + \log \tbinom{ h_\vee }{ 2 } \right\} n^{ - \alpha / 2 } \right.\\
\left.+ \log \tbinom{h_\vee}{2} \left( \rho_n n \right)^{-1} \right) \! .
\end{multline*}
Lemma~\ref{denom} yields the denominator of~\eqref{order}, and the result follows by taking the ratio of these terms.
\end{proof}

\begin{lemma}\label{f-small-domain}
Assume the setting of Theorem~\ref{graphconst}. Then
for $1 \leq i, j \leq n, \quad (a,b) : \bar A(z)_{ a b } \notin \{ 0 , 1 \} $
\begin{multline}
\label{eq:div-taylor}
n^2 \int_{\frac{j-1}{n}}^{\frac{j}{n}}
\int_{\frac{i-1}{n}}^{\frac{i}{n}}
\rho_n^{-1} \D\left\{ \rho_n f\left( x , y \right) \,\middle\vert\middle\vert\, \bar A(z)_{ a b } \right\} \,dx \,dy;
\\ = \rho_n^{-1} \D\left\{ \rho_n f\left( \xi_{(i)} , \xi_{(j)} \right) \,\middle\vert\middle\vert\, \bar A(z)_{ a b } \right\} + {\cal O}_P \left( \left\{ \log \left( 1 / \rho_n \right) + \log \tbinom{ h_\vee }{ 2 } \right\} n^{-\alpha / 2} \right) .
\end{multline}
\end{lemma}

\begin{proof}
The result follows from a Taylor series of the integrand of~\eqref{eq:div-taylor}, which we will show to converge everywhere on the domain of integration, as long as $ \bar A(z)_{ a b } \notin \{ 0 , 1 \} $. We begin by noting that whenever $ f \in \operatorname{\textrm{H\"older}}^\alpha(M) $, we have from Lemma~\ref{ftilde} that for all $ (x,y)\in \left( \frac{i-1}{n},\frac{i}{n} \right) \times \left( \frac{j-1}{n},\frac{j}{n} \right) $,
\begin{align*}
\E \left| f\left( x , y \right) - f\left( \xi_{(i)} , \xi_{(j)} \right) \right|
& \leq \E \left| f\left( x , y \right) - f\bigl( \tfrac{i}{n+1} , \tfrac{j}{n+1} \bigr) \right|
\\ &\qquad + \E \left| f\bigl( \tfrac{i}{n+1} , \tfrac{j}{n+1} \bigr) - f\left( \xi_{(i)} , \xi_{(j)} \right) \right|
\\ & \leq M \, \bigl\{ 2^{-1/2} ( n + 1 ) \bigr\}^{ - \alpha } + M \left\{ 2(n+2) \right\}^{ - \alpha / 2 } .
\end{align*}
From Markov's inequality, $ f\left( \xi_{(i)} , \xi_{(j)} \right) = f\left( x , y \right) + {\cal O}_P \bigl( n^{ - \alpha / 2 } \bigr) $ for every fixed $(x,y)$ in the domain of interest. Thus the following Taylor series holds whenever $ f \in \operatorname{\textrm{H\"older}}^\alpha(M) $ and $ \bar A(z)_{ a b } \notin \{ 0 , 1 \} $:
\begin{multline}
\label{eq:div-taylor-pw}
\rho_n^{-1} \D\left\{ \rho_n f\left( \xi_{(i)} , \xi_{(j)} \right) \,\middle\vert\middle\vert\, \bar A(z)_{ a b } \right\} 
= \rho_n^{-1} \D\left\{ \rho_n f\left( x , y \right) \,\middle\vert\middle\vert\, \bar A(z)_{ a b } \right\}
\\ + \log \left\{ \frac{ \rho_n f\left( x , y \right) }{ 1 - \rho_n f\left( x , y \right) } \cdot \frac{ 1 - \bar A(z)_{ a b } }{ \bar A(z)_{ a b } } \right\} \left\{ f\left( \xi_{(i)} , \xi_{(j)} \right) - f\left( x , y \right) \right\} + o_P \bigl( n^{ - \alpha / 2 } \bigr) .
\end{multline}
To bound the second term in~\eqref{eq:div-taylor-pw}, let $ l = \inf_{x\in(0,1)} f(x,x) $ and $ u = \sup_{x\in(0,1)} f(x,x) $. Since $ \bar A(z)_{ aa } \notin \{0,1\} $, we may bound the magnitudes of $ \log \bar A(z)_{ aa } , \log \left\{ 1 - \bar A(z)_{ a a } \right\}$ via $ \log \binom{h_a}{2} \leq \log \binom{ h_\vee }{ 2 } $. Then
\begin{multline}
\label{eq:Taylor-bnd-log}
\E \left| \log \left\{ \frac{ \rho_n f\left( x , y \right) }{ 1 - \rho_n f\left( x , y \right) } \cdot \frac{ 1 - \bar A(z)_{ a b } }{ \bar A(z)_{ a b } } \right\} \right|
\leq \log \left\{ ( \rho_n l )^{-1} \right\}  \\
+ \log \left\{ (1-\rho_n u )^{-1} \right\}+ 2 \log \tbinom{ h_\vee }{ 2 } .
\end{multline}
The first two terms in~\eqref{eq:Taylor-bnd-log} are bounded by hypothesis, and then we apply Markov's inequality to~\eqref{eq:div-taylor-pw}.
\end{proof}

\begin{lemma}\label{diag-term}
Assume the setting of Theorem~\ref{graphconst}. Then
\begin{multline}
\label{eq:diag-term-sum}
n^{-2 } \sum_{ 1 \leq i \leq n : \bar A(z)_{ z_i z_i } \notin \{0,1\} } \hskip-0.8cm \rho_n^{-1} \D\left\{ \rho_n f\left( \xi_i , \xi_i \right) \,\middle\vert\middle\vert\, \bar A(z)_{ z_i z_i } \right\}\\
= {\cal O}_P\left( \left\{ \log\left( 1 / \rho_n \right) + \rho_n^{-1} \log \tbinom{h_\vee}{2} \right\} n^{-1} \right).
\end{multline}
\end{lemma}

\begin{proof}
Let $ l = \inf_{x\in(0,1)} f(x,x) $ and $ u = \sup_{x\in(0,1)} f(x,x) $. Since $ \bar A(z)_{ aa } \notin \{0,1\} $, we may bound the magnitudes of $ \log \bar A(z)_{ aa } $ and $ \log \left\{ 1 - \bar A(z)_{ a a } \right\}$ via $ \log \binom{h_a}{2} \leq \log \binom{ h_\vee }{ 2 } $. We bound the expectation of each summand in~\eqref{eq:diag-term-sum} for $ 1 \leq i \leq n $
\begin{multline*}
\!\!\!\!\!\!\E \left[ f\left( \xi_i , \xi_i \right) \log\left\{ \frac{ \rho_n f\left( \xi_i , \xi_i \right) }{ \bar A(z)_{ z_i z_i } } \right\} + \rho_n^{-1} \left\{ 1-\rho_n f\left( \xi_i , \xi_i \right) \right\} \log\left\{ \frac{ 1-\rho_n f\left( \xi_i , \xi_i \right) }{ 1-\bar A(z)_{ z_i z_i } } \right\} \right] , 
\\ \le u \left\{ \log ( \rho_n l )^{-1} + \log \tbinom{ h_\vee }{ 2 } \right\}
+ \rho_n^{-1} \left[ \log \left\{ (1-\rho_n u )^{-1} \right\}
+ \log \tbinom{ h_\vee }{ 2 } \right]\\
= {\cal O}\left( \log\left( 1 / \rho_n \right) + \rho_n^{-1} \log h_\vee^2 \right).
\end{multline*}
The result then follows from linearity of expectation and Markov's inequality, as per Lemma~\ref{f-small-domain}.
\end{proof}

\begin{lemma}\label{denom}
Assume the setting of Theorem~\ref{graphconst}. Then
\begin{equation*}
\iint_{ \hat f \notin \{0,1\} } f\left( x , y \right) \,dx \,dy = \frac{ 1 - \frac{ 1 }{ n } }{ \rho_n \binom{n}{2} } \sum_{ i < j: \bar A(z)_{z_i z_j } \notin \{0,1\} } p_{ij} + {\cal O}_P\bigl( n^{ - \alpha / 2 } \bigr).
\end{equation*}
\end{lemma}
\begin{proof}
We start by discretizing the integral. We therefore write that
\begin{multline*}
\iint_{ \hat f \notin \{0,1\} } f\left( x , y \right) \,dx \,dy
= \sum_{ a, b : \bar A(z)_{ a b } \notin \{0,1\} } \sum_{ j = n H(b-1)+1}^{ n H(b) } \sum_{ i = n H(a-1)+1}^{ n H(a) } \\
\cdot \int_{\frac{j-1}{n}}^{\frac{j}{n}}\int_{\frac{i-1}{n}}^{\frac{i}{n}} 
f\left(x,y\right)\,dx\,dy \! = \!\!\!\! \!\!\! \sum_{ \bar A(z)_{z_i z_j } \notin \{0,1\} } 
\!\! \frac{ p_{ij} }{ \rho_n n^2 } + \!\!\!\! \!\!\!\! \sum_{ \bar A(z)_{ a b } \notin \{0,1\} } \sum_{ j = n H(b-1)+1}^{ n H(b) } 
\sum_{ i = n H(a-1)+1}^{ n H(a) } \\
\cdot \int_{\frac{j-1}{n}}^{\frac{j}{n}}\int_{\frac{i-1}{n}}^{\frac{i}{n}} 
\left\{ f\left(x,y\right)-f\left(\xi_{(i)},\xi_{(j)}\right) \right\} \,dx \,dy ,
\end{multline*}
where the latter term may be bounded using the technique of Lemma~\ref{f-small-domain}, yielding
\begin{equation}
\label{squeeze1}
\textstyle
\left| \iint_{ \hat f \notin \{0,1\} } f\left( x , y \right) \,dx \,dy - \sum_{ i , j: \bar A(z)_{z_i z_j } \notin \{0,1\} } \frac{ p_{ij} }{ \rho_n n^2 } \right|
= {\cal O}_P\bigl( n^{ - \alpha / 2 } \bigr) .
\end{equation}
Note $\sum_{ i , j: \bar A(z)_{z_i z_j } \notin \{0,1\} } p_{ij}
= 2 \sum_{ i < j: \bar A(z)_{z_i z_j } \notin \{0,1\} } p_{ij}
+ \sum_{ 1 \leq i \leq n : \bar A(z)_{z_i z_j } \notin \{0,1\} } p_{ii} $, so that
\begin{equation*}
\textstyle
\E \rho_n^{-1} n^{-2} \sum_{ 1 \leq i \leq n : \bar A(z)_{z_i z_j } \notin \{0,1\} } p_{ii}
\leq n^{-2} \sum_{ i = 1 }^n \E f\left( \xi_i , \xi_i \right) = {\cal O}\bigl( n^{-1} \bigr).
\end{equation*}
Applying Markov's theorem and combining the result with~\eqref{squeeze1} then yields the stated result.
\end{proof}

\newcommand{\IRGtheorem}{\ref{excess}}
\section{Proof of Theorem~\protect\IRGtheorem{} and lemmas}
\label{sec:excessProof}

\subsection{Proof of Theorem~\protect\IRGtheorem{}}

The proof is divided into four steps, with each the subject of a technical lemma proved in Section~\ref{sec:thm-excess-aux}. 

Lemma~\ref{lem:KLConv} yields the key first step, which is to relate $ \D\left( p_{ij} \,\middle\vert\middle\vert\, \bar A_{ z_i z_j } \right) $ to $ \D\left( p_{ij} \,\middle\vert\middle\vert\, \bar p_{ z_i z_j } \right) $ for any $z \in \mathcal{Z}_k$, assuming that $ \bar A_{ z_i z_j } \notin \{0,1\} $. This ensures that both terms are finite, and hence comparable. To obtain sufficient variance reduction in this setting, every $ \bar A_{ z_i z_j } $ must concentrate to its mean $ \bar p_{ z_i z_j } $, in that the ratio of mean to standard deviation must shrink. The minimum effective block sample size $ \binom{ h_\wedge }{ 2 } \rho_\wedge $ must grow quickly enough that this takes place, even for the sparsest of all possible fitted blocks.

\begin{lemma}\label{lem:KLConv}
Assume conditions~\ref{cond:rhoAvg}--\ref{cond:nab} of Theorem~\ref{excess}, and that $ \binom{ h_\wedge }{ 2 } \rho_\wedge = \omega \bigl( \log \binom{ h_\wedge }{ 2 } \bigr) $. Then
\begin{multline*}
0 \leq \!\! \textstyle
\sum_{ i < j : \bar A_{ z_i z_j } \notin \{0,1\} } \left\{ \D\left( p_{ij} \,\middle\vert\middle\vert\, \bar A_{ z_i z_j } \right) - \D\left( p_{ij} \,\middle\vert\middle\vert\, \bar p_{ z_i z_j } \right) \right\}
\\ = \mathcal{O}_P\left( \frac{ 2 \log \left| \mathcal{Z}_k \right| + \binom{k+1}{2} }{ \binom{n}{2} \bar \rho } \sum_{i<j} p_{ij} \right) \! , \, \forall z \in \mathcal{Z}_k .
\end{multline*}
\end{lemma}

Our next step relies on controlling $ \Pr ( \bar A_{ z_i z_j } \!\in\! \{0,1\} ) $ uniformly in $z$, via Lemma~\ref{lem:normalization}.

\begin{lemma}\label{lem:normalization}
Assume conditions~\ref{cond:rhoAvg}--\ref{cond:nab} of Theorem~\ref{excess}. Then 
\begin{equation*}
\textstyle \sum_{ i < j } \I\left( \bar A_{ z_i z_j } \in \{0,1\} \right) 
= \mathcal{O}_P \left( e^{ - \binom{ h_\wedge }{ 2 } \rho_\wedge +\log ( 1/ \bar \rho ) } \sum_{i<j} p_{ij} \right) , \,\, \forall z \in \mathcal{Z}_k .
\end{equation*}
\end{lemma}

This result shows that the set of all $ \bar A_{ z_i z_j } \in \{0,1\} $ has vanishing relative cardinality relative to $ \sum_{i<j} p_{ij} $, no matter which $z \in \mathcal{Z}_k$ is chosen. It is a direct consequence of condition~\ref{cond:nab} of Theorem~\ref{excess}, which ensures that the minimum fitted block size is uniformly lower-bounded by $h_\wedge = \omega(1)$.
 
Lemma~\ref{lem:normalization} has two immediate consequences. First, we may apply it to conclude that
\begin{equation}
\label{eq:thm-denom}
\frac{ \sum_{ i < j : \bar A_{ z_i z_j } \notin \{0,1\} } p_{ij} }{ \sum_{ i<j } p_{ij} } 
= 1 + \mathcal{O}_P \left( e^{ - \binom{ h_\wedge }{ 2 } \rho_\wedge +\log ( 1/ \bar \rho ) } \right) , \,\, \forall z \in \mathcal{Z}_k .
\end{equation}
Second, it enables us to substitute for the term $ \sum_{ i < j : \bar A_{ z_i z_j } \notin \{0,1\} } \D\left( p_{ij} \,\middle\vert\middle\vert\, \bar p_{ z_i z_j } \right) $ in Lemma~\ref{lem:KLConv} as follows.

\begin{lemma}\label{lem:KLdiff}
Assume conditions~\ref{cond:rhoAvg}--\ref{cond:nab} of Theorem~\ref{excess}. Then uniformly for all $ z \in \mathcal{Z}_k $,
\begin{multline*}
0 \leq \textstyle
\sum_{ i<j } \D\left( p_{ij} \,\middle\vert\middle\vert\, \bar p_{ z_i z_j } \right) 
- \sum_{ i < j : \bar A_{ z_i z_j } \notin \{0,1\} } \D\left( p_{ij} \,\middle\vert\middle\vert\, \bar p_{ z_i z_j } \right) 
\\= \textstyle \mathcal{O}_P \left( e^{ - \binom{ h_\wedge }{ 2 } \rho_\wedge +\log ( 1/ \bar \rho ) } \sum_{i<j} p_{ij} \right) .
\end{multline*}
\end{lemma}

Thus whenever all of the above quantities are $o_P(1)$, we may combine Lemmas~\ref{lem:KLConv} and~\ref{lem:KLdiff} with~\eqref{eq:thm-denom} to obtain our first claimed result: for any choice of $ z \in \mathcal{Z}_k $, deterministic or random, we have that
\begin{multline}
\label{eq:norm-risk}
\!\!\!\!\! \frac{ \sum_{ i < j : \bar A_{ z_i z_j } \notin \{0,1\} } \! \D\left( p_{ij} \! \,\middle\vert\middle\vert\, \! \bar A_{ z_i z_j } \! \right) \! - \! \sum_{ i<j } \! \D\left( p_{ij} \! \,\middle\vert\middle\vert\, \bar p_{ z_i z_j } \! \right) }{ \sum_{ i < j : \bar A_{ z_i z_j } \notin \{0,1\} } p_{ij} } \! \\ = \! \mathcal{O}_P\!\left( \tfrac{ 2 \log \left| \mathcal{Z}_k \right| + \binom{k+1}{2} }{ \binom{n}{2} \bar \rho } \! + \! e^{ - \binom{ h_\wedge }{ 2 } \rho_\wedge + \log ( 1/ \bar \rho ) } \right) \!\!\!
\end{multline}
whenever conditions~\ref{cond:rhoAvg}--\ref{cond:nab} of Theorem~\ref{excess} hold, $ \binom{ h_\wedge }{ 2 } \rho_\wedge = \omega \bigl( \log \binom{ h_\wedge }{ 2 } \bigr) $ and the argument of the right-hand side of~\eqref{eq:norm-risk} is $o_P(1)$. Under these conditions, the numerator term of~\eqref{eq:norm-risk}, when scaled by $ \sum_{ i < j } p_{ij} $, converges in probability to $0$ and hence in law, whereas~\eqref{eq:thm-denom} converges in probability to a non-zero constant. Thus by Slutsky's theorem, their ratio converges in law, and hence also in probability as per~\eqref{eq:norm-risk}. Separating terms on the left-hand side of~\eqref{eq:norm-risk}, and then multiplying the latter numerator term by $ \sum_{ i<j } p_{ij} / \sum_{ i<j } p_{ij} $, we obtain the first result of result of Theorem~\ref{excess}, as stated in~\eqref{eq:norm-risk-thm}.

We now establish sufficient conditions for~\eqref{eq:norm-risk}. We see immediately that $ \binom{ h_\wedge }{ 2 } \rho_\wedge = \omega \bigl( \log ( 1 / \bar \rho ) \bigr) $ must hold. Since Lemma~\ref{lem:KLConv} requires that $ \binom{ h_\wedge }{ 2 } \rho_\wedge = \omega \bigl( \log \binom{ h_\wedge }{ 2 } \bigr) $, we obtain the combined requirement
\begin{equation}
\label{eq:h-cond}
h_\wedge^2 \rho_\wedge = \omega \bigl( \max \left\{ \log h_\wedge^2 , \log ( 1 / \bar \rho ) \right\} \bigr) \Leftarrow h_\wedge^2 \rho_\wedge = \omega \bigl( \log n \bigr) .
\end{equation}
To see that this condition will be satisfied if the effective sample size of every possible fitted block is $ \omega \bigl( \log n \bigr) $, first note that $ h_\wedge \leq n $, and so $ \log h_\wedge^2 = \mathcal{O} \bigl( \log n \bigr) $. Now observe that because $ \rho_\wedge \leq \bar \rho$, it follows that $ h_\wedge^2 \rho_\wedge = \omega \bigl( \log h_\wedge^2 \bigr) $ implies $ h_\wedge^2 \bar \rho = \omega \bigl( \log h_\wedge^2 \bigr) $, or equivalently, $ \log ( 1 / \bar \rho ) = o \bigl( \log( h_\wedge^2 / \log h_\wedge^2 ) \bigr) $. Since $ h_\wedge \leq n $, this in turn implies $ \log ( 1 / \bar \rho ) = o \bigl( \log n \bigr)$. Thus $ h_\wedge^2 \rho_\wedge = \omega \bigl( \log n \bigr) $ implies~\eqref{eq:h-cond} as claimed.

To achieve convergence in probability,~\eqref{eq:norm-risk} also requires $ n^2 \bar \rho = \omega \bigl( \log \left| \mathcal{Z}_k \right| + \binom{k+1}{2} \bigr) $. To simplify this requirement and obtain a sufficient condition, observe that $ \log \left| \mathcal{Z}_k \right| \leq n \log k $, since $ \mathcal{Z}_k \subseteq \{ 1, \ldots, k \}^n $. Now write $ \binom{ k + 1 }{ 2 } = k^2 \left\{ 1/2 + \mathcal{O}(1) \right\} $, and let $ \bar h = n / k $. From these simplifications we obtain $ \bar \rho = \omega \bigl( \log( n / \bar h ) /n + \bar h^{-2} \bigr) $, which is implied by  $ \bar h^2 \bar \rho = \omega \bigl( \max \left\{ \bar h^2 / n , 1 \right\} \log n \bigr) $.

Finally, observe that since the results above hold uniformly over all $ z \in \mathcal{Z}_k $, they also hold for $ z = \hat z(A,\mathcal{Z}_k) $, the maximum profile likelihood estimator of $z$. The following lemma relates this choice to its oracle counterpart $\bar z(p,\mathcal{Z}_k) $---the best choice of $ z \in \mathcal{Z}_k $---enabling us to strengthen~\eqref{eq:norm-risk}.

\begin{lemma}\label{lem:uniformCons}
Assume conditions~\ref{cond:rhoAvg} and~\ref{cond:rhoMin} of Theorem~\ref{excess}. Then it follows from the arguments of Theorems~2 and~3 of \citet*{choi2012stochastic} that for any $ \hat z(A,\mathcal{Z}_k)$ and $\bar z(p,\mathcal{Z}_k) $ as per~\eqref{eq:MPLE} and~\eqref{eq:MPLEoracle},
\begin{multline*}
0 \leq \textstyle
\sum_{ i<j } \left\{ \D\left( p_{ij} \,\middle\vert\middle\vert\, \bar p_{ \hat z_i \hat z_j } \right) 
- \D\left( p_{ij} \,\middle\vert\middle\vert\, \bar p_{ \bar z_i \bar z_j } \right) \right\}
\! = \! \mathcal{O}_P \! \! \left( 
\frac{ \log \left| \mathcal{Z}_k \right| \! + \! \binom{k+1}{2} \! \log \! \left\{ \! \binom{n}{2} / \binom{k+1}{2} \! + \! 1 \! \right\} }{ \binom{n}{2} \bar \rho }  \right)\\
\textstyle
+\mathcal{O}_P \! \! \left( \frac{ \log \left( 1 / \rho_\wedge \right) \log \left| \mathcal{Z}_k \right| }{ 3 \binom{n}{2} \bar \rho } \! \left( \! 1 \! + \! \sqrt{ 1 \! + \! \frac{ 18 \binom{n}{2} \bar \rho }{ \log \left| \mathcal{Z}_k \right| } } \right) \!  \sum_{i<j} p_{ij} \! \right) \! .
\end{multline*}
\end{lemma}

Since $ \bar z(p,\mathcal{Z}_k) $ results in the minimum value of $\sum_{ i < j } \D\left( p_{ij} \,\middle\vert\middle\vert\, \bar p_{ z_i z_j } \right)$, this difference is nonnegative. Its convergence in probability to $0$ when suitably normalized is due to the maximizing properties of $ \hat z(A,\mathcal{Z}_k) $ and $ \bar z(p,\mathcal{Z}_k) $. Thus we conclude that $ \hat z(A,\mathcal{Z}_k) $ serves as an empirical proxy for $ \bar z(p,\mathcal{Z}_k) $.

To complete the proof, set $ z = \hat z(A,\mathcal{Z}_k) $ in~\eqref{eq:norm-risk} and combine it with Lemma~\ref{lem:uniformCons}. Comparing terms, we see that the latter's will dominate the rate of convergence, and so we upper-bound them using $ \bar h = n / k = \omega(1) $, subadditivity of the square root and the fact that $ \binom{n}{2} / \binom{k+1}{2} \leq \bar h^2 $. We thus obtain
\begin{align}
\nonumber
& \frac{ \sum_{ i < j : \bar A_{ \hat z_i \hat z_j } \notin \{0,1\} } \D\left( p_{ij} \,\middle\vert\middle\vert\, \bar A_{ \hat z_i \hat z_j } \right) - \min_{ z \in \mathcal{Z}_k } \sum_{ i<j } \D\left( p_{ij} \,\middle\vert\middle\vert\, \bar p_{ z_i z_j } \right) }{ \sum_{ i<j } p_{ij} } 
\\ \nonumber & =
\textstyle \mathcal{O}_P \left( \max \left[ \frac{ n \log \left( \frac{ n }{ \bar h } \right) \! + \! \binom{ \frac{ n }{ \bar h }+1}{2} \log \left\{ \bar h^2 \left( 1 \! + \! \bar h^{-2} \right) \right\} }{ \binom{n}{2} \bar \rho } , \sqrt{ \frac{ 2 \log ( \rho_\wedge^{-1} )^2 \, n \log \left( \frac{ n }{ \bar h }\right) }{ \binom{n}{2} \bar \rho } } \! \left( \! 1 \! + \! \sqrt{ \frac{ n \log \left( \frac{ n }{ \bar h } \right) }{ \frac{ 9 }{ 2 } \binom{n}{2} \bar \rho } } \right) \! \right] \right)
\\ \textstyle \nonumber & = \textstyle \mathcal{O}_P \left( \max \left[ \frac{ n \log \left( \frac{ n }{ \bar h } \right) + \frac{ n^2 }{ \bar h^2 } \left( 1 + \frac{ \bar h }{ n } \right) \log \left[ \bar h \left\{ 1 + o(1) \right\} \right] }{ n^2 \bar \rho \left\{ 1 + o(1) \right\} } , \sqrt{ \frac{ \log \left( 1 / \rho_\wedge \right)^2 \log \left( n / \bar h \right) }{ n \bar \rho \left\{ 1 + o(1) \right\} } } \left\{ 1 + o(1) \right\} \right] \right)
\\ \textstyle \nonumber & = \textstyle \mathcal{O}_P \left( \max \left[ \frac{ \bar h^{-2} \left( 1 + \bar h / n \right) \log \bar h }{ \bar \rho } , \max \left\{ \frac{ \log \left( n / \bar h \right) }{ n \bar \rho } , \sqrt{ \frac{ \log \left( 1 / \rho_\wedge \right)^2 \log \left( n / \bar h \right) }{ n \bar \rho } } \right\} \right] \right)
\\ \nonumber & = \textstyle \mathcal{O}_P \left( \max \left[ \frac{ \bar h^{-2} \left\{ 1 + \mathcal{O}(1) \right\} \log \bar h }{ \bar \rho } , \sqrt{ \frac{ \log \left( n / \bar h \right) }{ n \bar \rho } } \max \left\{ \sqrt{ \frac{ \log \left( n / \bar h \right) }{ n \bar \rho } } , \log \left( 1 / \rho_\wedge \right) \right\} \right] \right)
\\ \label{eq:conv-term} & = \textstyle \mathcal{O}_P \left( \max \left[ \frac{ \log \bar h }{ \bar h^2 \bar \rho } , \sqrt{ \frac{ \log \left( 1 / \rho_\wedge \right)^2 \log \left( n / \bar h \right) }{ n \bar \rho } } \right] \right) ,
\end{align}
where the final line follows because $ \log( n / \bar h ) = o( n \bar \rho) $ is needed for~\eqref{eq:conv-term} to be $ o_P(1) $, whereas $ \rho_\wedge \leq \rho < 1/2 $ implies that $ \log( 1 / \rho_\wedge )^2 > \log( 2 )^2 = \omega\bigl( \, \log( n / \bar h ) / ( n \bar \rho) \, \bigr)$. Thus we have derived the claimed rate of convergence, with a sufficient condition being that $ \bar h^2 \bar \rho = \omega \bigl( \max \left\{ \bar h^2 / n , 1 \right\} \log^3 n \bigr) $, since together $ \bar h^2 \bar \rho = \omega\bigl( \log n \bigr) $ and $ \rho = \omega\bigl( \log(n)^3 / n \bigr) $ imply that~\eqref{eq:conv-term} is $o_P(1)$. 

To complete the proof of Theorem~\ref{excess}, we now re-interpret the above results under the scaled exchangeable random graph model of~\eqref{pij-model}. Lemmas~\ref{lem:KLConv}--\ref{lem:uniformCons} then hold for every realized value of $\xi$, and thus the implicit conditioning on $\xi$ inherent to these results can be removed. Specifically, in Lemmas~\ref{lem:KLConv} and~\ref{lem:uniformCons}, we may marginalize~\eqref{eq:KLconcBnd} and~\eqref{eq:concIneq} respectively via the law of total probability, noting that their right-hand sides do not depend on $\xi$. For Lemmas~\ref{lem:normalization} and~\ref{lem:KLdiff}, we simply note that the bound of~\eqref{eq:AGMIneq} holds for all $\xi$.

\subsection{Proofs and auxiliary lemmas needed for Theorem~\protect\IRGtheorem{}}
\label{sec:thm-excess-aux}

\begin{proof}[Lemma~\ref{lem:KLConv}]
We write
\begin{align}
\nonumber 
&\sum_{ i < j : \bar A_{ z_i z_j } \notin \{0,1\} } 
 \left\{ \D\left( p_{ij} \,\middle\vert\middle\vert\, \bar A_{ z_i z_j } \right) - \D\left( p_{ij} \,\middle\vert\middle\vert\, \bar p_{ z_i z_j } \right) \right\}
\\ \nonumber & \quad= \sum_{ i < j : \bar A_{ z_i z_j } \notin \{0,1\} } \left\{ p_{ij} \log \left( \frac{ \bar p_{ z_i z_j } }{ \bar A_{ z_i z_j } } \right) + \left( 1 - p_{ij} \right) \log \left( \frac{ 1 - \bar p_{ z_i z_j } }{ 1 - \bar A_{ z_i z_j } } \right)
\right\} 
\\ \nonumber & \quad= \sum_{ a \leq b : \bar A_{ ab } \notin \{0,1\} } \mathop{\sum_{\,\,i \in z^{-1}(a),}}_{j \in z^{-1}(b)} \left\{ p_{ij} \log \left( \frac{ \bar p_{ z_i z_j } }{ \bar A_{ z_i z_j } } \right) + \left( 1 - p_{ij} \right) \log \left( \frac{ 1 - \bar p_{ z_i z_j } }{ 1 - \bar A_{ z_i z_j } } \right) \right\} 
\\ \nonumber &\quad = \sum_{ a \leq b : \bar A_{ ab } \notin \{0,1\} } \Biggl\{ \log \left( \frac{ \bar p_{ ab } }{ \bar A_{ ab } } \right) \mathop{\sum_{\,\,i \in z^{-1}(a),}}_{j \in z^{-1}(b)} p_{ij} + \log \left( \frac{ 1 - \bar p_{ ab } }{ 1 - \bar A_{ ab } } \right) \mathop{\sum_{\,\,i \in z^{-1}(a),}}_{j \in z^{-1}(b)} \left( 1 - p_{ij} \right) \Biggr\} 
\\ \nonumber & \quad= \sum_{ a \leq b : \bar A_{ ab } \notin \{0,1\} } \left\{ \log \left( \frac{ \bar p_{ ab } }{ \bar A_{ ab } } \right) h_{ab}^2 \bar p_{ ab } + \log \left( \frac{ 1 - \bar p_{ ab } }{ 1 - \bar A_{ ab } } \right) h_{ab}^2 \left( 1 - \bar p_{ ab } \right) \right\}
\\ \label{eq:sumKLs} & \quad= \sum_{ a \leq b : \bar A_{ ab } \notin \{0,1\} } h_{ab}^2 \D\left( \bar p_{ ab } \,\middle\vert\middle\vert\, \bar A_{ ab } \right).
\end{align}
Since~\eqref{eq:sumKLs} is a sum of Kullback--Leibler divergences, it is nonnegative. To show its convergence when suitably normalized, we appeal to Lemma~\ref{lem:divControl} below, which implies the following under conditions~\ref{cond:rhoAvg}--\ref{cond:nab} of Theorem~\ref{excess} and the hypothesis $ \binom{ h_\wedge }{ 2 } \rho_\wedge = \omega \bigl( \log \binom{ h_\wedge }{ 2 } \bigr) $: 

For every $\epsilon > 0$, eventually in $n$ and with $1^+/2$ approaching arbitrarily closely to $1/2$,
\begin{align}
\nonumber
& \textstyle\Pr \left( \max_{ z \in \mathcal{Z}_k } \sum_{ a \leq b : \bar A_{ ab } \notin \{0,1\} } h_{ab}^2 \D\left( \bar p_{ ab } \,\middle\vert\middle\vert\, \bar A_{ ab } \right) \geq \epsilon \sum_{i<j} p_{ij} \right)
\\ \nonumber & \textstyle \hskip3cm \leq \exp \left( \log \left| \mathcal{Z}_k \right| - \frac{ \left\{ \epsilon \sum_{i<j} p_{ij} - \frac{ 1^+ }{ 2 } \binom{k+1}{2} \right\}^2 }{ 2 \epsilon \sum_{i<j} p_{ij} + \frac{ 1^+ }{ 2 } \binom{k+1}{2} } \right) 
\\ \label{eq:KLconcIntermed} & \textstyle \hskip3cm \leq \exp \left( \log \left| \mathcal{Z}_k \right| - \frac{ \epsilon \sum_{i<j} p_{ij} \max \left\{ \epsilon \sum_{i<j} p_{ij} - 1^+ \binom{k+1}{2} , 0 \right\} }{ 2 \epsilon \sum_{i<j} p_{ij} + \frac{ 1^+ }{ 2 } \binom{k+1}{2} } \right) 
\\ \nonumber & \textstyle\hskip3cm \leq \exp \left( \log \left| \mathcal{Z}_k \right| - \frac{ \max \left\{ \epsilon \sum_{i<j} p_{ij} - 1^+ \binom{k+1}{2} , 0 \right\} }{ 2 + \frac{ 1^+ }{ 2 } \binom{k+1}{2} / \left( \epsilon \sum_{i<j} p_{ij} \right) } \right) 
\\ \textstyle\label{eq:KLconcBnd} & \textstyle\hskip3cm \leq \exp\left( \log \left| \mathcal{Z}_k \right| - \frac{ \max \left\{ \epsilon \binom{n}{2} \bar \rho - 1^+ \binom{k+1}{2} , 0 \right\} }{ 2 + \frac{ 1^+ }{ 2 } \binom{k+1}{2} / \left\{ \epsilon \binom{n}{2} \bar \rho \right\} } \right) ,
\end{align}
where~\eqref{eq:KLconcIntermed} follows as $ \epsilon \sum_{i<j} p_{ij} \geq 0$ and $ ( 1^+ / 2 ) \binom{k+1}{2} \geq 0$ eventually in $n$, and~\eqref{eq:KLconcBnd} follows from condition~\ref{cond:rhoAvg} of Theorem~\ref{excess}, by which $ \sum_{i<j} p_{ij}(n) \geq \binom{n}{2} \, \bar \rho(n) $ eventually in $n$.
\end{proof}

\begin{proof}[Lemma~\ref{lem:normalization}]
We will bound $ \Pr( \bar A_{ z_i z_j } \in \{0,1\} ) $ uniformly in $z$. Observe that for any $1 \leq a \leq b \leq k$, conditionally on any $z \in \mathcal{Z}_k$, we have by the arithmetic--geometric mean inequality that
\begin{align}
\nonumber
\Pr \left( \bar A_{ab} \in \{0,1\} \,\vert\, Z = z \right)
& = \Pr \left( \bar A_{ab} = 0 \,\vert\, Z = z \right) + \Pr \left( \bar A_{ab} = 1 \,\vert\, Z = z \right)
\\ \nonumber & = \mathop{\prod_{\,\,i \in z^{-1}(a),}}_{j \in z^{-1}(b)} \left( 1 - p_{ij} \right) + \mathop{\prod_{\,\,i \in z^{-1}(a),}}_{j \in z^{-1}(b)} p_{ij}
\\ \label{eq:AGMIneq} & \leq \left( 1 - \bar p(z)_{ab} \right)^{ h_{ab}^2 } + \left( \bar p(z)_{ab} \right)^{ h_{ab}^2 } .
\end{align}

Conditions~\ref{cond:rhoMin} and~\ref{cond:nab} of Theorem~\ref{excess} stipulate that for every pair $(a,b)$ and every $ z \in \mathcal{Z}_k$, eventually in $n$, $\rho_\wedge(n) \leq \bar p_{ab}(n) \leq 1 - \sqrt{ \rho_\wedge(n) } $ and $h_\wedge(n) \leq h_{a}(n)$. Hence~\eqref{eq:AGMIneq} implies that, eventually in $n$, for
$1 \leq a \leq b \leq k$
\begin{align}
\nonumber
&\Pr \left( \bar A_{ab} \in \{0,1\} \,\vert\, Z = z \right)
 \leq ( 1 - \rho_\wedge ) ^{ h_{ab}^2 } + ( 1 - \sqrt{ \rho_\wedge } ) ^{ h_{ab}^2 } ;
\\ \nonumber &\Rightarrow \max_{a \leq b} \, \Pr \left( \bar A_{ab} \in \{0,1\} \,\vert\, Z = z \right)
\leq 2 ( 1 - \rho_\wedge ) ^{ \binom{ h_\wedge }{ 2 } } ; 
\\ \label{prob1} &\Rightarrow \max_{ z \in \mathcal{Z}_k } \, \max_{ i < j } \, \Pr \left( \bar A_{z_i z_j} \in \{0,1\} \,\vert\, Z = z \right)
 \leq 2 ( 1 - \rho_\wedge ) ^{ \binom{ h_\wedge }{ 2 } } .
\end{align}

Since the conditional probability $ \Pr \big( \bar A_{z_i z_j} \in \{0,1\} \,\vert\, Z = z \big) $ is upper-bounded by~\eqref{prob1} uniformly for every value of $ z \in \mathcal{Z}_k $, this same bound also holds after marginalizing out $Z$. Thus, eventually in $n$,
\begin{equation}
\label{eq:badSetBnd}
\Pr\left( \bar A_{ z_i z_j } \in \{0,1\} \right) 
\leq 2 ( 1 - \rho_\wedge ) ^{ \binom{ h_\wedge }{ 2 } } .
\end{equation}

Applying Markov's inequality, we see that for any $\epsilon > 0$, eventually in $n$,
\begin{align*}
\Pr \left( \sum_{ i < j } \I\left( \bar A_{ z_i z_j } \in \{0,1\} \right) \geq \epsilon \sum_{i<j} p_{ij} \right) 
& \leq \frac{ \sum_{ i < j } \Pr\left( \bar A_{ z_i z_j } \in \{0,1\} \right) }{ \epsilon \sum_{i<j} p_{ij} } 
\\ & \leq \frac{ \binom{n}{2} 2 ( 1 - \rho_\wedge ) ^{ \binom{ h_\wedge }{ 2 } } }{ \epsilon \sum_{i<j} p_{ij} } 
\\ & \leq \frac{ 2 ( 1 - \rho_\wedge ) ^{ \binom{ h_\wedge }{ 2 } } }{ \epsilon \bar \rho } 
\\ & \leq \frac{ 2 \exp\left\{ - \binom{ h_\wedge }{ 2 } \rho_\wedge \right\} }{ \epsilon \bar \rho }
\\ & = \frac{ \exp\left\{ - \binom{ h_\wedge }{ 2 } \rho_\wedge + \log \left( 1 / \bar \rho \right) \right\} }{ ( \epsilon / 2 ) } ,
\end{align*}
where the second inequality follows directly from~\eqref{eq:badSetBnd}, the third inequality follows from condition~\ref{cond:rhoAvg} of Theorem~\ref{excess}, by which $ \sum_{i<j} p_{ij}(n) \geq \binom{n}{2} \, \bar \rho(n) $ eventually in $n$, and the final inequality follows from the fact that $ \log \bigl\{ ( 1 - \rho_\wedge ) ^{ \binom{ h_\wedge }{ 2 } } \bigr\} = \binom{ h_\wedge }{ 2 } \log ( 1 - \rho_\wedge ) \leq - \binom{ h_\wedge }{ 2 } \rho_\wedge$.
\end{proof}

\begin{proof}[Lemma~\ref{lem:KLdiff}]
First, we express the term of interest as a sum of nonnegative random variables:
\begin{equation*}
\sum_{ i<j } \D\left( p_{ij} \,\middle\vert\middle\vert\, \bar p_{ z_i z_j } \right) 
- \!\!\!\! \!\!\!\! \sum_{ i < j : \bar A_{ z_i z_j } \notin \{0,1\} } \!\!\!\! \!\!\!\!\!\! \D\left( p_{ij} \,\middle\vert\middle\vert\, \bar p_{ z_i z_j } \right) 
= \sum_{ i < j } \D\left( p_{ij} \,\middle\vert\middle\vert\, \bar p_{ z_i z_j } \right) \I( \bar A_{ z_i z_j } \in \{0,1\} ) .
\end{equation*}
To show the claimed convergence in probability, we write
\begin{align*}
0 &\leq \sum_{ i < j } \D\left( p_{ij} \!\! \,\middle\vert\middle\vert\, \bar p_{ z_i z_j } \right) \I( \bar A_{ z_i z_j } \in \{0,1\} )
\\ & = - \sum_{ i < j } \left\{ p_{ij} \log \left( \bar p_{ z_i z_j } \right) + ( 1 - p_{ij} ) \log \left( 1 - \bar p_{ z_i z_j } \right) \right\} \I( \bar A_{ z_i z_j } \in \{0,1\} )
\\ & \hskip1.5cm + \sum_{ i < j } \left\{ p_{ij} \log \left( p_{ij} \right) + ( 1 - p_{ij} ) \log \left( 1 - p_{ij} \right) \right\} \I( \bar A_{ z_i z_j } \in \{0,1\} )
\\ & \leq - \sum_{ i < j } \left\{ p_{ij} \log \left( \bar p_{ z_i z_j } \right) + ( 1 - p_{ij} ) \log \left( 1 - \bar p_{ z_i z_j } \right) \right\} \I( \bar A_{ z_i z_j } \in \{0,1\} )
\\ & = - \sum_{ a \leq b } \mathop{\sum_{\,\,i \in z^{-1}(a),}}_{j \in z^{-1}(b)} \left\{ p_{ij} \log \left( \bar p(z)_{ab} \right) + \left( 1 - p_{ij} \right) \log \left( 1 - \bar p( z)_{ab} \right) \right\} \I( \bar A_{ z_i z_j } \in \{0,1\} )
\\ & = - \sum_{ a \leq b } h_{ab}^2 \left\{ \bar p(z)_{ab} \log \left( \bar p(z)_{ab} \right) + \left( 1 - \bar p(z)_{ab} \right) \log \left( 1 - \bar p(z)_{ab} \right) \right\} \I( \bar A_{ab} \in \{0,1\} )
\\ & \leq \sum_{ a \leq b } h_{ab}^2 ( \log 2 ) \I( \bar A_{ab} \in \{0,1\} )
\\ & = ( \log 2 ) \sum_{ i < j } \I( \bar A_{ z_i z_j } \in \{0,1\} ) .
\end{align*}
The result then follows from Lemma~\ref{lem:normalization}, which establishes that for every $ z \in \mathcal{Z}_k $, we have $ \sum_{ i < j } \I( \bar A_{ z_i z_j } \!\in\! \{0,1\} ) = \mathcal{O}_P \bigl( e^{ - \binom{ h_\wedge }{ 2 } \rho_\wedge +\log ( 1/ \bar \rho ) } \sum_{i<j} p_{ij} \bigr) $ under conditions~\ref{cond:rhoAvg}--\ref{cond:nab} of Theorem~\ref{excess}.
\end{proof}

\begin{proof}[Lemma~\ref{lem:uniformCons}]
In the notation of \citet*{choi2012stochastic}, define for any fixed $ z \in \mathcal{Z}_k $
\begin{align*}
\bar L(z) & = \sum_{i<j} \left\{ p_{ij} \log \bar p_{ z_i z_j } + ( 1 - p_{ij} ) \log \left( 1 - \bar p_{ z_i z_j } \right) \right\} ;
\\ \Rightarrow \bar z(p,\mathcal{Z}_k) & = \operatornamewithlimits{argmax}_{ z \in \mathcal{Z}_k } \bar L(z)
= \operatornamewithlimits{argmin}_{ z \in \mathcal{Z}_k } \sum_{ i < j } \D\left( p_{ij} \,\middle\vert\middle\vert\, \bar p_{ z_i z_j } \right) .
\end{align*}
where the implication follows directly from the definition of the ``oracle'' MPLE in $ \bar z(p,\mathcal{Z}_k) $ in~\eqref{eq:MPLEoracle}. Thus
\begin{equation*}
0 \leq 
\sum_{ i<j } \left\{ \D\left( p_{ij} \,\middle\vert\middle\vert\, \bar p_{ \hat z_i \hat z_j } \right) 
- \D\left( p_{ij} \,\middle\vert\middle\vert\, \bar p_{ \bar z_i \bar z_j } \right) \right\}
= \bar L ( \bar z ) - \bar L ( \hat z ) , \quad \bar z, \hat z \in \mathcal{Z}_k .
\end{equation*}
By construction, since $ \bar z(p,\mathcal{Z}_k) $ maximizes $ \bar L (z) $ over $ \mathcal{Z}_k $, this difference is nonnegative. Similarly, from~\eqref{eq:MPLE} we see that $\hat z(A,\mathcal{Z}_k)$ maximizes $L(A; z)$ over $ \mathcal{Z}_k $, and so $ L(A; \hat z) - L(A; \bar z) \geq 0$. Hence,
\begin{align}
\nonumber
0 \leq 
\bar L ( \bar z ) - \bar L ( \hat z )
& \nonumber \leq \bar L ( \bar z ) - \bar L ( \hat z ) + \left\{ L(A; \hat z) - L(A; \bar z) \right\} , \quad \bar z, \hat z \in \mathcal{Z}_k
\\ \nonumber & = \bar L ( \bar z ) - L(A; \bar z) + L(A; \hat z) - \bar L ( \hat z ) 
\\ \label{eq:MPLETriangle} & \leq \left| \bar L ( \bar z ) - L(A; \bar z) \right| + \left| L(A; \hat z ) - \bar L ( \hat z ) \right| ,
\end{align}
and so the result will follow from~\eqref{eq:MPLETriangle} if we can show that $ \left| \bar L ( \bar z ) - L(A; \bar z) \right| $ and $ \left| L(A; \hat z ) - \bar L ( \hat z ) \right| $ both converge in probability to zero when suitably renormalized. We accomplish this in the manner of \citet*[Theorem~2]{choi2012stochastic}, who establish that $ \max_{ z \in \mathcal{Z}_k } \left| \bar L ( z ) - L(A; z) \right| / \sum_{i<j} p_{ij} $ converges as required. Since this result holds for the maximum over all $ z \in \mathcal{Z}_k$, then it must also hold for both $\hat{z}$ and $\bar{z}$, and we can therefore apply this same result twice.

In particular, Theorem~2 of \citet{choi2012stochastic} shows that for any fixed $n$, whenever $ \max_{ij} \left| \operatorname{logit} \bar p_{ z_i z_j } \right| $ is finite for all $ z \in \mathcal{Z}_k$, it holds that for all nonempty $ \mathcal{Z}_k \subseteq \{1,\ldots,k\}^n $ and any $\epsilon > 0$, 
\begin{multline}
\label{eq:concIneq}
\Pr \left( \max_{ z \in \mathcal{Z}_k } \left| L(A;z) - \bar L(z) \right| \geq 2 \epsilon \sum_{i<j} p_{ij} \right) \\
\leq \textstyle \left| \mathcal{Z}_k \right| \exp\left[ \binom{k+1}{2} \log \left\{ \binom{n}{2} / \binom{k+1}{2} + 1 \right\} - \epsilon \sum_{ i < j } p_{ij} \right]
\\ + \sum_{ z \in \mathcal{Z}_k } 2 \exp\left\{ - \frac{ \big( \epsilon \sum_{ i < j } p_{ij} \big)^2 / 2 }{ \sum_{ i<j } p_{ij} \left| \operatorname{logit} \bar p_{ z_i z_j } \right|^2 + (1/3) \big( \epsilon \sum_{ i < j } p_{ij} \big) \max_{i<j} \left| \operatorname{logit} \bar p_{ z_i z_j } \right| } \right\} .
\end{multline}
From condition~\ref{cond:rhoMin} of Theorem~\ref{excess}, we have that each $ p_{ij} (n) \in (0,1) $ eventually in $n$. This implies that $ \max_{ij} \left| \operatorname{logit} \bar p_{ z_i z_j }(n) \right| $ will eventually be finite for all $ z \in \mathcal{Z}_k$, and thus~\eqref{eq:concIneq} holds eventually in $n$.

To simplify the right-hand side of~\eqref{eq:concIneq}, we upper-bound $ \left| \operatorname{logit} \bar p_{ z_i z_j } \right| $ via $\max_{ i < j } \left| \operatorname{logit} \bar p_{ z_i z_j } \right| $, which allows a factor of $ \sum_{ i < j } p_{ij} $ to be canceled:
\begin{multline*}
\Pr \left( \max_{ z \in \mathcal{Z}_k } \left| L(A;z) - \bar L(z) \right| \geq 2 \epsilon \sum_{i<j} p_{ij} \right) 
\leq \textstyle \left| \mathcal{Z}_k \right| \\
\textstyle \exp\left[ \binom{k+1}{2} \log \left\{ \binom{n}{2} / \binom{k+1}{2} + 1 \right\} - \epsilon \sum_{ i < j } p_{ij} \right]
\\ + \sum_{ z \in \mathcal{Z}_k } 2 \exp\left\{ - \frac{ ( \epsilon^2 / 2 ) \sum_{ i < j } p_{ij} }{ \max_{i<j} \left| \operatorname{logit} \bar p_{ z_i z_j } \right|^2 + ( \epsilon / 3 ) \max_{i<j} \left| \operatorname{logit} \bar p_{ z_i z_j } \right| } \right\} .
\end{multline*}
Next, we upper-bound $ \max_{ i < j } \left| \operatorname{logit} \bar p_{ z_i z_j } \right| $ uniformly in $z$ via \\
$ \max_{z \in \mathcal{Z}_k} \left\{ \max_{ i < j } \left| \operatorname{logit} \bar p_{ z_i z_j } \right| \right\} $. This highlights the importance of bounding $p_{ij}$ away from $0$ and $1$. We may now sum over $ z \in \mathcal{Z}_k $ to obtain
\begin{multline*}
\Pr \left( \max_{ z \in \mathcal{Z}_k } \left| L(A;z) - \bar L(z) \right| \geq 2 \epsilon \sum_{i<j} p_{ij} \right) 
\leq \textstyle \left| \mathcal{Z}_k \right| \\
\exp\left[ \binom{k+1}{2} \log \left\{ \binom{n}{2} / \binom{k+1}{2} + 1 \right\} - \epsilon \sum_{ i < j } p_{ij} \right]
\\ \textstyle + 2 \left| \mathcal{Z}_k \right| \exp\left\{ - \frac{ ( \epsilon^2 / 2 ) \sum_{ i < j } p_{ij} }{ \left[ \max_{z \in \mathcal{Z}_k} \left\{ \max_{ i < j } \left| \operatorname{logit} \bar p_{ z_i z_j } \right| \right\} \right]^2 + ( \epsilon / 3 ) \max_{z \in \mathcal{Z}_k} \left\{ \max_{ i < j } \left| \operatorname{logit} \bar p_{ z_i z_j } \right| \right\} } \right\} .
\end{multline*}
Condition~\ref{cond:rhoMin} stipulates that every $ \bar p_{ z_i z_j } $ satisfies $\rho_\wedge(n) \leq \bar p_{ z_i z_j }(n) \leq 1 - \sqrt{ \rho_\wedge(n) } $ eventually in $n$, so
\begin{align*}
\max_{z \in \mathcal{Z}_k} &\left\{ \max_{ i < j } \left| \operatorname{logit} \bar p_{ z_i z_j }(n) \right| \right\} 
 = \max_{z \in \mathcal{Z}_k} \left\{ \max_{ i < j } \left| \log \left( \frac{ \bar p_{ z_i z_j }(n) }{ 1 - \bar p_{ z_i z_j }(n) } \right) \right| \right\}
\\ & = \max_{z \in \mathcal{Z}_k} \left[ \max_{ i < j } \left\{ \max \, \log \left( \frac{ \bar p_{ z_i z_j }(n) }{ 1 - \bar p_{ z_i z_j }(n) } \right), \log \left( \frac{ 1 - \bar p_{ z_i z_j }(n) }{ \bar p_{ z_i z_j }(n) } \right) \right\} \right] 
\\ & \leq \max_{z \in \mathcal{Z}_k} \left[ \max_{ i < j } \left\{ \max \, \log \left( \frac{ 1 - \sqrt{ \rho_\wedge(n) } }{ \sqrt{ \rho_\wedge(n) } } \right), \log \left( \frac{ 1 - \rho_\wedge(n) }{ \rho_\wedge(n) } \right) \right\} \right]
\\ & \leq \log \left\{ 1 / \rho_\wedge(n) \right\} ,
\end{align*}
which is finite, as condition~\ref{cond:rhoAvg} specifies that $0 < \rho_{\wedge}(n) < 1/2$ for all $n$.

Finally, condition~\ref{cond:rhoAvg} of Theorem~\ref{excess} ensures that $ \binom{n}{2} \, \bar \rho(n) \leq \sum_{i<j} p_{ij}(n)$ eventually in $n$. Thus, recalling~\eqref{eq:MPLETriangle}, we obtain the claimed result, since we have shown that for all $n$ sufficiently large,
\begin{align*}
& \Pr \left( \max_{ z \in \mathcal{Z}_k } \left| L(A;z) - \bar L(z) \right| \geq 2 \epsilon \sum_{i<j} p_{ij} \right) 
\\
&\qquad \leq \textstyle \exp\left[ \log \left| \mathcal{Z}_k \right| + \binom{k+1}{2} \log \left\{ \binom{n}{2} / \binom{k+1}{2} + 1 \right\} - \epsilon \tbinom{n}{2} \bar \rho \right] 
\\ & \hskip2cm + 2 \exp\left\{ \log \left| \mathcal{Z}_k \right| - \frac{ \binom{n}{2} \bar \rho }{ \log \left( 1 / \rho_\wedge \right)^2 } \left( \frac{ \epsilon^2 / 2 }{ 1 + ( \epsilon / 3) / \log \left( 1 / \rho_\wedge \right) } \right) \right\} 
\\ & \leq 4 \textstyle \exp\left\{ \log \left| \mathcal{Z}_k \right| \phantom{\frac{ - \binom{n}{2} \bar \rho }{ \log \left( 1 / \rho_\wedge \right)^2 } \left( \frac{ \epsilon^2 / 2 }{ 1 + ( \epsilon / 3) / \log \left( 1 / \rho_\wedge \right) } \right)} \right.
\\ & \left. + \max \left[ \textstyle \binom{k+1}{2} \log \left\{ \binom{n}{2} / \binom{k+1}{2} + 1 \right\} - \epsilon \tbinom{n}{2} \bar \rho, \frac{ - \binom{n}{2} \bar \rho }{ \log \left( 1 / \rho_\wedge \right)^2 } \left( \frac{ \epsilon^2 / 2 }{ 1 + ( \epsilon / 3) / \log \left( 1 / \rho_\wedge \right) } \right) \right] \right\}
.
\end{align*}
\begin{equation*}
\vspace{-2\baselineskip}
\end{equation*}
\end{proof}

\begin{lemma}\label{lem:divControl}
Assume conditions~\ref{cond:rhoAvg}--\ref{cond:nab} of Theorem~\ref{excess} and the hypothesis $ \binom{ h_\wedge }{ 2 } \rho_\wedge = \omega \bigl( \log \binom{ h_\wedge }{ 2 } \bigr) $, which together ensure that for every $ z \in \mathcal{Z}_k $,
\begin{equation}
\label{eq:essRate}
\frac{ \sqrt{ \log \left( h_{ab}^2 \right) / h_{ab}^2 } }{ \min \left( \bar p_{ab} , 1 - \bar p_{ab} \right) / \sqrt{ \bar p_{ab} } } = o(1) , \quad 1 \leq a \leq b \leq k .
\end{equation}

Then for every $\epsilon > 0$, we have eventually in $n$ that
\begin{equation*}
\Pr \left( \max_{ z \in \mathcal{Z}_k } \sum_{ a \leq b : \bar A_{ ab } \notin \{0,1\} } h_{ab}^2 \D\left( \bar p_{ ab } \,\middle\vert\middle\vert\, \bar A_{ ab } \right) \geq \epsilon \right)
\leq \exp \left( \log \left| \mathcal{Z}_k \right| - \frac{ \left\{ \epsilon - \frac{ 1^+ }{ 2 } \binom{k+1}{2} \right\}^2 }{ 2 \epsilon + \frac{ 1^+ }{ 2 } \binom{k+1}{2} } \right) ,
\end{equation*}
with $ 1^+ / 2 $ approaching arbitrarily closely to $1/2$ from above, at the rate given by~\eqref{eq:essRate}.
\end{lemma}

\begin{proof}
Observe that for any fixed $ z \in \mathcal{Z}_k $, we may re-express \\
$\sum_{ a \leq b : \bar A_{ ab } \notin \{0,1\} } h_{ab}^2 \D\left( \bar p_{ ab } \,\middle\vert\middle\vert\, \bar A_{ ab } \right)$ as a sum of the terms whose moments will be bounded by Lemma~\ref{PB-KL}:
\begin{equation*}
\sum_{ a \leq b : \bar A_{ ab } \notin \{0,1\} } h_{ab}^2 \D\left( \bar p_{ ab } \,\middle\vert\middle\vert\, \bar A_{ ab } \right) 
= \sum_{ a \leq b } g\left( h_{ab}^2 \bar A_{ ab } \right) , \quad z \in \mathcal{Z}_k \text{ fixed} .
\end{equation*}
Here, setting $ X_n = h_{ab}^2 \bar A_{ ab } $ in~\eqref{eq:fXn} of Lemma~\ref{PB-KL}, we define $ g\left( h_{ab}^2 \bar A_{ ab } \right) $ as
\begin{multline*}
g\left( h_{ab}^2 \bar A_{ ab } \right) 
\\= 
\begin{cases}
h_{ab}^2 \left\{ \bar p_{ab} \log \left( \frac{ \bar p_{ab} }{ \bar A_{ab} } \right) + ( 1 - \bar p_{ab} ) \log \left( \frac{ 1 - \bar p_{ab} }{ 1 - \bar A_{ab} } \right) \right\} & \text{if $ h_{ab}^2 \bar A_{ ab } \in \{1, \ldots, h_{ab}^2 - 1 \} $,}
\\ 0 & \text{if $ h_{ab}^2 \bar A_{ ab } \in \{0, h_{ab}^2 \} $.}
\end{cases}
\end{multline*}
By hypothesis, the conditions of Lemma~\ref{PB-KL} apply for all $ 1 \leq a \leq b \leq k $ and every $ z \in \mathcal{Z}_k $, and so each $g\left( h_{ab}^2 \bar A_{ ab } \right)$ behaves like a chi-square variate on $1$ degree of freedom in terms of its $m$th moment where $m = 1, 2, \ldots$ 
\begin{equation}
\label{eq:momentCondX}
\E \left\{ g\left( h_{ab}^2 \bar A_{ ab } \right)^m \right\} \leq \frac{ \Gamma\left( m + \frac{1}{2} \right) }{ \sqrt{\pi} } \left\{ 1 + \mathcal{O} \left( \frac{ \sqrt{ \log \left( h_{ab}^2 \right) / h_{ab}^2 } }{ \min \left( \bar p_{ab} , 1 - \bar p_{ab} \right) / \sqrt{ \bar p_{ab} } } \right) \right\} .
\end{equation}

Controlling the moments of $g\left( h_{ab}^2 \bar A_{ ab } \right)$ enables us to apply a Bernstein concentration inequality due to \citet[Lemma~8]{birge1998minimum}. To do so requires the existence of constants $v^2$ and $c$ such that
\begin{equation}
\label{eq:BernsMomentCondX}
\tbinom{k+1}{2}^{-1} \sum_{a \leq b} \E \left\{ g\left( h_{ab}^2 \bar A_{ ab } \right)^m \right\} 
\leq \frac{m!}{2} v^2 c^{m-2} , \quad m = 2, 3, \ldots .
\end{equation}
By hypothesis, 
\begin{align*}
\frac{ \Gamma\left( m + \frac{1}{2} \right) }{ \sqrt{\pi} } \left\{ 1 + \mathcal{O} \left( \frac{ \sqrt{ \log \left( h_{ab}^2 \right) / h_{ab}^2 } }{ \min \left( \bar p_{ab} , 1 - \bar p_{ab} \right) / \sqrt{ \bar p_{ab} } } \right) \right\} 
& = \frac{ \Gamma\left( m + \frac{1}{2} \right) }{ \sqrt{\pi} } \left\{ 1 + o(1) \right\}
\\ & < \frac{ 3 }{ 4} + \delta ,
\end{align*}
eventually in $n$, for every $ \delta > 0 $. Thus we fix $ v^2$ arbitrarily close to $ 3/4 $, and write $ v^2 = 3^+/4$. To ensure that~\eqref{eq:BernsMomentCondX} is satisfied for each $m$, we then let $ c = 1 $. 

We can see from~\eqref{eq:momentCondX} that these choices of $v^2,c$ yield
\begin{align*}
\tbinom{k+1}{2}^{-1} \sum_{a \leq b} \E \left\{ g\left( h_{ab}^2 \bar A_{ ab } \right)^m \right\}
& \leq \frac{ \Gamma\left( m + \frac{1}{2} \right) }{ \sqrt{\pi} } \left\{ 1 + o(1) \right\} , \quad m = 2, 3, \ldots
\\ & < \frac{ \Gamma\left( m + 1 \right) }{ \sqrt{\pi} } , \quad \text{ eventually in $n$},
\\ & \leq \frac{m!}{2} v^2 c^{m-2} , \quad m = 2, 3, \ldots ,
\end{align*}
and thus~\eqref{eq:BernsMomentCondX} holds eventually in $n$. Lemma~8 of \cite{birge1998minimum} then shows that for
\begin{equation*}
Y = \sum_{ a \leq b } g\left( h_{ab}^2 \bar A_{ ab } \right) , \quad \text{with $ z \in \mathcal{Z}_k $ fixed},
\end{equation*}
the following concentration inequality holds for any $\epsilon > 0$:
\begin{align}
\nonumber \Pr \left( Y - \E Y \geq \tbinom{k+1}{2} \epsilon \right) 
& \leq \exp \left( -\frac{ \binom{k+1}{2} \epsilon^2 / 2 }{ v^2 + c \epsilon } \right) 
\\ \label{eq:KLBerns} \Rightarrow \Pr \left( Y \geq \epsilon \right) 
& \leq \exp \left( - \frac{ ( \epsilon - \E Y)^2 / 2 }{ \binom{k+1}{2} v^2 + c ( \epsilon - \E Y ) } \right) .
\end{align}

Observe that since $ \E Y \geq 0$,~\eqref{eq:KLBerns} still holds if we replace $ \E Y $ with an upper bound $u$, because for any $ u \geq \E Y \geq 0 $, the event $ Y - u \geq \epsilon $ implies the event $ Y - \E Y \geq \epsilon $, and so $ \Pr \left( Y - u \geq \epsilon \right) \leq \Pr \left( Y - \E Y \geq \epsilon \right) $. Thus, we may substitute the eventual upper bound $ u = ( 1^+ / 2 ) \binom{k+1}{2} \geq \E Y $ from~\eqref{eq:momentCondX} into~\eqref{eq:KLBerns}, where $ ( 1^+ / 2 ) $ is arbitrarily close to $ 1 / 2 $. Substituting $ ( 1^+ / 2 ) \binom{k+1}{2} $ in place of $ \E Y $ in~\eqref{eq:KLBerns}, along with the constants $ v^2 = 3^+ / 4$ and $ c = 1 $, we see that for any $ \epsilon > 0 $, eventually in $n$,
\begin{equation*}
\Pr\left( Y \geq \epsilon \right) 
\leq \exp \left( - \frac{ \left\{ \epsilon - \frac{ 1^+ }{ 2 } \binom{k+1}{2} \right\}^2 / 2 }{ \binom{k+1}{2} \frac{ 3^+ }{ 4 } + \left\{ \epsilon - \frac{ 1^+ }{ 2 } \binom{k+1}{2} \right\} } \right) .
\end{equation*}
Simplifying this expression and applying a union bound over all $z \in \mathcal{Z}_k$ then yields the stated result.
\end{proof}

\begin{lemma}\label{PB-KL}
Let $X_n$ denote a sequence of Poisson--Binomial variates, each with mean $\mu_n$, and define
\begin{equation}
\label{eq:fXn}
g(X_n) = 
\begin{cases}
\mu_n \log \left( \frac{ \mu_n }{ X_n } \right) + ( n - \mu_n ) \log \left( \frac{ n - \mu_n }{ n - X_n }\right) & \text{if $X_n \in \{1, 2, \ldots , n-1\}$,}
\\ 0 & \text{if $X_n \in \{0, n\}$.}
\end{cases}
\end{equation}
If $ \min \left( \mu_n , n - \mu_n \right) = \omega \bigl( \sqrt{ \mu_n \log \{ \max \left( \mu_n , n - \mu_n \right) \} } \bigr) $, then the moments of $ g(X_n) $ satisfy for $m = 1, 2, \ldots$
\begin{align*}
\E \left\{ g(X_n)^m \right\} \leq \frac{ \Gamma\left( m + \frac{1}{2} \right) }{ \sqrt{\pi} } \left\{ 1 + \mathcal{O} \left( \frac{ \sqrt{ \mu_n \log \{ \max \left( \mu_n , n - \mu_n \right) \} } }{ \min \left( \mu_n , n - \mu_n \right) } \right) \right\} .
\end{align*}
\end{lemma}

\begin{proof}
To simplify notation, we suppress the dependence of $X$ and $\mu$ on $n$ throughout; note, however, that $ m \in \{1, 2, \ldots\} $ is fixed and so does not depend on $n$. Using the fact that $ g(0) = g(n) = 0 $, we write
\begin{align}
\nonumber
\E \left\{ g(X)^m \right\}
& = \sum_{k=0}^{n} g(k)^m \Pr\left( X = k \right) , \quad m = 1, 2, \ldots
\\ \nonumber & = \sum_{k=1}^{n-1} g(k)^m \Pr\left( X = k \right)
\\ \label{eq:KLSum} & = \left( \sum_{k=1}^{k_1} + \sum_{k=k_1+1}^{k_2-1} + \sum_{k=k_2}^{n-1} \right) g(k)^m \Pr\left( X = k \right),
\end{align}
with $k_1, k_2$ chosen to balance the contribution of the central sum in~\eqref{eq:KLSum} with that of the tail sums in~\eqref{eq:KLSum}:$\!$
\begin{subequations}
\label{eq:kSpec}
\begin{align}
k_1 & = \max \left\{ 1 , \left\lfloor \mu - \sqrt{ 2 \mu ( m + \delta ) \log \mu } \right\rfloor \right\} ,
\\ k_2 & = \min \left\{ \left\lceil \mu + \sqrt{ 2 \mu ( m + \delta ) \log ( n - \mu ) } \right\rceil , n-1 \right\}
\end{align}
\end{subequations}
for any fixed $ \delta > 0 $. Since $g(k) \geq 0$ for every value of $k$,~\eqref{eq:KLSum} implies that
\begin{align}
\nonumber
& \E \left\{ g(X)^m \right\}\leq \left\{ \max_{ 1 \leq k \leq k_1 } g(k)^m \right\} \sum_{k=1}^{k_1} \Pr\left( X = k \right) + \hskip-.2cm \sum_{k=k_1+1}^{k_2-1} \hskip-.2cm g(k)^m \Pr\left( X = k \right) \\
& \nonumber\qquad + \left\{ \max_{ k_2 < k < n } g(k)^m \right\} \hskip-.05cm \sum_{k=k_2}^{n-1} \hskip-.15cm \Pr\left( X = k \right)
\\ & \leq \left\{ \max_{ 1 \leq k \leq k_1 } g(k)^m \right\} \Pr\left( X \leq k_1 \right) \hskip.425cm + \hskip-.2cm \sum_{k=k_1+1}^{k_2-1} \hskip-.2cm g(k)^m \Pr\left( X = k \right) \nonumber\\
&\qquad + \left\{ \max_{ k_2 < k < n } g(k)^m \right\} \Pr\left( X \geq k_2 \right) .\label{eq:tailTerm} 
\end{align}

We now bound the two tail terms in~\eqref{eq:tailTerm}. From the definitions of $k_1$ and $k_2$ in~\eqref{eq:kSpec}, our hypothesis $ \min \left( \mu , n - \mu \right) = \omega \bigl( \sqrt{ \mu \log \{ \max \left( \mu , n - \mu \right) \} } \bigr) $ implies that eventually in $n$,
\begin{subequations}
\label{eq:kSpecEv}
\begin{align}
k_1 & = \mu - \epsilon_1 , \quad \epsilon_1 \geq \sqrt{ 2 \mu ( m + \delta ) \log ( \mu ) } ,
\\ k_2 & = \mu + \epsilon_2 , \quad \epsilon_2 \geq \sqrt{ 2 \mu ( m + \delta ) \log ( n - \mu ) } .
\end{align}
\end{subequations}

Now recall the standard Chernoff bounds for Poisson--Binomial variates, which hold for any $ \epsilon > 0$:
\begin{align*}
\Pr \left( X \leq \mu - \epsilon \right)
& \leq \exp\left( - \frac{ \epsilon^2 }{ 2 \mu } \right) ,
\\ \Pr \left( X \geq \mu + \epsilon \right)
& \leq \exp\left\{ - \frac{ \epsilon^2 }{ 2 \mu } \, \left( 1 + \frac{ \epsilon }{ 3 \mu } \right)^{-1} \right\} .
\end{align*}
Applying these bounds to $X \leq \mu - \epsilon_1$ and $X \geq \mu + \epsilon_2$, respectively, we conclude that eventually in $n$,
\begin{subequations}
\label{eq:Chernoff}
\begin{align}
\Pr \left( X \leq k_1 \right)
& \leq \mu^{ - ( m + \delta ) } ,
\\ \nonumber \Pr \left( X \geq k_2 \right)
& \leq \textstyle \exp\left\{ - ( m + \delta ) \log ( n - \mu ) \left( 1 + \frac{ \sqrt{ 2 ( m + \delta ) } }{ 3 } \sqrt{ \frac{ \log ( n - \mu ) }{ \mu } } \right)^{-1} \right\} 
\\ & \textstyle = ( n - \mu )^{ - ( m + \delta ) } \left\{ 1 + \mathcal{O} \left( \sqrt{ \frac{ \log ( n - \mu ) }{ \mu } } \right) \right\} ,
\end{align}
\end{subequations}
with the hypothesis $ \min \left( \mu , n - \mu \right) = \omega \bigl( \sqrt{ \mu \log \{ \max \left( \mu , n - \mu \right) \} } \bigr) $ implying that $ \mu = \omega \big( \log ( n - \mu ) \big) $.

This hypothesis also implies that $1 < \mu < n - 1$ eventually in $n$. Since $g(k)$ is strictly decreasing on $1 \leq k < \mu$ and strictly increasing on $\mu < k \leq n-1$, we have for $m = 1, 2, \ldots$ that $\max_{ 1 \leq k \leq k_1 } g(k)^m = g(1)^m \leq \left( \mu \log \mu \right)^m $ and $ \max_{ k_2 < k < n } g(k)^m = g(n-1)^m \leq \left\{ ( n - \mu ) \log ( n - \mu ) \right\}^m$ eventually in $n$.

Combining these two upper bounds with~\eqref{eq:tailTerm} and~\eqref{eq:Chernoff}, we conclude that, eventually in $n$,
\begin{multline}
\label{eq:tailBounds}
\E \left\{ g(X)^m \right\}
\leq \log ( \mu )^m \mu^{ - \delta } + \sum_{k=k_1+1}^{k_2-1} g(k)^m \Pr\left( X = k \right) 
\\ + \log ( n - \mu )^m ( n - \mu )^{ - \delta } \left\{ 1 + \mathcal{O} \left( \sqrt{ \frac{ \log ( n - \mu ) }{ \mu } } \right) \right\} .
\end{multline}
As a final step, we bound $\sum_{k=k_1+1}^{k_2-1} g(k)^m \Pr\left( X = k \right)$ in~\eqref{eq:tailBounds}. Recognizing $g(k)$ from~\eqref{eq:fXn} as a scaled form of a Bernoulli Kullback--Leibler divergence, we have by the Taylor expansion of Lemma~\ref{KLDiv} that 
\begin{multline}
\label{eq:censKLTaylor}
g(k) \leq \frac{ n(k-\mu)^2 }{ 2\mu(n-\mu) } \\ \cdot \left\{ 1 + \tfrac{2}{3} \tfrac{ \left| k - \mu \right| }{ \min( \mu, n - \mu ) } \bigl( 1 - \tfrac{ |k-\mu| }{ \min ( \mu, n - \mu ) } \bigr)^{-3} \right\} , \quad | k - \mu | < \min( \mu , n - \mu ) .
\end{multline}

Now,~\eqref{eq:kSpecEv} implies that for all $n$ sufficiently large, \\
$ | k - \mu | \leq \sqrt{ 2 \mu ( m + \delta ) \log \{ \max( \mu, n - \mu ) \} } + 1 $ whenever $k \in \{k_1, \ldots , k_2\}$, and so
\begin{align}
\nonumber
\frac{ | k - \mu | }{ \min \left( \mu , n - \mu \right) } 
& \leq \sqrt{ 2 ( m + \delta ) } \left[ \frac{ \sqrt{ \mu \log \{ \max( \mu, n - \mu ) \} } }{ \min \left( \mu , n - \mu \right) } \right] \\
&\nonumber \cdot\left[ 1 + \frac{ 1 }{ \sqrt{ 2 \mu ( m + \delta ) \log \{ \max( \mu, n - \mu ) \} } } \right]
\\ \label{eq:kTerm} & = \mathcal{O} \left( \frac{ \sqrt{ \mu \log \{ \max( \mu, n - \mu ) \} } }{ \min( \mu , n - \mu ) } \right) , \quad k_1 \leq k \leq k_2 ,
\end{align}
since the hypothesis $ \min \left( \mu , n - \mu \right) = \omega \bigl( \sqrt{ \mu \log \{ \max \left( \mu , n - \mu \right) \} } \bigr) $ implies that $ \mu = \omega( \log n ) $. From~\eqref{eq:kTerm}, we see that this hypothesis also implies that the Lagrange remainder term in~\eqref{eq:censKLTaylor} is $o(1)$.

Therefore, we may use the Taylor expansion of~\eqref{eq:censKLTaylor} to obtain the upper bound
\begin{align}
\nonumber
\sum_{k=k_1+1}^{k_2-1} & g(k)^m \Pr\left( X = k \right)
\\ \nonumber & \leq \sum_{k=k_1+1}^{k_2-1} \left\{ \frac{ n (k-\mu)^2 }{ 2\mu(n-\mu) } \right\}^m \left\{ 1 + \mathcal{O} \left( \frac{ | k - \mu | }{ \min( \mu , n - \mu ) } \right) \right\}^m \Pr\left( X = k \right) 
\\ \label{eq:middleSum} & = \left\{ \frac{ n }{ 2\mu(n-\mu) } \right\}^m \left\{ 1 + \mathcal{O} \left( \frac{ \sqrt{ \mu \log \{ \max( \mu, n - \mu ) \} } }{ \min( \mu , n - \mu ) } \right) \right\} \\
\nonumber &\cdot \sum_{k=k_1+1}^{k_2-1} (k-\mu)^{2m} \Pr\left( X = k \right) .
\end{align}
Noting that each term appearing in the sum of~\eqref{eq:middleSum} is nonnegative, we see that
\begin{align*}
\sum_{k=k_1+1}^{k_2-1} (k-\mu)^{2m} \Pr\left( X = k \right)
& \leq \left( \sum_{k=0}^{k_1} + \sum_{k=k_1+1}^{k_2-1} + \sum_{k=k_2}^n \right) (k-\mu)^{2m} \Pr\left( X = k \right)
\\ & = \E \left\{ (X-\mu)^{2m} \right\}, \quad m = 1, 2, \ldots,
\end{align*}
with each $\E \left\{ (X-\mu)^{2m} \right\}$ an even-order central moment of the Poisson--Binomial random variable $X$. 

\citet[Theorem~3.A.37]{shaked1994stochastic} show that \\ $Y \sim \operatorname{Binomial}(n,\mu/n)$ is larger than $X$ in the convex order, meaning that $\E \phi(X) \leq \E \phi(Y)$ holds for all convex functions $\phi: \mathbb{R} \rightarrow \mathbb{R}$ for which the expectations exist. Since the even-order central moments $\E (Y-\mu)^{2m}$ exist and are convex for all $m = 1, 2, \ldots$, it follows that
\begin{equation*}
\E \left\{ (X-\mu)^{2m} \right\} \leq \E \left\{ (Y-\mu)^{2m} \right\} , \quad m = 1, 2, \ldots ,
\end{equation*}
where $X$ is the Poisson--Binomial variate under study and the random variable $Y \sim \operatorname{Binomial}(n,\mu/n)$ has a matched mean.

As observed by \citet{romanovsky1923note}, the central moments of the Binomial distribution admit a recurrence relation that allows each of their leading-order terms to be expressed in closed form:
\begin{equation*}
\E \left\{ (Y-\mu)^{2m} \right\} = (2m-1)!! \left( \var Y \right)^m \left\{ 1 + \mathcal{O}\left( \frac{ 1 }{ \var Y } \right) \right\} ,
\end{equation*}
with\vspace{-1.5\baselineskip}%
\begin{align*}
\hskip0.4cm \var Y & = \frac{ \mu ( n - \mu ) }{ n } 
\\ & = \left\{ \frac{ \max \left( \mu , n - \mu \right) }{ n } \right\} \min \left( \mu , n - \mu \right)
\\ & = \Theta \bigl( \min \left( \mu , n - \mu \right) \bigr) .
\end{align*}
Thus we have from~\eqref{eq:middleSum} that
\begin{align}
\nonumber
\sum_{k=k_1+1}^{k_2-1} g(k)^m \Pr\left( X = k \right) 
& \leq \left\{ \frac{ n }{ 2\mu(n-\mu) } \right\}^m \left\{ 1 + \mathcal{O} \left( \frac{ \sqrt{ \mu \log \{ \max( \mu, n - \mu ) \} } }{ \min( \mu , n - \mu ) } \right) \right\} 
\\ \nonumber & \hskip-.25cm \cdot \left[ (2m-1)!! \left\{ \frac{ \mu(n-\mu) }{ n } \right\}^m \left\{ 1 + \mathcal{O}\left( \frac{ 1 }{ \min \left( \mu , n - \mu \right) } \right) \right\} \right]
\\ \label{eq:centreBound} & = \frac{ (2m-1)!! }{ 2^m } \left\{ 1 + \mathcal{O} \left( \frac{ \sqrt{ \mu \log \{ \max( \mu, n - \mu ) \} } }{ \min( \mu , n - \mu ) } \right) \right\} ,
\end{align} 
where the combination of the $ \mathcal{O}( \cdot ) $ terms follows because $ \mu = \omega( \log n ) $ is implied by the hypothesis that $ \min \left( \mu , n - \mu \right) = \omega \bigl( \sqrt{ \mu \log \{ \max \left( \mu , n - \mu \right) \} } \bigr) $. Finally, combining~\eqref{eq:tailBounds} with~\eqref{eq:centreBound}, and noting that $ (2m-1)!! / 2^m = \Gamma(m + 1/2) / \sqrt{\pi}$, we obtain for any choice of $ \delta > 0 $ and every fixed $ m = 1, 2, \ldots $ that
\begin{multline*}
\E \left\{ g(X)^m \right\}
\leq \log ( \mu )^m \mu^{ - \delta } + \frac{ \Gamma(m + 1/2) }{ \sqrt{\pi} } \left\{ 1 + \mathcal{O} \left( \frac{ \sqrt{ \mu \log \{ \max( \mu, n - \mu ) \} } }{ \min( \mu , n - \mu ) } \right) \right\} 
\\ + \log ( n - \mu )^m ( n - \mu )^{ - \delta } \left\{ 1 + \mathcal{O} \left( \sqrt{ \frac{ \log ( n - \mu ) }{ \mu } } \right) \right\} ,
\end{multline*}
eventually in $n$. To complete the proof, observe that $ \delta > 0 $ can be chosen for each $m$ such that the terms $ \log ( \mu )^m \mu^{ - \delta } $ and $ \log ( n - \mu )^m ( n - \mu )^{ - \delta } $ tend to $0$ arbitrarily quickly in $n$, thus yielding the theorem.
\end{proof}

\newcommand{\Graphontheorem}{\ref{absolute}}
\section{Proof of Theorem~\protect\Graphontheorem{} and lemmas}\label{sec:sbs-pf}

\subsection{Proof of Theorem~\protect\Graphontheorem{}}

\begin{proof}
Recall that our aim is to establish~\eqref{eq:min-risk-Holder}, which asserts that $ \min_{ z \in \mathcal{Z}_k } \sum_{ i<j } \D\left( p_{ij} \,\middle\vert\middle\vert\, \bar p_{ z_i z_j } \right) = {\cal O}_P \bigl( \left\{ n^{-\alpha} + ( n / h_\vee )^{-2\alpha} \right\} \cdot \sum_{ i<j } p_{ij} \bigr) $. We will do so by upper-bounding this risk in terms of a random community assignment vector $ \tilde z^* $ that depends on the ordered sample $\{ \xi_{(i)} \}_{i=1}^n $ of $\operatorname{Uniform}(0,1)$ variates that index the graphon $f$. Convergence of this ordered sample to the lattice $ (n+1)^{-1} ( 1, \ldots, n ) $, coupled with the uniform continuity of $f$, as enforced by a H\"older assumption, will yield the result.

We proceed as follows. Let $ z^* $ be any minimizer of $ \sum_{ i<j } \D\left( p_{ij} \,\middle\vert\middle\vert\, \bar p_{ z_i z_j } \right) $ over the set $ \mathcal{Z}_k $ of admissible blockmodel assignment vectors, and define $ \tilde z^*_i = H_{k,z^*}^{-1} \, \{ (i)^{-1} / n \} $, with $(i)^{-1}$ the rank of $\xi_i$ from smallest to largest. Thus $ \tilde z^* = H_{k,z^*}^{-1} \circ (\cdot)^{-1} $, and therefore by construction, condition~\ref{cond:oracle-perm} of the theorem ensures that $ \tilde z^* \in \mathcal{Z}_k $ for any $ z^* \in \mathcal{Z}_k $. Hence we have the following upper bound:
\begin{equation*}
\min_{ z \in \mathcal{Z}_k } \sum_{ i<j } \D\left( p_{ij} \,\middle\vert\middle\vert\, \bar p_{ z_i z_j } \right)
\leq \sum_{ i<j } \D\left( p_{ij} \,\middle\vert\middle\vert\, \bar p_{ {\tilde z}^*_i {\tilde z}^*_j } \right)
= \sum_{ i<j } \D\left( p_{(i)(j)} \,\middle\vert\middle\vert\, \bar p_{ {\tilde z}^*_{(i)} {\tilde z}^*_{(j)} } \right) , 
\end{equation*}
with equality stemming from the fact that the sum over all $ i < j $ is invariant to permutation, and hence we may re-order it in accordance with the ordered sample $\{ \xi_{(i)} \}_{i=1}^n $.

Conditions~\ref{cond:Holder} and~\ref{cond:rho} of the theorem then imply that Lemma~\ref{graphon-oracle} holds, thereby completing the proof. 
\end{proof}

\subsection{Auxiliary lemmas needed for Theorem~\protect\Graphontheorem{}}

\begin{lemma}\label{graphon-oracle}
$\!$If $r_n \rightarrow 0 $ in Lemma~\ref{taylor-div}, then
\begin{equation*}
\frac{ \sum_{i<j} \!\D\!\left( p_{(i)(j)} \middle\vert\middle\vert\bar p_{ {\tilde z}_{(i)} {\tilde z}_{(i)} } \right) }{ \sum_{i<j} \rho_n f\!\left( \xi_i , \xi_j \right)} = {\cal O}_P\bigl( r_n^2 \bigr).
\end{equation*}
\end{lemma}

\begin{proof}
This follows from via Slutsky's theorem, after combining the results of Lemmas~\ref{sump} and~\ref{comby}:
\begin{align*}
\textstyle \tbinom{n}{2}^{-1} \sum_{i<j} f\left( \xi_i , \xi_j \right) 
& = \textstyle \iint_{(0,1)^2} f\left(x,y\right)\,dx\,dy+{\cal O}_P\bigl( n^{-1/2} \bigr) ,
\\ \textstyle \left\{ \rho_n \tbinom{n}{2} \right\}^{-1} \sum_{i<j} \D\left( p_{(i)(j)} \,\middle\vert\middle\vert\, \bar p_{ {\tilde z}_{(i)} {\tilde z}_{(i)} } \right) 
& =\textstyle {\cal O}_P\left( r_n^2 \right) .
\end{align*} 
Since the denominator term converges in probability to a constant, it also converges in law. Thus by Slutsky's theorem, the ratio converges in law to a constant, and hence it also converges in probability.
\end{proof}

\begin{lemma}\label{sump}
Let $f$ be a symmetric measurable function on $(0,1)^2$ with bounded magnitude, and let $\{\xi_{i} \}_{i=1}^n$ be a random sample of $\operatorname{Uniform}(0,1)$ variates. Then
\begin{equation*}
\textstyle \tbinom{n}{2}^{-1} \sum_{i<j} f\left( \xi_i , \xi_j \right) = \iint_{(0,1)^2} f\left(x,y\right)\,dx\,dy+{\cal O}_P\bigl( n^{-1/2} \bigr).
\end{equation*}
\end{lemma}

\begin{proof}
The result follows from Chebyshev's inequality. We obtain the necessary moments as
\begin{align}
\nonumber
\textstyle\E \tbinom{n}{2}^{-1} \sum_{i<j} f\left( \xi_i , \xi_j \right) 
& = \textstyle\iint_{(0,1)^2} f\left(x,y\right)\,dx\, dy,\\
\textstyle\var \tbinom{n}{2}^{-1} \sum_{i<j} f\left( \xi_i , \xi_j \right)
& = \textstyle\tbinom{n}{2}^{-2} \sum_{i<j} \sum_{k<l} \cov\left\{f\left(\xi_{i},\xi_{j} \right),f\left(\xi_{k},\xi_{l} \right)\right\}.
\label{covarsum}
\end{align}
Since $\left|f\left(x,y\right)\right|$ is bounded by hypothesis, $\left|\cov\left\{f\left(\xi_{i},\xi_{j} \right),f\left(\xi_{k},\xi_{l} \right)\right\}\right|$ is also bounded. 
Furthermore, since elements of $\{\xi_{1},\dots, \xi_{n}\}$ are independent, any individual covariance term appearing in the sum of~\eqref{covarsum} can be nonzero only if $ \left(i=k\right) \cup \left(i=l\right) \cup \left(j=k\right) \cup \left(j=l\right)$. Thus we conclude that
\begin{multline*}
\textstyle\var \tbinom{n}{2}^{-1} \sum_{i<j} f\left( \xi_i , \xi_j \right)
\\ = \textstyle {\cal O}\Bigl(\tbinom{n}{2}^{-2} \sum_{i<j} \sum_{k<l} \left\{ \I\left(i=k\right) + \I\left(i=l\right) + \I\left(j=k\right) + \I\left(j=l\right) \right\} \Bigr).
\end{multline*}
The right-hand side of this expression is ${\cal O}\bigl( n^{-1} \bigr)$, and so Chebyshev's inequality yields the result.
\end{proof}

\begin{lemma}\label{comby}
Whenever $r_n \rightarrow 0 $ in~\eqref{eq:rn-defn} from Lemma~\ref{taylor-div}, we have that 
\begin{equation*}
\textstyle \left\{ \rho_n \tbinom{n}{2} \right\}^{-1} \sum_{i<j} \D\left( p_{(i)(j)} \,\middle\vert\middle\vert\, \bar p_{ {\tilde z}_{(i)} {\tilde z}_{(i)} } \right) = {\cal O}_P\left( r_n^2 \right) .
\end{equation*}
\end{lemma}

\begin{proof}
The result follows by combining Lemmas~\ref{taylor-div} and~\ref{divergenceMC2}. From Lemma~\ref{taylor-div}, we have directly that
\begin{equation*}
\rho_n^{-1} \D\left( p_{(i)(j)} \,\middle\vert\middle\vert\, \bar p_{ {\tilde z}_{(i)} {\tilde z}_{(i)} } \right)
= \rho_n^{-1} \D\left\{ p_{(i)(j)} \,\middle\vert\middle\vert\, \rho_n \bar f\left(\xi_{(i)},\xi_{(j)} \right) \right\} + {\cal O}_P\left( r_n^2 \right) 
\end{equation*}
under the hypothesis that $r_n \rightarrow 0$, and thus
\begin{align*}
& \textstyle \left\{ \rho_n \tbinom{n}{2} \right\}^{-1} \sum_{i<j} \D\left( p_{(i)(j)} \,\middle\vert\middle\vert\, \bar p_{ {\tilde z}_{(i)} {\tilde z}_{(i)} } \right)
\\ & \quad  = \textstyle\left\{ \rho_n \tbinom{n}{2} \right\}^{-1} \sum_{i<j} \D\left\{ p_{(i)(j)} \,\middle\vert\middle\vert\, \rho_n \bar f\left(\xi_{(i)},\xi_{(j)} \right) \right\} + {\cal O}_P\left( r_n^2 \right) 
\\ & \quad  =\textstyle \left\{ \rho_n \tbinom{n}{2} \right\}^{-1} \sum_{i<j} \D\left\{\rho_n f\left(\xi_{i}, \xi_{j}\right) \,\middle\vert\middle\vert\, \rho_n {\bar f}\left(\xi_{i}, \xi_{j}\right) \right\} \!+\! {\cal O}_P\left( r_n^2 \right)\! ,
\end{align*}
after re-ordering the sum and applying the identity $ p_{ij} = \rho_n f\left(\xi_{i}, \xi_{j}\right)$. The right-hand side of this expression is treated by Lemma~\ref{divergenceMC2}, which shows whenever $ \max_{1 \leq a, b \leq k} \Delta_{ab} = o(1)$ in~\eqref{divergenceMC} that
\begin{multline}
\label{divergenceMC3}
\! \left\{ \rho_n \tbinom{n}{2} \right\}^{-1} \E \sum_{i<j} \D\left\{\rho_n f\left(\xi_{i}, \xi_{j}\right) \,\middle\vert\middle\vert\, \rho_n {\bar f}\left(\xi_{i}, \xi_{j}\right) \right\} 
\\ = \frac{ \rho_n M^2 \textstyle \left( \sqrt{2} \max_{1 \leq a \leq k } h_a / n \right)^{2\alpha} \left\{ 1 + o(1) \right\} }{ \min_{1\le a,b\le k}\left\{\min\left(\rho_n{\bar f}_{ab},1-\rho_n{\bar f}_{ab}\right)\right\} } .
\end{multline}
Since~\eqref{eq:delta-bnd} of Lemma~\ref{divergenceMC2} upper-bounds each $\Delta_{ab} $ by the ratio of terms $ \rho_n M \, \bigl( \sqrt{2} \max_{ a } h_a / n \bigr)^\alpha / \min \left( \rho_n {\bar f}_{ab} , 1 -\rho_n {\bar f}_{ab} \right) $, we see that $ \Delta_{ab} = {\cal{O}}\left( r_n \right)$, and so the hypothesis $r_n \rightarrow 0$ is sufficient to imply that $ \max_{ a, b } \Delta_{ab} = o(1)$. 

We also see that the main term in~\eqref{divergenceMC3} is $ {\cal{O}}\left( r_n^2 \right) $, since the quantity $ \min_{1\le a,b\le k}\left\{\min\left({\bar f}_{ab},\rho_n^{-1}-{\bar f}_{ab}\right)\right\} \leq \sup_{(x,y)\in (0,1)^2}f\left(x,y\right) $, and thus after applying Markov's inequality via~\eqref{divergenceMC3}, we obtain the result.
\end{proof}

\begin{lemma}\label{taylor-div}
Let $ f $ be a symmetric $\operatorname{\textrm{H\"older}}^\alpha(M)$ function on $(0,1)^2$, with $ {\bar f}\left(x,y; h\right) = {\bar f}_{ H^{-1}(x) H^{-1}(y) } $ its stepfunction approximation, and let $\{\xi_{(i)}\}_{i=1}^n$ be an ordered sample of independent $\operatorname{Uniform}(0,1)$ random variables. Assume $\rho_n>0$ and $ 0 < \rho_n f\left(x,y\right) < 1 $ everywhere on $(0,1)^2$. Then for any ${\tilde z}$ such that $ \Pi_{{\tilde z}} = (\cdot)^{-1} $, with $(i)^{-1}$ denoting the rank of $\xi_i$ from smallest to largest, we have
\begin{equation*}
\rho_n^{-1} \D\left( p_{(i)(j)} \,\middle\vert\middle\vert\, \bar p_{ {\tilde z}_{(i)} {\tilde z}_{(i)} } \right)
= \rho_n^{-1} \D\left\{ p_{(i)(j)} \,\middle\vert\middle\vert\, \rho_n \bar f\left(\xi_{(i)},\xi_{(j)} \right) \right\} + {\cal O}_P\left( r_n^2 \right) 
\end{equation*}
whenever\vspace{-\baselineskip}%
\begin{equation}
\label{eq:rn-defn}
r_n = \frac{ \rho_n M \, 2^{\alpha / 2 } \left\{ \frac{ 2^{1 - \alpha } }{ n^{ \alpha / 2 } }
+ \frac{ 2 \left( \max_{1 \leq a \leq k } h_a \right)^\alpha + 1 + 2^\alpha \I\left( {\tilde z}_{(i)} = {\tilde z}_{(j)} \right) }{ n^\alpha } \right\} }{\min_{1\le a,b\le k}\left\{\min\left(\rho_n {\bar f}_{ab},1-\rho_n{\bar f}_{ab}\right)\right\}} \rightarrow 0 .
\end{equation}
\end{lemma}

\begin{proof}
We apply Taylor's theorem, after first establishing via Markov's inequality that
\begin{equation}
\delta_n = \frac{ \bar p_{ {\tilde z}_{(i)} {\tilde z}_{(i)} } - \rho_n \bar f\left(\xi_{(i)},\xi_{(j)} \right) }{ \min \left\{ \rho_n \bar f\left( \xi_{(i)} , \xi_{(j)} \right)
, 1 - \rho_n \bar f\left( \xi_{(i)}, \xi_{(j)} \right) \right\} } 
= {\cal O}_P\left( r_n \right) .
\label{rn}
\end{equation}
To show~\eqref{rn}, we lower-bound the denominator of $\delta_n$, and then apply Lemma~\ref{lem:pbar-ftilde} to upper-bound $ \E \left|\delta_n \right|$:
\begin{equation*}
\E \left|\delta_n \right| \leq \E \frac{ \rho_n\left|\rho_n^{-1}\bar p_{ {\tilde z}_{(i)} {\tilde z}_{(i)} } - \bar f\left(\xi_{(i)},\xi_{(j)} \right)\right| }{\min_{1\le a,b\le k}\left\{\min\left(\rho_n{\bar f}_{ab},1-\rho_n{\bar f}_{ab}\right)\right\}} \leq r_n .
\end{equation*}

We now apply Taylor's theorem to expand $ \D\left( p_{(i)(j)} \,\middle\vert\middle\vert\, \bar p_{ {\tilde z}_{(i)} {\tilde z}_{(i)} } \right) $ as a function of $\delta_n$ about the point $ \rho_n \bar f\left(\xi_{(i)},\xi_{(j)} \right) $. Writing $ \bar {\bar p}_{(i)(j)} $ for $ \rho_n \bar f\left(\xi_{(i)},\xi_{(j)} \right) $, we have that if $r_n \rightarrow 0 $, then
\begin{align}
\nonumber
 &\Bigl| \D\left( p_{(i)(j)} \,\middle\vert\middle\vert\, \bar p_{ {\tilde z}_{(i)} {\tilde z}_{(i)} } \right) 
- \D\left( p_{(i)(j)} \,\middle\vert\middle\vert\, \bar {\bar p}_{(i)(j)} \right) \Bigr|
= \left| \left( \bar {\bar p}_{(i)(j)} - p_{(i)(j)} \right)
\frac{ \bar p_{ {\tilde z}_{(i)} {\tilde z}_{(i)} } - \bar {\bar p}_{(i)(j)} }{ \bar {\bar p}_{(i)(j)} \left( 1 - \bar {\bar p}_{(i)(j)} \right) } \right.
\\ \nonumber & \quad \left. + \frac{1}{2} \left\{ p_{(i)(j)} \left( 1 - 2 \bar {\bar p}_{(i)(j)} \right) + \bar {\bar p}_{(i)(j)}^2 \right\}
\biggl\{ \frac{ \bar p_{ {\tilde z}_{(i)} {\tilde z}_{(i)} } - \bar {\bar p}_{(i)(j)} }{ \bar {\bar p}_{(i)(j)} \left( 1 - \bar {\bar p}_{(i)(j)} \right) } \biggr\}^2 
+o_{P}\left( \rho_n r_n^2 \right) \right|
\\ \nonumber & \,\, = \left| \frac{ \bar {\bar p}_{(i)(j)} - p_{(i)(j)} }{ \max \left( \bar {\bar p}_{(i)(j)} , 1 - \bar {\bar p}_{(i)(j)} \right) } \, \delta_n + \frac{ p_{(i)(j)} \left( 1 - 2 \bar {\bar p}_{(i)(j)} \right) + \bar {\bar p}_{(i)(j)}^2 }{ 2 \left\{ \max \left( \bar {\bar p}_{(i)(j)} , 1 - \bar {\bar p}_{(i)(j)} \right) \right\}^2 } \, \delta_n^2 + o_{P}\left( \rho_n r_n^2 \right) \right| 
\\ \label{eq:taylor-final} & \,\, < 2 \rho_n M \textstyle \left( \sqrt{2} \max_{1 \leq a \leq k } h_a / n \right)^\alpha \left| \delta_n \right| + 3\rho_n 
\sup_{(x,y)\in (0,1)^2}f\left(x,y\right)
\delta_n^2 + o_{P}\left( \rho_n r_n^2 \right) ,
\end{align}
where the terms in~\eqref{eq:taylor-final} follow because, by Lemma~\ref{normbound}, $ \left| \bar {\bar p}_{(i)(j)} - p_{(i)(j)} \right| \leq \rho_n M \, \bigl( \sqrt{2} \max_{1 \leq a \leq k } h_a / n \bigr)^\alpha $, since $ f \in \operatorname{\textrm{H\"older}}^\alpha(M)$; also, since $ 0 < \bar {\bar p}_{(i)(j)} < 1 $, we have that $ \left | 1 - 2 \bar {\bar p}_{(i)(j)} \right| / \max \left( \bar {\bar p}_{(i)(j)} , 1 - \bar {\bar p}_{(i)(j)} \right) < 1 $; and likewise we have $ \max \left( \bar {\bar p}_{(i)(j)} , 1 - \bar {\bar p}_{(i)(j)} \right) \geq 1/2 $. Since $ f \in \operatorname{\textrm{H\"older}}^\alpha(M)$ is bounded by hypothesis, the right-hand side of~\eqref{eq:taylor-final} is $ {\cal O}_P\left( \rho_n r_n^2 \right) $. The lemma follows from multiplying both sides of~\eqref{eq:taylor-final} by $ \rho_n^{-1} $.
\end{proof}

\begin{lemma}\label{lem:pbar-ftilde}
Let $ f $ be a symmetric $\operatorname{\textrm{H\"older}}^\alpha(M)$ function on $(0,1)^2$, and let $\{\xi_{(i)}\}_{i=1}^n$ be an ordered sample of independent $\operatorname{Uniform}(0,1)$ variates. Let $ \rho_n > 0 $ and define for $ z_i = H^{-1} \! \left\{ \Pi_z(i) / n \right\} $:
\begin{equation}
\label{eq:barp}
\bar p(z)_{ ab }
= \frac{ 1 }{ h_{ab }^{2} } \sum_{ j = n H(b-1)+1 }^{ n H(b) } \sum_{ i = n H(a-1)+1}^{ n H(a) \I\left( a \neq b \right) + \left( j - 1 \right) \I\left( a = b \right) } \rho_n f\left(\xi_{\Pi_z^{-1}(i)},\xi_{\Pi_z^{-1}(j)} \right) .
\end{equation}
Then for any ${\tilde z}$ such that $ \Pi_{{\tilde z}} = (\cdot)^{-1} $, with $(i)^{-1}$ denoting the rank of $\xi_i$ from smallest to largest, we have
\begin{multline}
\label{eq:pbar-f-bnd}
\E \left| \rho_n^{-1} \, \bar p_{ {\tilde z}_{(i)} {\tilde z}_{(j)} } \! - \! \bar f\left(\xi_{(i)},\xi_{(j)} \right) \right| 
\\ \leq M \, 2^{\alpha / 2 } \left\{ \frac{ 2^{1 - \alpha } }{ n^{ \alpha / 2 } }
\! + \! \frac{ 2 \left( \max_{1 \leq a \leq k } h_a \right)^\alpha \!\! + \! 1 \! + \! 2^\alpha \I\left( {\tilde z}_{(i)} \! = \! {\tilde z}_{(j)} \right) }{ n^\alpha } \right\} \! .
\end{multline}
\end{lemma}

\begin{proof}
Define the $k\times k$ matrix $ \tilde f$ such that $ \rho_n^{-1} \, \bar p( \tilde z )_{ab} = \tilde f( \tilde z )_{ a b } + {\cal O}_P\bigr( n^{-\alpha/2}\bigr) $ when $f$ is $\alpha$-H\"older:
\begin{equation}
\label{eq:tildef}
\tilde f( z )_{ a b } = \frac{ 1 }{h_{a b }^2} \sum_{ j = n H(b-1)+1 }^{ n H(b) } \sum_{ i = n H(a-1)+1}^{ n H(a) \I\left( a \neq b \right) + \left( j - 1 \right) \I\left( a = b \right) } f\left( \tfrac{ \Pi_z^{-1} \left\{ (i)^{-1} \right\} }{n+1} , \tfrac{ \Pi_z^{-1} \left\{ (j)^{-1} \right\} }{n+1} \right) .
\end{equation}
Note that $ \tilde f( \tilde z ) $ is deterministic, since the set of admissible $ \tilde z $ has been chosen such that $ \Pi_z^{-1} \left\{ (i)^{-1} \right\} = i $ for all $ 1 \leq i \leq n $. We will then obtain our claimed result by bounding the expectation of
\begin{multline}
\label{eq:barp-triangle}
\left| \rho_n^{-1} \, \bar p_{ {\tilde z}_{(i)} {\tilde z}_{(j)} } - \bar f\left(\xi_{(i)},\xi_{(j)} \right) \right| 
\\ \leq 
\left| \rho_n^{-1} \, \bar p_{ {\tilde z}_{(i)} {\tilde z}_{(j)} } - {\tilde f}_{ {\tilde z}_{(i)} {\tilde z}_{(j)} } \right| 
+ \left| {\tilde f}_{ {\tilde z}_{(i)} {\tilde z}_{(j)} } - \bar f\left( i_n,j_n\right) \right|
+ \left| \bar f\left( i_n,j_n\right) - \bar f\left(\xi_{(i)},\xi_{(j)} \right) \right| .
\end{multline}

We begin with the final term in~\eqref{eq:barp-triangle}, for which Lemma~\ref{ftilde} immediately yields
\begin{equation}
\label{eq:barp-triangleI}
 \E \, \bigl| \bar f\left( i_n,j_n\right) - \bar f\left(\xi_{(i)},\xi_{(j)} \right) \bigr| \leq M \left\{ 2(n+2) \right\}^{- \alpha / 2 } + 2 \textstyle M \left( \sqrt{2} \max_{1 \leq a \leq k } h_a / n \right)^\alpha . 
\end{equation} 

Next we consider the first term in~\eqref{eq:barp-triangle}. To bound its expectation, note that both $ \rho_n^{-1} \, \bar p( {\tilde z} )_{ a b } $ and $ {\tilde f}( {\tilde z} )_{ a b } $ are averages over the same subset of indices $ (i,j) $. From~\eqref{eq:barp} and~\eqref{eq:tildef}, we then have that
\begin{align}
& \nonumber \E \left| \rho_n^{-1} \, \bar p( {\tilde z} )_{ a b } - {\tilde f}( {\tilde z} )_{ a b } \right| 
\\ & \,\, \leq \frac{ 1 }{h_{a b }^2} \sum_{ j = n H(b-1)+1 }^{ n H(b) } 
\!\!\!\! \sum_{ i = n H(a-1)+1}^{ n H(a) \I\left( a \neq b \right) + \left( j - 1 \right) \I\left( a = b \right) } \!\!\! \E \left| f\left(\xi_{(i)},\xi_{(j)} \right) - f\bigl( \tfrac{i}{n+1} , \tfrac{j}{n+1} \bigr) \right|
\\ \label{eq:barp-triangleII} & \,\, 
\leq 1 \cdot M \left\{ 2(n+2) \right\}^{ - \alpha / 2 } ,
\end{align}
with the final inequality following again from Lemma~\ref{ftilde}. Since~\eqref{eq:barp-triangleII} holds uniformly over all $ {\tilde z} $ and every $ 1 \leq a , b \leq k $, we have bounded $ \E \, \bigl| \rho_n^{-1} \, \bar p_{ {\tilde z}_{(i)} {\tilde z}_{(j)} } - {\tilde f}_{ {\tilde z}_{(i)} {\tilde z}_{(j)} } \bigr| $. 

It remains only to bound $ \E \, \bigl| {\tilde f}_{ {\tilde z}_{(i)} {\tilde z}_{(j)} } - \bar f\left( i_n,j_n\right) \bigr| $. We will do so using the following deterministic upper bound, which we prove below, and which holds uniformly over all $ {\tilde z} $ and every $ 1 \leq a , b \leq k $:
\begin{align}
\label{eq:fbarbarint}
\left| {\tilde f}( {\tilde z} )_{ a b } - {\bar f}( {\tilde z} )_{ab} \right| 
& \leq M \, \bigl\{ \sqrt{2} / ( n + 1 ) \bigr\}^\alpha + M \, \bigl( \sqrt{ 2 } \, h_a / n \bigr)^{\alpha} \left( h_a - 1 \right)^{-1} \I\left( a = b \right)
\\ \label{eq:barp-triangleIII} & \leq M \, 2^{\alpha / 2 } \, n^{ - \alpha } \left\{ 1 + 2^\alpha \I\left( a = b \right) \right\} .
\end{align}
Here the second inequality following because, by definition, any $H(\cdot)$ has $\min_{1\leq a \leq k} h_a \geq 2 $. 

Lemma~\ref{partition} yields $ \left( i_n , j_n \right) \in \omega_{ {\tilde z}_{(i)} {\tilde z}_{(j)} } $ for any $ \tilde z $; thus $ {\bar f}_{ {\tilde z}_{(i)} {\tilde z}_{(j)} } = \bar f\left( i_n , j_n \right) $, and so if~\eqref{eq:fbarbarint} holds, then it applies to $ \bigl| {\tilde f}_{ {\tilde z}_{(i)} {\tilde z}_{(j)} } - \bar f\left( i_n, j_n \right) \bigr| $. Finally, summing~\eqref{eq:barp-triangleI},~\eqref{eq:barp-triangleII} and \eqref{eq:barp-triangleIII} to obtain~\eqref{eq:pbar-f-bnd} completes the proof.

To establish~\eqref{eq:fbarbarint}, let $ i_n = i / (n+1) $, and multiply $ {\tilde f}( {\tilde z} )_{ a b } $ from~\eqref{eq:tildef} by $ 1 = n^2 / n^2 $ to obtain
\begin{align}
\nonumber
{\tilde f}( {\tilde z} )_{ a b } & =
\frac{ n^2 }{h_{a b }^2} 
\sum_{ j = n {H}(b-1)+1}^{ n {H}(b) } 
\sum_{ i = n {H}(a-1)+1}^{ n {H}(a) \I\left( a \neq b \right) + \left( j - 1 \right) \I\left( a = b \right) } 
\left( \frac{ 1 }{ n^2 } \right) 
f\left(i_n,j_n \right), \,\, 1 \leq a < b \leq k,
\\ \nonumber & = 
\frac{ n^2 }{h_{a b }^2} 
\sum_{ j = n {H}(b-1)+1}^{ n {H}(b) } 
\sum_{ i = n {H}(a-1)+1}^{ n {H}(a) \I\left( a \neq b \right) + \left( j - 1 \right) \I\left( a = b \right) } 
\left(
\int_{\frac{j-1}{n}}^{\frac{j}{n}}
\int_{\frac{i-1}{n}}^{\frac{i}{n}} 
\, dx \,dy \right) 
f\left(i_n,j_n \right)
\\ \label{eq:fbar-splitI} & =
\frac{ n^2 }{h_{a b }^2} 
\sum_{ j = n {H}(b-1)+1}^{ n {H}(b) } 
\sum_{ i = n {H}(a-1)+1}^{ n {H}(a) \I\left( a \neq b \right) + \left( j - 1 \right) \I\left( a = b \right) } 
\\ & \qquad \cdot
\int_{\frac{j-1}{n}}^{\frac{j}{n}}
\int_{\frac{i-1}{n}}^{\frac{i}{n}} 
\left[ {f}\left(x,y\right) + \left\{ {f}\left(i_n,j_n \right) - {f}\left(x,y\right) \right\} \right]
\,dx\, dy .
\end{align}
From~\eqref{eq:fbar-splitI} we will obtain the left-hand side of~\eqref{eq:fbarbarint}, plus a remainder term when $ a = b $, by writing
\begin{multline}
\label{eq:fbar-split}
{\tilde f}( {\tilde z} )_{ a b } - \frac{ n^2 }{h_{a b }^2} 
\sum_{ j = n {H}(b-1)+1}^{ n {H}(b) } 
\sum_{ i = n {H}(a-1)+1}^{ n {H}(a) \I\left( a \neq b \right) + \left( j - 1 \right) \I\left( a = b \right) } 
\int_{\frac{j-1}{n}}^{\frac{j}{n}}
\int_{\frac{i-1}{n}}^{\frac{i}{n}} 
f\left(x,y \right) \, dx \, dy 
\\ = {\tilde f}( {\tilde z} )_{ a b } - 
\begin{cases}
\displaystyle
\frac{n^2}{ h_a h_b } 
\int_{ H(b-1)}^{H(b)}
\int_{ H(a-1)}^{H(a)}
f\left(x,y \right) \, dx \, dy 
& a \neq b ,
\\ \displaystyle \frac{ n^2 }{ \binom{h_b}{2} }
\sum_{ j = n H(b-1) + 1}^{ n H(b) }
\int_{\frac{j-1}{n}}^{\frac{j}{n}}
\left(
\int_{ H(a-1)}^{y} 
- 
\int_{\frac{j-1}{n}}^y
\right)
f\left(x,y \right) \, dx \, dy
& a = b .
\end{cases}
\end{multline}
We recognize the first case in~\eqref{eq:fbar-split} as $ {\bar f}( {\tilde z} )_{ab , a \neq b} $. Since $f$ is symmetric, the second case can be written
\begin{align*} 
{\bar f}( {\tilde z} )_{bb} & + 
\sum_{ j = n H(b-1) + 1}^{ n H(b) }
\int_{\frac{j-1}{n}}^{\frac{j}{n}}
\left[ \left\{ \frac{ n^2 }{ \binom{h_b}{2} } - \frac{ 2 n^2 }{ h_b^2 } \right\}
\int_{ H(b-1)}^{y} 
- \frac{ n^2 }{ \binom{h_b}{2} }
\int_{\frac{j-1}{n}}^y
\right]
f\left(x,y \right) \, dx \, dy
\\ & = {\bar f}( {\tilde z} )_{bb} +
\frac{ 1 }{ h_b - 1 }
\sum_{ j = n H(b-1) + 1}^{ n H(b) }
\int_{\frac{j-1}{n}}^{\frac{j}{n}}
\left\{ \frac{ 2 n^2 }{ h_b^2 }
\int_{ H(b-1)}^{y} 
- \frac{ 2 n^2 }{ h_b }
\int_{\frac{j-1}{n}}^y
\right\}
f\left(x,y \right) \, dx \, dy
\\ & = {\bar f}( {\tilde z} )_{bb} +
\frac{ 1 }{ h_b - 1 }
\left\{ {\bar f}( {\tilde z} )_{bb} -
\frac{ 1 }{ h_b }
\sum_{ j = n H(b-1) + 1}^{ n H(b) }
2 n^2 \int_{\frac{j-1}{n}}^{\frac{j}{n}}
\int_{\frac{j-1}{n}}^y
f\left(x,y \right) \, dx \, dy
\right\}
\\ & = {\bar f}( {\tilde z} )_{bb} +
\frac{ 1 }{ h_b - 1 }
\left[
\frac{ 1 }{ h_b }
\sum_{ j = n H(b-1) + 1}^{ n H(b) }
2 n^2 \int_{\frac{j-1}{n}}^{\frac{j}{n}}
\int_{\frac{j-1}{n}}^y
\left\{ {\bar f}( {\tilde z} )_{bb} -
f\left(x,y \right) \right\} \, dx \, dy
\right] .
\end{align*}
Since $ \bar f\left(x,y ; h \right) = {\bar f}( {\tilde z} )_{bb} $ on the domain of interest $ \omega_{bb} = \left[ H(b-1) , H(b) \right)^2 $, we conclude 
\begin{align*}
\frac{ 1 }{ h_b }
\sum_{ j = n H(b-1) + 1}^{ n H(b) }
2 n^2 \int_{\frac{j-1}{n}}^{\frac{j}{n}}
\int_{\frac{j-1}{n}}^y
\left| {\bar f}( {\tilde z} )_{bb} -
f\left(x,y \right) \right| \, dx \, dy
& \leq 1 \cdot 1 \cdot 
\left\| {\bar f} - f\vert_{\omega_{bb}} \right\|_{ L_\infty \left( \omega_{bb} \right) }
\\ & \leq M \, \bigl( \sqrt{ 2 } \, h_b / n \bigr)^{\alpha} ,
\end{align*}
with the latter inequality from~\eqref{eq:lips-omega} of Lemma~\ref{normbound}, since $f \in \operatorname{\textrm{H\"older}}^\alpha(M)$. This yields the upper bound term in~\eqref{eq:fbarbarint} specific to $ a = b $. To derive the main term in~\eqref{eq:fbarbarint}, we return to~\eqref{eq:fbar-splitI}, noting from Lemma~\ref{ftilde}:
\begin{align*} 
\left| \frac{n^2}{h_{a b }^2}
\vphantom{\int_{\frac{j-1}{n}}^{\frac{j}{n}}\int_{\frac{i-1}{n}}^{\frac{i}{n}}}
\right. & \left.
\sum_{ j = n H(b-1) + 1}^{ n H(b) }
\sum_{ i = n H(a-1) + 1}^{ n H(a) \I\left( a \neq b \right) + ( j - 1 ) \I\left( a = b \right) } 
\int_{\frac{j-1}{n}}^{\frac{j}{n}}
\int_{\frac{i-1}{n}}^{\frac{i}{n}} 
\left\{ {f}\left(i_n,j_n \right)-{f}\left(x,y\right)\right\} 
\, dx \, dy\right|
\\ & \le \frac{ 1 }{h_{a b }^2}
\sum_{ j = n H(b-1) + 1}^{ n H(b) }
\sum_{ i = n H(a-1) + 1}^{ n H(a) \I\left( a \neq b \right) + ( j - 1 ) \I\left( a = b \right) } 
\\ & \qquad \cdot
n^2
\int_{\frac{j-1}{n}}^{\frac{j}{n}}
\int_{\frac{i-1}{n}}^{\frac{i}{n}} 
\left| f\left(i_n,j_n \right)-{f}\left(x,y\right) \right|
\, dx \, dy
\\ & \le 1 \cdot 1 \cdot M \, \bigl\{ \sqrt{2} / ( n + 1 ) \bigr\}^\alpha .
\end{align*}
\begin{equation*}
\vspace{-2\baselineskip}
\end{equation*}
\end{proof}

\begin{lemma}\label{normbound}
Let $ f $ be a $\operatorname{\textrm{H\"older}}^\alpha(M)$ function on $(0,1)^2$, with $\bar f\left(x,y; h\right) = {\bar f}_{ H^{-1}(x) H^{-1}(y) } $ its stepfunction approximation. Then for all $ 0 < p \leq \infty $,
\begin{equation*}
\|f-{\bar f}\|_{ L_p \left( ( 0 , 1 )^2 \right) } \leq \textstyle M \left( \sqrt{2} \max_{1 \leq a \leq k } h_a / n \right)^\alpha .
\end{equation*}
\end{lemma}

\begin{proof}
Let $ \omega_{ab} = \left[ H(a-1) , H(a) \right) \times \left[ H(b-1) , H(b) \right) \subseteq (0,1)^2 $, and denote by $f\vert_{\omega_{ab}}$ the restriction of $f$ to $\omega_{ab}$. By the definitions of ${\bar f}_{ab}$ and $\bar f\left(x,y\right)$,
\begin{align}
\nonumber
\left| {\bar f}_{ab} - f\left(x,y\right) \right| & = \left| \frac{1}{|\omega_{ab}|} \iint_{ \omega_{ab} } f\left(x',y'\right) \, dx' \,dy' - f\left(x,y\right) \right| , \,\, \left(x,y\right) \in (0,1)^2
\\ \nonumber
\Rightarrow \, \left| {\bar f}\left(x,y\right)-f\left(x,y\right)\right|
&\leq \frac{1}{|\omega_{ab}|} \iint_{ \omega_{ab} } \left| f\left(x',y'\right) - f\left(x,y\right) \right| \, dx' \,dy' , \quad \left(x,y\right) \in \omega_{ab} 
\\ \nonumber
\Rightarrow \, \left\| {\bar f} - f\vert_{\omega_{ab}} \right\|_{ L_\infty \left( \omega_{ab} \right) } 
&\leq \frac{1}{|\omega_{ab}|} \iint_{\omega_{ab}} \left| f\left(x',y'\right) - f\left(x,y\right)\right| \, dx' \,dy' , \quad \left(x,y\right) \in \omega_{ab} 
\\ \label{eq:lips-omega-pre}
&\leq \frac{M}{|\omega_{ab}|} \iint_{\omega_{ab}} \left| (x, y)-(x', y') \right|^{\alpha} \, dx' \,dy' , \quad \left(x,y\right) \in \omega_{ab} ,
\end{align}
since $ \left| f\left( x,y\right)-f\left(x',y'\right)\right| \leq M \left| \left(x,y\right) - \left(x',y'\right) \right|^\alpha = M \, \{ (x-x')^2 + (y-y')^2 \}^{\alpha/2} $ holds on $(0,1)^2$.

To simplify~\eqref{eq:lips-omega-pre}, note that the diameter $ \sup_{(x,y),(x',y') \in \omega_{ab}} \left| (x, y) - (x', y') \right| $ of the rectangular domain $\omega_{ab}$ evaluates to $ \sqrt{ h_a^2 + h_b^2 } / n $, where $ h_a = H(a) - H(a-1) $. Thus~\eqref{eq:lips-omega-pre} implies
\begin{equation}
\label{eq:lips-omega}
\left\| {\bar f} - f\vert_{\omega_{ab}} \right\|_{ L_\infty \left( \omega_{ab} \right) } \leq M \, \bigl( { \textstyle \sqrt{ h_a^2 + h_b^2 } } / n \bigr)^{\alpha},\quad 1\le a,b\le k,
\end{equation}
and so we immediately conclude $ \| \bar f - f \|_{ L_\infty \left( ( 0 , 1 )^2 \right) } \leq M \left( \sqrt{2} \max_a h_a / n \right)^\alpha $. Thus for any $0<p<\infty$,
\begin{align*}
\| \bar f - f \|_{ L_p \left( ( 0 , 1 )^2 \right) }^p 
& = \iint_{(0,1)^2} \left| \bar f\left(x,y\right)- \left(x,y\right)\right|^p\, dx \,dy
\\ & \leq\iint_{(0,1)^2} \left\{ \| \bar f - f \|_{ L_\infty \left( ( 0 , 1 )^2 \right) } \right\}^p \, dx \,dy .%
\vspace{-\baselineskip}%
\end{align*}%
\end{proof}

\begin{lemma}\label{ftilde}
Let $ f $ be a $\operatorname{\textrm{H\"older}}^\alpha(M)$ function on $(0,1)^2$, and let $\{\xi_{(i)}\}_{i=1}^n$ be an ordered sample of independent $\operatorname{Uniform}(0,1)$ random variables. Then, recalling that $\E\xi_{(i)}=i/(n+1)$, we have for $ 1 \leq i, j \leq n $:
\begin{alignat*}{2}
\E \left| f\left( \xi_{(i)},\xi_{(j)}\right) - f\bigl( \tfrac{i}{n+1} , \tfrac{j}{n+1} \bigr) \right|^{\beta} & \leq M^{\beta} \left\{ 2(n+2) \right\}^{ - \alpha\beta / 2 } , \hskip2cm 0<\beta\leq 2 ;
\\ \E \left| \bar f\left( \xi_{(i)},\xi_{(j)}\right) - \bar f\bigl( \tfrac{i}{n+1} , \tfrac{j}{n+1} \bigr) \right| & \leq M \left\{ 2(n+2) \right\}^{ - \alpha / 2 } + 2 M \textstyle \left( \sqrt{2} \max_{1 \leq a \leq k } h_a / n \right)^\alpha \!\! , \,\,
\end{alignat*}
where $ {\bar f}\left(x,y; h\right) = {\bar f}_{ H^{-1}(x) H^{-1}(y) } $ is the stepfunction approximation of $f$. Furthermore, we have for $ 1 \leq i, j \leq n $ that
\begin{equation*}
\left| {f}\bigl( \tfrac{i}{n+1} , \tfrac{j}{n+1} \bigr) -{f}\left(x,y\right)\right| \le M \, \bigl\{ \sqrt{2} / ( n + 1 ) \bigr\}^\alpha, \qquad (x,y)\in \left( \tfrac{i-1}{n},\tfrac{i}{n} \right) \times \left( \tfrac{j-1}{n},\tfrac{j}{n} \right).
\end{equation*} 
\end{lemma}

\begin{proof}
Let $ i_n = \E\xi_{(i)} = i / (n+1) $. Since $f \in \operatorname{\textrm{H\"older}}^\alpha(M)$, it holds everywhere on $(0,1)^2$ that
\begin{equation*}
\label{eq:lips-ij}
\left| f\left( \xi_{(i)},\xi_{(j)}\right)-f\left( i_n,j_n\right)\right|^{\beta}\leq \left\{M\left| (
\xi_{(i)} , \xi_{(j)})-(
i_n, j_n)\right|^{\alpha}\right\}^{\beta}, \quad 1 \leq i, j \leq n ,
\end{equation*}
where $\left|\cdot\right|$ is the Euclidean metric on ${\mathbb{R}}^2$. By Jensen's inequality, we have for any $0 < \alpha \beta \leq 2 $ that for $ 1 \leq i, j \leq n $,
\begin{align*}
\E \left\{ ( \xi_{(i)} - i_n )^2 + ( \xi_{(j)} - j_n )^2 \right\}^{ \alpha\beta / 2 }
& \le \left( \var \xi_{(i)}+\var \xi_{(j)} \right)^{ \alpha \beta/ 2 }
\leq \left\{ 2(n+2) \right\}^{ - \alpha \beta/ 2 },
\end{align*}
with the latter inequality via $ \var \xi_{(i)} = i_n ( 1 - i_n ) / (n+2) \leq (1/4) / (n+2)$. This proves the first result. For the second, we use Lemma~\ref{normbound} and a chaining argument, since $\bar f $ is piecewise-constant on blocks:
\begin{align*}
\left| {\bar f}\left(\xi_{(i)}, \xi_{(j)}\right) - {\bar f}\left(i_n,j_n \right) \right|
& \leq \left| \left( \bar f - f \right) \left(\xi_{(i)}, \xi_{(j)}\right) \right|
+ \left| f\left(\xi_{(i)},\xi_{(j)}\right) - f\left(i_n,j_n \right) \right|
\\ & \quad + \left| \left( f - \bar f \right) \left(i_n,j_n \right) \right|
\\ & \leq \left| f\left(\xi_{(i)},\xi_{(j)}\right) - f\left(i_n,j_n \right) \right| + \textstyle 2 M \left( \sqrt{2} \max_{1 \leq a \leq k } h_a / n \right)^\alpha .
\end{align*}
Finally, $f \in \operatorname{\textrm{H\"older}}^\alpha(M)$ implies for $ (x,y)\in \left( \frac{i-1}{n},\frac{i}{n} \right) \times \left( \frac{j-1}{n},\frac{j}{n} \right) $ the uniform upper bound for $1 \leq i , j \leq n$:
\begin{align*}
\left| {f}\left(i_n,j_n \right)-{f}\left(x,y\right)\right|&\le M \sup_{ (x,y)\in \left( \frac{i-1}{n},\frac{i}{n} \right) \times \left( \frac{j-1}{n},\frac{j}{n} \right) }
\bigl\{ \left(i_n-x\right)^2+\left(j_n-y\right)^2\bigr\}^{\alpha/2}
\\ & \le M \left[\max_{ 1 \le i \le n} \left\{ 2 \max \left( \frac{ (i_n)^2 }{ n^2 } , \frac{ (1-i_n)^2 }{ n^2 } \right) \right\} \right]^{\alpha/2} .
\end{align*}
\begin{equation*}
\vspace{-2\baselineskip}
\end{equation*}
\end{proof}

\begin{lemma}\label{divergenceMC2}
Let $ f $ be a symmetric $\operatorname{\textrm{H\"older}}^\alpha(M)$ function on $(0,1)^2$, with stepfunction approximation $ {\bar f}\left(x,y; h\right) = {\bar f}_{ H^{-1}(x) H^{-1}(y) } $, and let $\{\xi_{(i)}\}_{i=1}^n$ be an ordered sample of independent $\operatorname{Uniform}(0,1)$ random variables. Then whenever $\rho_n>0$ and $ 0 < \rho_n f\left(x,y\right) < 1 $ everywhere on $(0,1)^2$, 
\begin{multline}
\label{divergenceMC}
\left\{ \rho_n \tbinom{n}{2} \right\}^{-1} \E_{\xi} \sum_{i<j} \D\left\{\rho_n f\left(\xi_{i}, \xi_{j}\right) \,\middle\vert\middle\vert\, \rho_n {\bar f}\left(\xi_{i}, \xi_{j}\right) \right\} \\ \leq 
\frac{ \rho_n M^2 \textstyle \left( \sqrt{2} \max_{1 \leq a \leq k } h_a / n \right)^{2\alpha} }{ \min_{1\le a,b\le k}\left\{\min\left( \rho_n {\bar f}_{ab},1- \rho_n {\bar f}_{ab}\right)\right\} }
\cdot \max_{ 1 \leq a, b \leq k } \left[ 1 + \Delta_{ab} \left\{ 1 + \frac{2}{3} \frac{ 1 + 2 \Delta_{ab} }{ \left( 1 - \Delta_{ab} \right)^3 } \right\} \right] ,
\end{multline}
where for $f\vert_{\omega_{ab}}$ the restriction of $f$ to $\omega_{ab} = \left[ H(a-1) , H(a) \right) \times \left[ H(b-1) , H(b) \right) $, we define
\begin{equation}
\label{eq:delta-bnd} 
\Delta_{ab} = \frac{ \rho_n \left\| f\vert_{\omega_{ab}} - {\bar f}_{ab} \right\|_{ L_\infty \left( \omega_{ab} \right) } }{ \min \left( \rho_n {\bar f}_{ab} , 1 - \rho_n {\bar f}_{ab} \right) } 
\leq \frac{ \rho_n M \left( \sqrt{2} \max_{1 \leq a \leq k } h_a / n \right)^\alpha }{ \min \left( \rho_n {\bar f}_{ab} , 1 - \rho_n {\bar f}_{ab} \right) } , \quad 1 \leq a , b \leq k .
\end{equation}
\end{lemma}

\begin{proof}
Since $\{ \xi_i \}_{i=1}^n$ is a random sample of $\operatorname{Uniform}(0,1)$ variates, and $f$ is symmetric, we have
\begin{multline}
\label{eq:KLscaledgraphon}
\left\{ \rho_n \tbinom{n}{2} \right\}^{-1} \E_{\xi} \sum_{i<j} \D\left\{\rho_n f\left(\xi_{i}, \xi_{j}\right) \,\middle\vert\middle\vert\, \rho_n {\bar f}\left(\xi_{i}, \xi_{j}\right) \right\} 
\\ = \iint_{ (0,1)^2 } \hskip-0.6cm \rho_n^{-1} \D\left\{\rho_n f(x, y) \,\middle\vert\middle\vert\, \rho_n {\bar f}\left(x, y\right) \right\}\,dx\,dy . \!
\end{multline}
Let $p=\rho_n \bar f$ and $\delta= \rho_n (f-\bar f)$ pointwise on $(0,1)^2$, in order to apply Lemma~\ref{KLDiv} to the integrand of~\eqref{eq:KLscaledgraphon}, and define the following ratio: $ \Delta_{ab} = \rho_n \left\| f\vert_{\omega_{ab}} - {\bar f}_{ab} \right\|_{ L_\infty \left( \omega_{ab} \right) } / \min \left(\rho_n {\bar f}_{ab} , 1 - \rho_n{\bar f}_{ab} \right)$. 
We may then write
\begin{align*}
& \iint_{ (0,1)^2 } \rho_n^{-1} \D\left\{\rho_n f(x, y) \,\middle\vert\middle\vert\, \rho_n {\bar f}\left(x, y\right) \right\}\,dx\,dy 
\\  & = \sum_{a=1}^k \sum_{b=1}^k \iint_{\omega_{ab}}
\rho_n^{-1} \D\left\{\rho_n f(x, y) \,\middle\vert\middle\vert\, \rho_n \bar f_{ab} \right\}\,dx\,dy 
\\ & \leq \sum_{a=1}^k \sum_{b=1}^k \iint_{\omega_{ab}} \rho_n^{-1}
\frac{ \left| \rho_n f(x, y) - \rho_n \bar f_{ab} \right|^2 }{ 2 \rho_n \bar f_{ab} \left( 1 - \rho_n \bar f_{ab} \right) } 
\left[ 1 + \Delta_{ab} \left\{ 1 + \frac{2}{3} \frac{ 1 + 2 \Delta_{ab} }{ \left( 1 - \Delta_{ab} \right)^3 } \right\} \right]
\,dx\,dy 
\\ &\leq \max_{ 1 \leq a, b \leq k } \left[ \frac{ 1 + \Delta_{ab} \left\{ 1 + \frac{2}{3} \frac{ 1 + 2 \Delta_{ab} }{ \left( 1 - \Delta_{ab} \right)^3 } \right\} }{ 2 \rho_n \bar f_{ab} \left( 1 - \rho_n \bar f_{ab} \right) } \right] \rho_n \| f - \bar f\|^2_{ L_2 \left( ( 0 , 1 )^2 \right) } .
\end{align*}
Our final step is to control the norms $ \left\| f\vert_{\omega_{ab}} - {\bar f}_{ab} \right\|_{ L_\infty \left( \omega_{ab} \right) } $ and $ \| f - \bar f\|^2_{ L_2 \left( ( 0 , 1 )^2 \right) } $ in this bound. To do so, we apply Lemma~\ref{normbound}, which asserts that whenever $ f \in \operatorname{\textrm{H\"older}}^\alpha(M)$, we have for all $1 \leq a , b \leq k$ that
\begin{equation}
\label{eq:normbound}
\left\| f\vert_{\omega_{ab}} - {\bar f}_{ab} \right\|_{ L_\infty \left( \omega_{ab} \right) } \leq \| f - \bar f\|_{ L_2 \left( ( 0 , 1 )^2 \right) } \leq M \textstyle \left( \sqrt{2} \max_{1 \leq a \leq k } h_a / n \right)^\alpha .
\end{equation}
The result follows from~\eqref{eq:normbound}, since by hypothesis $ \max \left( \rho_n {\bar f}_{ab} , 1 - \rho_n {\bar f}_{ab} \right) \geq 1/2 $ for every $(a , b )$, and so
\begin{equation*}
\frac{ \rho_n \| f - \bar f\|^2_{ L_2 \left( ( 0 , 1 )^2 \right) } }{ 2 \rho_n \bar f_{ab} \left( 1 - \rho_n \bar f_{ab} \right) } 
\leq \frac{ \rho_n \| f - \bar f\|^2_{ L_2 \left( ( 0 , 1 )^2 \right) } }{ \min \left( \rho_n {\bar f}_{ab} , 1 - \rho_n {\bar f}_{ab} \right) }
\leq \frac{ \rho_n M^2 \textstyle \left( \sqrt{2} \max_{1 \leq a \leq k} h_a / n \right)^{2\alpha} }{ \min \left( \rho_n {\bar f}_{ab} , 1 - \rho_n {\bar f}_{ab} \right) } . %
\vspace{-\baselineskip}%
\end{equation*}
\end{proof}

\begin{lemma}\label{KLDiv}
Consider the Bernoulli Kullback--Leibler divergence quantities $\D\left(p \,\middle\vert\middle\vert\, p+\delta \right)$ and $\D\left(p+\delta \,\middle\vert\middle\vert\, p \right)$,
where $0 < p < 1$ and $ -p \le \delta \le 1 - p $. If $ \left|\delta\right| < \min\left( p, 1 - p \right) $, then 
the following bounds hold:
\begin{align*}
\tfrac{\left|\D\left(p \,\middle\vert\middle\vert\, p +\delta \right)-\tfrac{\delta^2}{2p(1-p)}\right|}{\delta^2/\left\{2p(1-p)\right\}}&\le 
\tfrac{2}{3}\tfrac{|\delta|}{\min_{}\left(p,1-p\right)}\left(1-\tfrac{|\delta|}{\min_{}\left(p,1-p\right)}\right)^{-3},
\\
\tfrac{\left|\D\left(p+\delta \,\middle\vert\middle\vert\, p \right)-\tfrac{\delta^2}{2p(1-p)}\right|}{\delta^2/\left\{2p(1-p)\right\}}&\le \tfrac{|\delta|}{\min_{}\left(p,1-p\right)}\left\{ 1+
\tfrac{2}{3}
\left(1+
\tfrac{2|\delta|}{\min_{}\left(p,1-p\right)}\right)\left(1-\tfrac{|\delta|}{\min_{}\left(p,1-p\right)}\right)^{-3}\right\}.
\end{align*}
Now consider $\rho_n, f, g > 0$ such that $ 0 < \rho_n f , \rho_n g < 1 $. Then $ \left| f - g \right|^2 \leq 2 f \rho_n^{-1} \D\left( \rho_n f \,\middle\vert\middle\vert\, \rho_n g \right) $.
\end{lemma}

\begin{proof}
The first result follows by manipulating a Taylor series expansion of $\D\left(p \,\middle\vert\middle\vert\, p+\delta \right)$ using the Lagrange form of the remainder.
For some $\delta',\delta''$ satisfying $0<\left|\delta'\right|<\left|\delta\right|$ and $0<\left|\delta''\right|<\left|\delta\right|$, we have
\begin{equation}
\label{eq:taylorKLr}
\D\left(p \,\middle\vert\middle\vert\, p+\delta \right)=\tfrac{\delta^2}{2p(1-p)}\left[1+\tfrac{2}{3} \tfrac{\delta}{\min_{}\left(p,1-p\right)} 
\left\{\tfrac{p^2\bigl(1-\tfrac{\delta''}{1-p} \bigr)^{-3}-
(1-p)^2\bigl(1+\tfrac{\delta'}{p} \bigr)^{-3}}{\max_{}\left(p,1-p\right)} \right\}\right]\!.
\end{equation}
The first result then follows by controlling the scaled difference of the remainder terms appearing in~\eqref{eq:taylorKLr}, both of which are non-negative. We upper-bound this difference by the maximum of these two quantities, writing \begin{multline*}
\max\bigl\{p^2\bigl(1-\frac{\delta''}{1-p} \bigr)^{-3} , (1-p)^2\bigl(1+\frac{\delta'}{p} \bigr)^{-3}\bigr\} \\ \leq 
\left\{\max_{}\left(p,1-p\right)\right\}^2 \left\{ 1 - \left|\delta\right| / \min_{}\left(p,1-p\right) \right\}^{-3}.
\end{multline*}
The second result follows similarly, by manipulating a Taylor series expansion of $\D\left(p+\delta \,\middle\vert\middle\vert\, p \right)$. 

The final result follows from rewriting $ \D\left( \rho_n f \,\middle\vert\middle\vert\, \rho_n g \right) $ as $ \D\left( \rho_n (g+d) \,\middle\vert\middle\vert\, \rho_n g \right) $, with $d = f-g$.  We first bound the second derivative of $ \D\left( \rho_n (g+d) \,\middle\vert\middle\vert\, \rho_n g \right) $ in $d$ below by $ \rho_n / f $, and then integrate twice, using that $ \D\left( \rho_n (g+d) \,\middle\vert\middle\vert\, \rho_n g \right) = 0 $ if $ d = 0 $.
\end{proof}

\begin{lemma}\label{partition}
Let $i_n=i/(n+1)$ and $j_n=j/(n+1)$. Then $(i_n,j_n)\in \omega_{a_i b_j}$, where $a_i$ and $b_j$ are defined by
\begin{equation*}
a_i = H^{-1}\left( i/n\right), \quad b_j = H^{-1}\left( j/n\right), \quad 1 \leq a,b \leq k, \,\, 1 \leq i,j \leq n .
\end{equation*}
\end{lemma}
\begin{proof}
From the definition of $a_i$ we may directly compute
\begin{align*}
H\left\{ a_i \right\}&=H\left\{ H^{-1}\left( i/n\right) \right\} = n^{-1} \!\!\!\! \!\!\!\!\!\!\!\! \sum_{a=1}^{\min\left\{ H^{-1}\left( i/n\right) ,k\right\}} \!\!\!\! h_{a}
\,\,\begin{cases}
=i/n & {\mathrm{if}} \, \sum_{a=1}^{a_i} h_a=i,\\
\ge (i+1)/n & {\mathrm{if}} \, \sum_{a=1}^{a_i} h_a\neq i.
\end{cases}
\end{align*}
We also have that
\begin{align*}
H\left( a_i -1\right) & = H\left\{ H^{-1}\left( i/n\right) -1\right\} 
\\ & = n^{-1} \!\!\!\! \!\!\!\!\!\!\!\! \sum_{a=1}^{\min\left\{ H^{-1}\left( i/n\right)-1,k\right\}} \!\!\!\! h_{a}
\,\,\begin{cases}
=(i-1)/n & {\mathrm{if}} \, \sum_{a=1}^{a_i-1} h_{a}=i-1,\\
\le (i-2)/n & {\mathrm{if}} \, \sum_{a=1}^{a_i-1} h_a\neq i-1 .
\end{cases}
\end{align*}
We have by definition that
$\omega_{a_i b_j}=\left[H\left\{ H^{-1}\left( i/n\right)-1
\right\},H\left\{ H^{-1}\left( i/n\right) \right\}\right)\times
\left[H\left\{ H^{-1}\left( j/n\right)-1
\right\},H\left\{ H^{-1}\left( j/n\right) \right\}\right)$. Since $H(\cdot)$ and its inverse $H^{-1}(\cdot)$ are non-decreasing functions, it follows that $ H\left\{H^{-1}\left( i/n\right)\right\} \geq i/n \geq i/(n+1) = i_n $. Thus the claimed upper bound is respected. Furthermore, for the lower limit, $ H\left\{ H^{-1}\left( i/n\right) -1\right\} \leq (i-1) / n \leq i_n$, as $ (i-1)/n \leq i/(n+1) = i_n \Leftrightarrow i \leq n+1$. Thus the claimed lower bound is also respected, and so by symmetry, we conclude that $(i_n,j_n)\in \omega_{a_i b_j}$.
\end{proof}

\section*{Acknowledgements}
We thank David Choi for helpful insight into blockmodels. Work supported in part by the US Army Research Office under PECASE Award W911NF-09-1-0555 and MURI Award 58153-MA-MUR; by the UK EPSRC under Mathematical Sciences Leadership Fellowship EP/I005250/1, Established Career Fellowship EP/K005413/1 and Developing Leaders Award EP/L001519/1; by the UK Royal Society under a Wolfson Research Merit Award; and by Marie Curie FP7 Integration Grant PCIG12-GA-2012-334622 within the 7th European Union Framework Program.

\bibliographystyle{imsart-nameyear}
\bibliography{nonparametric-graphon-estimation}

\providecommand*\hyphen{-}
\begin{thebibliography}{32}

\bibitem[\protect\citeauthoryear{Airoldi et~al.}{2008}]{airoldi2008mixed}
\begin{barticle}[author]
\bauthor{\bsnm{Airoldi},~\bfnm{E.~M.}\binits{E.~M.}},
  \bauthor{\bsnm{Blei},~\bfnm{D.~M.}\binits{D.~M.}},
  \bauthor{\bsnm{Fienberg},~\bfnm{S.~E.}\binits{S.~E.}} \AND
  \bauthor{\bsnm{Xing},~\bfnm{E.~P.}\binits{E.~P.}}
(\byear{2008}).
\btitle{Mixed membership stochastic blockmodels}.
\bjournal{J. Mach. Learn. Res.}
\bvolume{9}
\bpages{1981--2014}.
\end{barticle}
\endbibitem

\bibitem[\protect\citeauthoryear{Aldous}{1981}]{aldous1981representations}
\begin{barticle}[author]
\bauthor{\bsnm{Aldous},~\bfnm{D.~J.}\binits{D.~J.}}
(\byear{1981}).
\btitle{Representations for partially exchangeable arrays of random variables}.
\bjournal{J. Multivariate Anal.}
\bvolume{11}
\bpages{581--598}.
\end{barticle}
\endbibitem

\bibitem[\protect\citeauthoryear{Alon}{1995}]{alon1995note}
\begin{bincollection}[author]
\bauthor{\bsnm{Alon},~\bfnm{N.}\binits{N.}}
(\byear{1995}).
\btitle{A note on network reliability}.
In \bbooktitle{Discrete Probability and Algorithms}
(\beditor{\bfnm{D.}\binits{D.}~\bsnm{Aldous}},
  \beditor{\bfnm{P.}\binits{P.}~\bsnm{Diaconis}},
  \beditor{\bfnm{J.}\binits{J.}~\bsnm{Spencer}} \AND
  \beditor{\bfnm{J.~M.}\binits{J.~M.}~\bsnm{Steele}}, eds.)
\bpages{11--14}.
\bpublisher{Springer-Verlag}, \baddress{New York}.
\end{bincollection}
\endbibitem

\bibitem[\protect\citeauthoryear{Arias-Castro and
  Grimmett}{2013}]{arias2013cluster}
\begin{barticle}[author]
\bauthor{\bsnm{Arias-Castro},~\bfnm{E.}\binits{E.}} \AND
  \bauthor{\bsnm{Grimmett},~\bfnm{G.~R.}\binits{G.~R.}}
(\byear{2013}).
\btitle{Cluster detection in networks using percolation}.
\bjournal{Bernoulli}
\bvolume{19}
\bpages{676--719}.
\end{barticle}
\endbibitem

\bibitem[\protect\citeauthoryear{Ball, Britton and
  Sirl}{2013}]{ball2013network}
\begin{barticle}[author]
\bauthor{\bsnm{Ball},~\bfnm{F.}\binits{F.}},
  \bauthor{\bsnm{Britton},~\bfnm{T.}\binits{T.}} \AND
  \bauthor{\bsnm{Sirl},~\bfnm{D.}\binits{D.}}
(\byear{2013}).
\btitle{A network with tunable clustering, degree correlation and degree
  distribution, and an epidemic thereon}.
\bjournal{J. Math. Biol.}
\bvolume{66}
\bpages{979--1019}.
\end{barticle}
\endbibitem

\bibitem[\protect\citeauthoryear{Bickel and
  Chen}{2009}]{bickel2009nonparametric}
\begin{barticle}[author]
\bauthor{\bsnm{Bickel},~\bfnm{P.~J.}\binits{P.~J.}} \AND
  \bauthor{\bsnm{Chen},~\bfnm{A.}\binits{A.}}
(\byear{2009}).
\btitle{A nonparametric view of network models and {N}ewman--{G}irvan and other
  modularities}.
\bjournal{Proc. Natl. Acad. Sci. USA}
\bvolume{106}
\bpages{21068--21073}.
\end{barticle}
\endbibitem

\bibitem[\protect\citeauthoryear{Bickel, Chen and Levina}{2011}]{BickelLevina}
\begin{barticle}[author]
\bauthor{\bsnm{Bickel},~\bfnm{P.~J.}\binits{P.~J.}},
  \bauthor{\bsnm{Chen},~\bfnm{A.}\binits{A.}} \AND
  \bauthor{\bsnm{Levina},~\bfnm{E.}\binits{E.}}
(\byear{2011}).
\btitle{The method of moments and degree distributions for network models}.
\bjournal{Ann. Statist.}
\bvolume{39}
\bpages{2280--2301}.
\end{barticle}
\endbibitem

\bibitem[\protect\citeauthoryear{Birg{\'e} and
  Massart}{1998}]{birge1998minimum}
\begin{barticle}[author]
\bauthor{\bsnm{Birg{\'e}},~\bfnm{L.}\binits{L.}} \AND
  \bauthor{\bsnm{Massart},~\bfnm{P.}\binits{P.}}
(\byear{1998}).
\btitle{Minimum contrast estimators on sieves: Exponential bounds and rates of
  convergence}.
\bjournal{Bernoulli}
\bvolume{4}
\bpages{329--375}.
\end{barticle}
\endbibitem

\bibitem[\protect\citeauthoryear{Bollob{\'a}s, Janson and
  Riordan}{2007}]{bollobas2007phase}
\begin{barticle}[author]
\bauthor{\bsnm{Bollob{\'a}s},~\bfnm{B.}\binits{B.}},
  \bauthor{\bsnm{Janson},~\bfnm{S.}\binits{S.}} \AND
  \bauthor{\bsnm{Riordan},~\bfnm{O.}\binits{O.}}
(\byear{2007}).
\btitle{The phase transition in inhomogeneous random graphs}.
\bjournal{Random Structures Algorithms}
\bvolume{31}
\bpages{3--122}.
\end{barticle}
\endbibitem

\bibitem[\protect\citeauthoryear{Bollob{\'a}s and
  Riordan}{2009}]{bollobas2009metric}
\begin{bincollection}[author]
\bauthor{\bsnm{Bollob{\'a}s},~\bfnm{B.}\binits{B.}} \AND
  \bauthor{\bsnm{Riordan},~\bfnm{O.}\binits{O.}}
(\byear{2009}).
\btitle{Metrics for sparse graphs}.
In \bbooktitle{Surveys in Combinatorics 2009}
(\beditor{\bfnm{S.}\binits{S.}~\bsnm{Huczynska}},
  \beditor{\bfnm{J.~D.}\binits{J.~D.}~\bsnm{Mitchell}} \AND
  \beditor{\bfnm{C.~M.}\binits{C.~M.}~\bsnm{Roney-Dougal}}, eds.)
\bpages{211--287}.
\bpublisher{Cambridge University Press}, \baddress{Cambridge, UK}.
\end{bincollection}
\endbibitem

\bibitem[\protect\citeauthoryear{Chatterjee}{2012}]{chatterjee2012matrix}
\begin{bmisc}[author]
\bauthor{\bsnm{Chatterjee},~\bfnm{S.}\binits{S.}}
(\byear{2012}).
\btitle{Matrix estimation by universal singular value thresholding}.
\bhowpublished{Preprint arXiv:1212.1247}.
\end{bmisc}
\endbibitem

\bibitem[\protect\citeauthoryear{Chatterjee, Diaconis and
  Sly}{2011}]{chatterjee2011random}
\begin{barticle}[author]
\bauthor{\bsnm{Chatterjee},~\bfnm{S.}\binits{S.}},
  \bauthor{\bsnm{Diaconis},~\bfnm{P.}\binits{P.}} \AND
  \bauthor{\bsnm{Sly},~\bfnm{A.}\binits{A.}}
(\byear{2011}).
\btitle{Random graphs with a given degree sequence}.
\bjournal{Ann. Appl. Probab..}
\bvolume{21}
\bpages{1400--1435}.
\end{barticle}
\endbibitem

\bibitem[\protect\citeauthoryear{Choi, Wolfe and
  Airoldi}{2012}]{choi2012stochastic}
\begin{barticle}[author]
\bauthor{\bsnm{Choi},~\bfnm{D.~S.}\binits{D.~S.}},
  \bauthor{\bsnm{Wolfe},~\bfnm{P.~J.}\binits{P.~J.}} \AND
  \bauthor{\bsnm{Airoldi},~\bfnm{E.~M.}\binits{E.~M.}}
(\byear{2012}).
\btitle{Stochastic blockmodels with a growing number of classes}.
\bjournal{Biometrika}
\bvolume{99}
\bpages{273--284}.
\end{barticle}
\endbibitem

\bibitem[\protect\citeauthoryear{Choi and Wolfe}{2013}]{choi2012co}
\begin{barticle}[author]
\bauthor{\bsnm{Choi},~\bfnm{D.~S.}\binits{D.~S.}} \AND
  \bauthor{\bsnm{Wolfe},~\bfnm{P.~J.}\binits{P.~J.}}
(\byear{2013}).
\btitle{Co-clustering separately exchangeable network data}.
\bjournal{Ann. Statist.}
\bnote{To appear (arXiv:1212.4093)}.
\end{barticle}
\endbibitem

\bibitem[\protect\citeauthoryear{DeVore}{1998}]{devore1998nonlinear}
\begin{barticle}[author]
\bauthor{\bsnm{DeVore},~\bfnm{Ronald~A}\binits{R.~A.}}
(\byear{1998}).
\btitle{Nonlinear approximation}.
\bjournal{Acta numerica}
\bvolume{7}
\bpages{51--150}.
\end{barticle}
\endbibitem

\bibitem[\protect\citeauthoryear{Diaconis}{1977}]{diaconis1977finite}
\begin{barticle}[author]
\bauthor{\bsnm{Diaconis},~\bfnm{P.}\binits{P.}}
(\byear{1977}).
\btitle{Finite forms of de~{F}inetti's theorem on exchangeability}.
\bjournal{Synthese}
\bvolume{36}
\bpages{271--281}.
\end{barticle}
\endbibitem

\bibitem[\protect\citeauthoryear{Diaconis and Janson}{2008}]{diaconis2007graph}
\begin{barticle}[author]
\bauthor{\bsnm{Diaconis},~\bfnm{P.}\binits{P.}} \AND
  \bauthor{\bsnm{Janson},~\bfnm{S.}\binits{S.}}
(\byear{2008}).
\btitle{Graph limits and exchangeable random graphs}.
\bjournal{Rend. Mat. Appl.}
\bvolume{28}
\bpages{33--61}.
\end{barticle}
\endbibitem

\bibitem[\protect\citeauthoryear{Durrett}{2007}]{Durrett}
\begin{bbook}[author]
\bauthor{\bsnm{Durrett},~\bfnm{R.}\binits{R.}}
(\byear{2007}).
\btitle{Random Graph Dynamics}.
\bpublisher{Cambridge University Press}, \baddress{Cambridge, UK}.
\end{bbook}
\endbibitem

\bibitem[\protect\citeauthoryear{Fienberg}{2012}]{fienberg2012brief}
\begin{barticle}[author]
\bauthor{\bsnm{Fienberg},~\bfnm{Stephen~E}\binits{S.~E.}}
(\byear{2012}).
\btitle{A brief history of statistical models for network analysis and open
  challenges}.
\bjournal{J. Comput. Graph. Statist.}
\bvolume{21}
\bpages{825--839}.
\end{barticle}
\endbibitem

\bibitem[\protect\citeauthoryear{Fienberg and
  Rinaldo}{2012}]{fienberg2012maximum}
\begin{barticle}[author]
\bauthor{\bsnm{Fienberg},~\bfnm{S.~E.}\binits{S.~E.}} \AND
  \bauthor{\bsnm{Rinaldo},~\bfnm{A.}\binits{A.}}
(\byear{2012}).
\btitle{Maximum likelihood estimation in log-linear models}.
\bjournal{Ann. Statist.}
\bvolume{40}
\bpages{996--1023}.
\end{barticle}
\endbibitem

\bibitem[\protect\citeauthoryear{Fishkind
  et~al.}{2013}]{fishkind2013consistent}
\begin{barticle}[author]
\bauthor{\bsnm{Fishkind},~\bfnm{D.~E.}\binits{D.~E.}},
  \bauthor{\bsnm{Sussman},~\bfnm{D.~L.}\binits{D.~L.}},
  \bauthor{\bsnm{Tang},~\bfnm{M.}\binits{M.}},
  \bauthor{\bsnm{Vogelstein},~\bfnm{J.~T.}\binits{J.~T.}} \AND
  \bauthor{\bsnm{Priebe},~\bfnm{C.~E.}\binits{C.~E.}}
(\byear{2013}).
\btitle{Consistent adjacency-spectral partitioning for the stochastic block
  model when the model parameters are unknown}.
\bjournal{SIAM J. Matrix Anal. Appl.}
\bvolume{34}
\bpages{23--39}.
\end{barticle}
\endbibitem

\bibitem[\protect\citeauthoryear{Green and
  Silverman}{1994}]{green1994nonparametric}
\begin{bbook}[author]
\bauthor{\bsnm{Green},~\bfnm{P.~J.}\binits{P.~J.}} \AND
  \bauthor{\bsnm{Silverman},~\bfnm{B.~W.}\binits{B.~W.}}
(\byear{1994}).
\btitle{Nonparametric Regression and Generalized Linear Models: A Roughness
  Penalty Approach}.
\bpublisher{Chapman \& Hall}, \baddress{London}.
\end{bbook}
\endbibitem

\bibitem[\protect\citeauthoryear{Hoover}{1979}]{hoover1979relations}
\begin{bunpublished}[author]
\bauthor{\bsnm{Hoover},~\bfnm{D.~N.}\binits{D.~N.}}
(\byear{1979}).
\btitle{Relations on probability spaces and arrays of random variables}.
\bnote{Princeton, NJ}.
\end{bunpublished}
\endbibitem

\bibitem[\protect\citeauthoryear{Janson}{2010}]{janson2010asymptotic}
\begin{barticle}[author]
\bauthor{\bsnm{Janson},~\bfnm{S.}\binits{S.}}
(\byear{2010}).
\btitle{Asymptotic equivalence and contiguity of some random graphs}.
\bjournal{Random Structures Algorithms}
\bvolume{36}
\bpages{26--45}.
\end{barticle}
\endbibitem

\bibitem[\protect\citeauthoryear{Lov{\'a}sz}{2012}]{lovasz2012large}
\begin{bbook}[author]
\bauthor{\bsnm{Lov{\'a}sz},~\bfnm{L.}\binits{L.}}
(\byear{2012}).
\btitle{Large Networks and Graph Limits}.
\bpublisher{American Mathematical Society}, \baddress{Providence, RI}.
\end{bbook}
\endbibitem

\bibitem[\protect\citeauthoryear{Olhede and Wolfe}{2013}]{olhede2013degree}
\begin{bmisc}[author]
\bauthor{\bsnm{Olhede},~\bfnm{Sofia~C}\binits{S.~C.}} \AND
  \bauthor{\bsnm{Wolfe},~\bfnm{Patrick~J}\binits{P.~J.}}
(\byear{2013}).
\btitle{Degree-based network models}.
\end{bmisc}
\endbibitem

\bibitem[\protect\citeauthoryear{Rinaldo, Petrovi\'c and
  Fienberg}{2013}]{rinaldo2013maximum}
\begin{barticle}[author]
\bauthor{\bsnm{Rinaldo},~\bfnm{A.}\binits{A.}},
  \bauthor{\bsnm{Petrovi\'c},~\bfnm{S.}\binits{S.}} \AND
  \bauthor{\bsnm{Fienberg},~\bfnm{S.~E.}\binits{S.~E.}}
(\byear{2013}).
\btitle{Maximum likelihood estimation in the {B}eta model}.
\bjournal{Ann. Statist.}
\bvolume{41}
\bpages{1085--1110}.
\end{barticle}
\endbibitem

\bibitem[\protect\citeauthoryear{Rohe, Chatterjee and
  Yu}{2011}]{rohe2011spectral}
\begin{barticle}[author]
\bauthor{\bsnm{Rohe},~\bfnm{K.}\binits{K.}},
  \bauthor{\bsnm{Chatterjee},~\bfnm{S.}\binits{S.}} \AND
  \bauthor{\bsnm{Yu},~\bfnm{B.}\binits{B.}}
(\byear{2011}).
\btitle{Spectral clustering and the high-dimensional stochastic blockmodel}.
\bjournal{Ann. Statist.}
\bvolume{39}
\bpages{1878--1915}.
\end{barticle}
\endbibitem

\bibitem[\protect\citeauthoryear{Romanovsky}{1923}]{romanovsky1923note}
\begin{barticle}[author]
\bauthor{\bsnm{Romanovsky},~\bfnm{V.}\binits{V.}}
(\byear{1923}).
\btitle{Note on the moments of a {B}inomial $(p+q)^n$ about its mean}.
\bjournal{Biometrika}
\bvolume{15}
\bpages{410--412}.
\end{barticle}
\endbibitem

\bibitem[\protect\citeauthoryear{Shaked and
  Shanthikumar}{1994}]{shaked1994stochastic}
\begin{bbook}[author]
\bauthor{\bsnm{Shaked},~\bfnm{M.}\binits{M.}} \AND
  \bauthor{\bsnm{Shanthikumar},~\bfnm{J.~G.}\binits{J.~G.}}
(\byear{1994}).
\btitle{Stochastic Orders and Their Applications}.
\bpublisher{Academic Press}, \baddress{Boston, MA}.
\end{bbook}
\endbibitem

\bibitem[\protect\citeauthoryear{Sussman, Tang and
  Priebe}{2013}]{sussman2013universally}
\begin{barticle}[author]
\bauthor{\bsnm{Sussman},~\bfnm{D.~L.}\binits{D.~L.}},
  \bauthor{\bsnm{Tang},~\bfnm{M.}\binits{M.}} \AND
  \bauthor{\bsnm{Priebe},~\bfnm{C.~E.}\binits{C.~E.}}
(\byear{2013}).
\btitle{Universally consistent latent position estimation and vertex
  classification for random dot product graphs}.
\bjournal{Ann. Statist.}
\bnote{In press (arXiv:1207.6745)}.
\end{barticle}
\endbibitem

\bibitem[\protect\citeauthoryear{Zhao, Levina and
  Zhu}{2012}]{zhao2012consistency}
\begin{barticle}[author]
\bauthor{\bsnm{Zhao},~\bfnm{Y.}\binits{Y.}},
  \bauthor{\bsnm{Levina},~\bfnm{E.}\binits{E.}} \AND
  \bauthor{\bsnm{Zhu},~\bfnm{J.}\binits{J.}}
(\byear{2012}).
\btitle{Consistency of community detection in networks under degree-corrected
  stochastic block models}.
\bjournal{Ann. Statist.}
\bvolume{40}
\bpages{2266--2292}.
\end{barticle}
\endbibitem

\end{thebibliography}

\end{document}